\RequirePackage{fix-cm}
\documentclass[11pt]{article}
\usepackage{authblk}
\usepackage[hidelinks]{hyperref}
\usepackage{mathtools}
\usepackage{amsthm}
\usepackage{amssymb}
\usepackage{thmtools}
\usepackage{biblatex}
\usepackage{tikz-cd}
\usepackage{dsfont}
\usepackage{comment}
\usepackage{enumitem}
\usepackage[margin=2.5cm]{geometry}

\usetikzlibrary{decorations.pathmorphing}
\declaretheorem[name=Theorem, numberwithin=section]{theorem}
\declaretheorem[name=Proposition, sibling=theorem]{proposition}
\declaretheorem[name=Lemma, sibling=theorem]{lemma}
\declaretheorem[name=Corollary, sibling=theorem]{corollary}
\theoremstyle{definition}
\declaretheorem[name=Example, numberwithin=section]{example}
\declaretheorem[name=Remark, sibling=example]{remark}
\newtheorem{claim}{Claim}

\newcommand{\up}{{\uparrow}}

\newcommand{\ca}[1]{\mathcal{#1}}
\newcommand{\bd}[1]{\mathbf{#1}}
\newcommand{\mf}{\mathsf}

\newcommand{\Om}{\Omega}
\newcommand{\Omi}{\Omega_{\ca I}}
\newcommand{\Omt}{\Omega_{\ca T}}
\newcommand{\pt}{\mathsf{pt}}

\newcommand{\ptt}{\pt_{\ca T}}
\newcommand{\Op}{\mathcal O}
\newcommand{\Or}{\mathcal{O}_{\mathcal{R}}}
\newcommand{\fse}{\mathsf{Filt}_{\mathcal{SE}}}
\newcommand{\fe}{\mathsf{Filt}_{\mathcal{E}}}
\newcommand{\fcp}{\mathsf{Filt}_{\ca{CP}}}
\newcommand{\Con}{\mathbf{Con}}
\newcommand{\ptc}{\pt_{\ca{C}}}
\newcommand{\Omc}{\Omega_{\ca{C}}}
\newcommand{\Topp}{\bd{Top}}
\newcommand{\Topz}{\bd{Top_0}}
\newcommand{\Topc}{\bd{Top_{\ca C}}}    
\newcommand{\Topt}{\bd{Top_{\ca T}}}
\newcommand{\Topi}{\bd{Top_{\ca I}}}
\newcommand{\Tc}{\bd{2}_{\ca C}}

\newcommand{\Frm}{\bd{Frm}}
\newcommand{\PFC}{\bd{PFC}}
\newcommand{\SFC}{\bd{SFC}}
\newcommand{\Set}{\bd{Set}}
\newcommand{\Ci}{\ca C_{\ca I}}
\newcommand{\Ct}{\ca C_{\ca T}}

\newcommand{\id}{\mathsf{id}}
\newcommand{\opp}[1]{#1^{\mathsf{op}}}
\newcommand{\bve}{\bigvee}  
\newcommand{\bwe}{\bigwedge}

\renewcommand{\epsilon}{\varepsilon}

\title{A general framework for the faithful pointfree representation of $T_0$-spaces}
\author{Rui Prezado\thanks{CMUC, Department of Mathematics, University of
Coimbra, 3000-143 Coimbra, Portugal}\hspace{2mm} and Anna Laura Suarez\thanks{School of
Mathematics, Hunan University, Changsha, Hunan 410082, China}}
\date{}

\begin{document}

\maketitle

\begin{abstract}
  We introduce a general framework for studying natural contravariant
  adjunctions that refine the adjunction between frames and spaces so that the
  fixpoints are $T_0$-spaces.

  Our objects of study are \textit{spatializable $\mathbf{Frm}$-concrete
  categories}, or \textit{SFC-categories}. These consist of a faithful functor
  $\Op:\ca C\to \Frm$ equipped with an object $\Tc\in \ca C$, satisfying
  compatibility conditions that ensure that $(\Tc,\mathbb{S})$ forms a
  dualizing object in the sense of Porst and Tholen, where $\mathbb{S}$
  denotes the Sierpiński space. 

  Three important instances of pointfree $T_0$ spaces present in the
  literature fit into this framework: strictly zero-dimensional biframes,
  MT-algebras, and Raney extensions. 

  We show SFC-categories are assembled in an ordered category -- a category
  enriched in preordered sets -- whose morphisms are suitable functors which
  preserve certain initial liftings. SFC-categories induce natural dual
  adjunctions, and morphisms between them will respectively induce suitable
  morphisms between these adjunctions.

  Motivated by the characterization of sober spaces as maximal objects in the
  fibers of $\Omega:\mathbf{Top}\to \mathbf{Frm}^{\mathsf{op}}$, and of
  $T_D$-spaces as the minimal ones, due to Banaschewski and Pultr, we study
  initial and terminal objects of fibers for an arbitrary SFC-category. We
  prove that the natural adjunction for fiber-initials has exactly the sober
  spaces as fixpoints, while for fiber-terminals contains at most
  $T_D$-spaces, recovering their results of in a much more general setting.
\end{abstract}

\tableofcontents
\section*{Introduction}

The classical adjunction between frames and topological spaces,
\[
\Om : \Topp^{\mf{op}}\rightleftarrows \Frm:\pt,
\]
restricts to a duality between the category of sober spaces and the category of spatial frames, but it does not capture all $T_0$ spaces. In particular, the unit of the adjunction for a $T_0$ space $X$ is a subspace embedding $\sigma_X:X \to \pt\Om X$, which is a homeomorphism if and only if $X$ is sober. 

The theories of Raney extensions (\cite{suarez25raney}), of strictly
zero-dimensional biframes (\cite{manuell15}), and of MT-algebras
(\cite{bezhanishvili23}), all provide pointfree faithful descriptions of the
category of $T_0$-spaces, by refining the natural adjunction $\Om\dashv \pt$.

The aim of this paper is to provide a general framework for the pointfree and
faithful treatment of $T_0$ spaces, by studying refinements of the classical
adjunction $\Om\dashv \pt$. To do so, we introduce a notion of
``\emph{spatializable \(\Frm\)-concrete category}'', or SFC-category for
short. This consists of a quadruple $(\Op,\ca C,\Tc,\iota_{\ca C})$ where
$\Op:\ca C\to \Frm$ is a faithful functor, $\Tc\in \ca C$, and $\iota_{\ca
C}:\Op \Tc\cong \bd 2$ an isomorphism, and such that compatibility conditions
are verified which ensure that the pair $(\Tc,\mathbb{S})$ forms a
\textit{dualizing object} in the sense of Porst and Tholen (\cite{porst91}).
The dualizing object induces a natural dual adjunction $\Omc\dashv \ptc$, as
shown below:

\begin{equation*}
    \begin{tikzcd}[row sep=huge,column sep=large]
        \ca{C}
            \ar[r,shift left=2,"\ptc"{above}] \ar[d,"\Op"]
        & \bd{Top}_{\ca C}^{\mathsf{op}}
            \ar[l,shift left=2,"\Omc"{below}]
            \ar[l,phantom,"\dashv"{rotate=-90,anchor=center}]
            \ar[d,hookrightarrow] \\
        \bd{Frm}
            \ar[r,shift left=2,"\pt"{above}]
        & \bd{Top}^{\mathsf{op}}
            \ar[l,shift left=2,"\Om"{below}]
            \ar[l,phantom,"\dashv"{rotate=-90,anchor=center}]
   \end{tikzcd}
\end{equation*}
There is always an isomorphism $\Op\Omc\cong \Om$, so that the diagram
commutes in the right adjoint direction. This has a \textit{mate} in the sense
of \cite{ks74}, which yields a natural subspace embedding $\ptc\to \pt\Op$,
filling in the square in the left adjoint direction.

We study morphisms between such SFC-categories, and call $\SFC$ the category
thus obtained, having $(\mathrm{1}_{\Frm},\Frm,\bd{2},\id_{\bd{2}})$ as its
terminal object. The theories of Raney extensions, of strictly zero-dimensional
biframes, and of MT-algebras are all examples of these SFC-categories.

We describe the notion of adjunction in the ordered category of \( \Frm
\)-concrete categories \cite{ahs90, por94}, and using this as a building block
for our setting, we confirm that SFC-functors will induce morphisms between
the induced natural adjunctions. We study the notion of adjunction for
PFC-categories, and show that right adjoints are morphisms in $\SFC$. 

We move on to study SFC-categories where $\Op$ has a left FC-adjoint, i. e.
those for which the free object in $\ca C$ over a frame exists.  In agreement
with the classical characterization of sober spaces, seen as objects of the
$\Frm$-concrete category $\Om:\Topz^{\mf{op}}\to \Frm$, as well as the
definitions of sober Raney extension and sober MT-algebras from
\cite{suarez25raney} and \cite{bezhanishvili23}, we define an object $C\in \ca
C$ to be sober if every point $p:\Op C\to \bd{2}$ of its underlying frame is
realized as $\Op \overline{p}$ for some point $\overline{p}:\ca C\to \Tc$ (i.
e. $C$ "has enough points"). We describe a pointfree notion of sobrification
for pointfree $T_0$-spaces, and find that when the category $\ca C$ has
pullbacks sober objects are coreflective in $\ca C$. 

Next, we consider the SFC-categories for which all $T_0$-spaces are fixpoints.
We recall the characterizations, due to Banaschewski and Pultr, of sober
spaces as the minimal elements and $T_D$ spaces as the maximal ones for the
$\Frm$-concrete category $\Om:\Topz^{\mf{op}}$, respectively. Inspired by this
perspective, we study the SFC-categories such that the fiber-initial and the
fiber-terminal objects induce their own natural adjunctions; we call these
\emph{total} SFC-categories. We let $I_{\ca I}:\Ci\to \ca C$ and $I_{\ca
T}:\Ct\to \ca C$ be the inclusions of the categories of fiber-initial objects
and that of fiber-terminal ones, respectively. The axioms ensure that
 
\begin{align*}
   (\Op I_{\ca I},\Ci,\Tc,\iota_{\ca C}) &&  (\Op I_{\ca T},\Ct,\Tc,\iota_{\ca C})
\end{align*}
are SFC-categories, each inducing a natural contravariant adjunction:
\begin{align*}
 \ptc:\Ci\leftrightarrows \Topp:\Om_{\ca I} &&  \ptc:\Ct\leftrightarrows \Topt:\Omc.
\end{align*}

Here, $\Topt$ denotes the category of spaces for which there exists an object in $\Ct$ with the desired universal property. All spaces in $\Topt$ are in fact fixpoints of the resulting adjunction, giving an natural adjunction that cannot be extended beyond its fixpoints; as is the case for the contravariant adjunction by Banaschewski and Pultr for $T_D$ spaces. In the natural dual adjunction $\ptc:\Ci\leftrightarrows\Topp:\Om_{\ca I}$, the fixpoints are exactly the sober spaces; in $\ptc:\Ct\leftrightarrows\Topt:\Omc X$ are at most the $T_D$-spaces. This result further supports the view already advanced in \cite{banaschewskitd}, where the $T_D$ axiom was first studied in relation to pointfree concepts, that the $T_D$-property and sobriety are each other's duals.

We find that there is a chain of morphisms between adjunctions:
\[
\begin{tikzcd}
    \Ct 
     \ar[r,shift left=2,"\ptc"{above}] 
     \ar[d,hookrightarrow]
    & \Topt
     \ar[l,shift left=2,"\Omc"{below}]
    \ar[l,phantom,"\dashv"{rotate=-90,anchor=center}]
    \ar[d,hookrightarrow] \\
    \ca C
      \ar[r,shift left=2,"\ptc"{above}] 
     \ar[d,"\ca{R}_{\ca I}",twoheadrightarrow]
    & \Topp
      \ar[l,shift left=2,"\Omc"{below}]
    \ar[l,phantom,"\dashv"{rotate=-90,anchor=center}]
    \ar[d,equal]\\
    \Ci
\ar[r,shift left=2,"\ptc"{above}] 
    & \Topp
     \ar[l,shift left=2,"\Om_{\ca I}"{below}]
    \ar[l,phantom,"\dashv"{rotate=-90,anchor=center}]
\end{tikzcd}
\]

\subsection*{Outline}

In Section 1 we recall the necessary background on categories and frames. We introduce the notion of $\Frm$-concrete category in a 2-categorical framework, and show that in this setting hom-categories are a preorder. We then introduce the notion of fiber as an instance of such a preorder poset. We also review the Bruns–Lakser completion, the Funayama embedding, and several characterizations of the $T_D$-axiom and its connection with essentiality in categories of lattices.

Section 2 introduces and explores our general framework. In Subsection \ref{ss motivating examples}, we illustrate the main motivating examples. In Subsection \ref{ss: defn topc ptc omc}, we introduce our main objects of study, \emph{spatializable}
$\Frm$-concrete categories, or \emph{SFC-categories}. We show that every SFC-category $(\Op,\ca
C,\Tc,\iota_{\ca C})$ admits a dualizing object, of which one half is the Sierpi\'{n}ski
space. The result by Porst and Tholen that dualizing objects induce natural
adjunctions, adapted to FC-categories, shows that there is a natural
adjunction $\ptc:\ca C\leftrightarrows\Topc:\Omc$, which we call
$\mf{Dual}(\Op,\ca C,\Tc,\iota_{\ca C})$. We show some preliminaries examples (Examples \ref{e: t0 spaces example}, \ref{e: td duality}, \ref{e: the sober example}) for categories of spaces.

The adjunction $\ptc:\ca C\leftrightarrows\Topc:\Omc$ is idempotent (Theorem \ref{t:
idempotence}), and $\Op$ determines a map in $\bd{RADJ}$ to the natural adjunction for frames and spaces (Proposition \ref{p: the zeta isomorphism and xi mate}).

In Subsection \ref{ss the ordered category sfc}, we define morphisms between SFC-categories and obtain the category $\SFC$; then we extend the assignment $\mf{Dual}(-)$ to a functor
$\mf{Dual}:\SFC\to \bd{RADJ}$ to the categories of adjunctions with right
maps. We also define the notion of adjunction between objects of $\SFC$, which
we call \emph{FC-adjunctions}, whose right adjoints preserve cartesian lifts
(Theorem \ref{l: right pfc adjoint preserve cartesian lifts}), as a result
we show that coreflectors coming from SFC-adjunctions are morphisms in
$\SFC$ (Proposition \ref{p: coreflective stable}). 

In Subsection \ref{ss: pointfree sobrification}, we develop a pointfree notion of sobrification. We restrict ourselves to SFC-categories where $\Op$ is a right FC-adjoint. We express the sober coreflection of an object of $\ca C$ as a pullback (Proposition \ref{p: sobrification as a pullback}), and deduce that when $\ca C$ has all pullbacks sober objects are coreflective in $\ca C$ (Corollary \ref{c: sufficient condition for sobrification to exist}). We also recover a generalization in this setting of the result in point-set topology stating that the sobrification of a subspace of a sober space is the intersection of all sober subspaces containing it (Theorem \ref{t: sobrification as a sober interior}).

In Subsection \ref{ss concrete examples 1}, we turn our attention to application of these results in concrete settings, and show that $\bd{Skula}$,
$\bd{MT}$, and $\bd{Raney}$ all provide instances of SFC-categories, thus
making their dual adjunctions with $\Topp$ a corollary of the result by Porst
and Tholen adapted to FC-categories (Theorems \ref{t: duality for raney
extensions}, \ref{t: duality for skula extensions}, and \ref{t: duality for
mt}). We show that the natural transformation $\mathit{ker}:\ca S\to \mf{Filt}$ naturally defines a functor $\ca S_{\ca S}:\bd{Skula}\to \bd{Raney}$, and that this coincides with a functor introduced in \cite{suarez26zero}. We slightly strengthen a result from \cite{bezhanishvili26raneyandmt} and show that there is a functor $\ca S_{\ca M}:\bd{MT}\to \bd{Raney}$. We show that the functors $\ca S_{\ca S}:\bd{Skula}\to \bd{Raney}$
and $\ca S_{\ca M}:\bd{MT}\to \bd{Raney}$ are morphisms in $\SFC$, and show that the map is $\bd{RADJ}$ they induce is also a left map of adjunctions (Propositions \ref{p: SS is morphism in SFC}
and \ref{p: SM is a morphism in SFC}).

Section 3 focuses on the subcategories of initial and terminal objects of fibers. In Subsection \ref{ss nat adj for fiber initials and terminals}, we slightly strengthen the characterization of Banaschewski and Pultr in \cite{banaschewskitd} of sober spaces as the fiber-maximal, and $T_D$-spaces as fiber-minimal objects of $\Om:\Topz^{\mf{op}}\to \Frm$ (Theorem \ref{t: sober and td as fiber initial and terminals}), and prove that the former coincide with fiber-initials and the latter with fiber-terminals. We introduce \emph{total} SFC-categories, those with all $T_0$-spaces as fixpoints and such that the fiber-initials and the fiber-terminals induce their own natural adjunction.
Letting $I_{\ca I}:\Ci\to \ca C$ be the inclusion of fiber-initial objects, and $I_{\ca T}:\Ct\to \ca C$ that of terminal-ones, we show that for a total SFC-category $(\Op,\ca C,\Tc,\iota_{\ca C})$, the natural adjunction $\mf{Dual}(\Op I_{\ca T},\Ct,\Tc,\iota_{\ca C})$ is a subobject of $\mf{Dual}(\Op,\ca C,\Tc)$, and that $\mf{Dual}(\Op I_{\ca I},\Ci,\Tc,\iota_{\ca C})$ is a quotient (Theorems \ref{t: the initial duality} and \ref{t: the terminal duality}).W e show that sober spaces coincide with the fixpoints of $\mf{Dual}(\Op I_{\ca I},\Ci,\Tc,,\iota_{\ca C})$, and that the fixpoints of $\mf{Dual}(\Op I_{\ca T},\Ct,\Tc,\iota_{\ca C})$ are $T_D$ spaces, although we do not have the reverse inclusion (Theorem \ref{t: main initial terminal theorem}).

In Subsection \ref{ss concrete examples 2}, we apply the results to the settings of Raney extensions, Skula extensions, and MT-algebras (Theorems \ref{t: terminal duality for raney}, \ref{t: terminal duality for skula}, \ref{t: terminal duality for mt}). We study fiber-terminal objects in the category of MT-algebras; we show that there exist fibers without terminal objects (Example \ref{e: a funayama which is not terminal}) and prove that in $\bd{MT}$, too, spectra of terminal objects coincide with the $T_D$-spaces (Proposition \ref{p: td are the spectra of mt terminals}). Finally, we obtain a duality for $T_D$-spaces as a result of a dual adjunction with fiber-terminal MT-algebras (Theorem \ref{t: terminal duality for mt}). 

\subsection*{Acknowledgements}

The first author acknowledges the partial support by CIDMA under the
Portuguese Foundation for Science and Technology (FCT,
\url{https://ror.org/00snfqn58}) Multi-Annual Financing Program for R\&D
Units, grants UID/4106/2025 and UID/PRR/4106/2025, the partial support by the
EPSRC Grant
\href{https://app.researchfish.com/awards/viewdetails/0?gorderby=organisation\&filter=67618964_5d4300dddf8976_AW}{UKRI174
(Toward a cohomology of theories)}, and the partial support by the Center for
Mathematics of the University of Coimbra (CMUC,
\url{https://doi.org/10.54499/UID/00324/2025}) under the Portuguese Foundation
for Science and Technology (FCT), Grants UID/00324/2025 and
  UID/PRR/00324/2025.

\section{Categorical foundations and background}

\subsection{Dual adjunctions}

We provide some background and well-known results on the theory of concrete
dual adjunctions. There are numerous references in the literature which study
duality theory from such an abstract perspective, among which we mention
\cite{dt89, porst91, cd98, hn18}.

We fix a faithful functor \( U \colon \ca C \to \ca S \).  Given morphisms \(
f \colon U(D) \to U(C) \) in \( \ca{S} \) and \( p \colon D \to C \) in \(
\ca{C} \), we say that \( f \) \textit{lifts to} \( p \) if \( U(p) = f \). We
highlight that lifts are necessarily unique.

A \textit{cartesian lifting}\footnote{Traditionally called \textit{$U$-initial
lifting}, as in \cite{porst91, hn18}.} (\textit{with respect to $U$}) of a
family of morphisms \( \{ f_i \colon X \to U(C) \mid i \in I \} \)\footnote{In
general, the codomain of each \( f_i \) is allowed to vary with \(i \in I\),
but we do not make use of this additional feature in the present work.} in \(
\ca S \) is a pair \( (\overline X, \theta) \) consisting of an object \(
\overline X \) in \( \ca C \) and an isomorphism \( \theta \colon U(\overline
X) \cong X \) such that any morphism \( h \colon U(D) \to U(\overline X) \)
lifts to a morphism \( D \to \overline X \) if and only if each morphism \(
f_i \theta h \colon U(D) \to U(C) \) lifts to a morphism \( D \to C \).

We say a family of morphisms \( \{ p_i \colon D \to C \mid i \in I\} \) in \(
\ca C \) is \textit{cartesian} (\textit{with respect to} $U$) if \(
(D,\id_{U(D)}) \) is a cartesian lifting (with respect to $U$) of the family
\( \{ U(p_i) \colon U(D) \to U(C) \mid i \in I \} \).

\begin{remark}
  Given \( (\overline X, \theta) \) a cartesian lifting w.r.t. $U$ of a family
  \( \{ f_i \colon X \to U(C) \mid i \in I \} \), we have that \( f_i\theta \)
  lifts for all \( i \in I \), since the identity on \( U(\overline X) \)
  lifts; we denote the lift of \( f_i\theta \) by \( \overline f_i \). We note
  that the family \( \{ \overline f_i \mid i \in I \} \) is cartesian w.r.t
  $U$, and we say this is the \textit{lifted family} of the given cartesian
  lift.
\end{remark}

\begin{remark}
  \label{r: morphism-lifting}
  Let \( g \colon Y \to X \) be a morphism in \( \ca S \), and let \(
  (\overline X, \theta) \) and \( (\overline Y, \omega) \) be cartesian lifts
  of the families
  \begin{equation*}
    \{ f_i \colon X \to U(C) \mid i \in I \}
      \qquad \text{and} \qquad
    \{ f_i g \colon Y \to U(C) \mid i \in I \},
  \end{equation*}
  respectively. We also let \( \{ \overline f_i \colon \overline X \to C \mid
  i \in I \} \) and \( \{ \overline h_i \colon \overline Y \to C \mid i \in I
  \} \) be the respective lifted families. Since we have
  \begin{equation*}
    U(\overline h_i) = f_ig\omega = f_i\theta \theta^{-1}g\omega,
  \end{equation*}
  we obtain the lift \( \overline g \colon \overline Y \to \overline X \) of
  \( \theta^{-1} g \omega \), satisfying \( \overline f_i \overline g =
  \overline h_i \) for all \( i \in I \). We will say that \( \overline g \)
  is the morphism induced by \( g \).
\end{remark}

We consider faithful functors \( |-| \colon \ca C \to \Set \) and \( |-|
\colon \ca X \to \Set \), and we say these map objects \( C \) in \( \ca C \)
and \( X \in \ca X \) to the \textit{carrier sets} \( |C| \) and \( |X| \)
of \( C \) and \( X \), respectively. 

Any dual adjunction
\begin{equation}
  \label{e: dual adj}
  \begin{tikzcd}[column sep=large]
    \ca C \ar[r, "F", shift left=2]
                \ar[r,"\dashv"{rotate=-90},phantom]
      & \opp{\ca X} \ar[l, "G", shift left=2]
  \end{tikzcd}
\end{equation}
with units \( \eta \colon \id \to GF \) in \( \ca C \) and \( \sigma \colon
\id \to FG \) in \( \ca X \) restricts to a dual equivalence
\begin{equation*}
  \begin{tikzcd}[column sep=large]
    \mf{Fix}(\eta)  \ar[r, "F", shift left=2]
                \ar[r,"\dashv"{rotate=-90},phantom]
      & \opp{\mf{Fix}(\sigma)} \ar[l, "G", shift left=2]
  \end{tikzcd}
\end{equation*}
where \( \mf{Fix}(\eta) \) and \( \mf{Fix}(\sigma) \) are respectively the
full subcategories of \( \ca C \) and \( \ca X \) of those objects for which
the components of \( \eta \) and \( \sigma \) are invertible. Furthermore, if
\eqref{e: dual adj} is an idempotent adjunction, we note that both are
reflective subcategories of \( \ca C \) and \( \ca X \) respectively; in fact,
they are precisely the Eilenberg-Moore categories for the induced monads
\cite{at69}.

Given a pair \( (\Tc \in \ca C,\mathbb S \in \ca X) \) of objects, we obtain
evaluation maps
\begin{align*}
  \nu_c &= 
    \begin{tikzcd}[ampersand replacement=\&]
      \ca C(C,\Tc) \ar[r,"|-|"] 
        \& \Set(|C|,|\Tc|) \ar[r,"\mathsf{ev}_c"]
        \& {|\Tc|}
    \end{tikzcd} \\
  \chi_x &= 
    \begin{tikzcd}[ampersand replacement=\&]
      \ca X(X,\mathbb S) \ar[r,"|-|"] 
        \& \Set(|X|,|\mathbb S|) \ar[r,"\mathsf{ev}_x"]
        \& {|\mathbb S|}
    \end{tikzcd} 
\end{align*}
for all \( c \in |C| \) and \( x \in |X| \). We say that the pair \( (\Tc,
\mathbb S) \) \textit{represents} the dual adjunction \eqref{e: dual adj} if
we have natural isomorphisms \( \phi \colon |F| \cong \ca C(-,\Tc) \) and \(
\gamma \colon |G| \cong \ca X(-,\mathbb S) \). 

We say that a dual adjunction \eqref{e: dual adj} represented by a pair \(
(\Tc,\mathbb S) \) of objects is \textit{natural} if the families 
\begin{align*}
  \{ \gamma_{F(C)}|\eta_C|(c) \colon F(C) \to \mathbb S &\mid c \in |C| \}
  & \{ \phi_{G(X)}|\epsilon_X|(x) \colon G(X) \to \Tc &\mid x \in |X| \}
\end{align*}
are cartesian with respect to \( |-| \colon \ca X \to \Set \) and \( |-|
\colon \ca C \to \Set \) for all \( C \in \ca C \) and all \( X \in \ca X \).

The following theorem of \cite{dt89, porst91} describes how to construct a
natural dual adjunction from a \textit{dualizing object}: this consists of a
triple \( (\Tc,\mathbb S,\iota) \) where \( \Tc \in \ca C \) and \( \mathbb S
\in \ca X \) are objects, and \( \iota \colon |\Tc| \cong |\mathbb S| \) is an
isomorphism, such that the following two conditions hold:
\begin{enumerate}[label=(C\arabic*),ref=C\arabic*]
  \item
    \label{i: cartlift1}
    The family \( \{ \iota \circ \nu_a \colon \ca C(C,\Tc) \to |\mathbb S|
    \mid a \in |C| \} \) has a cartesian lift \( (F(C),\phi_C) \) for all \( C
    \in \ca C \).
  \item
    \label{i: cartlift2}
    The family \( \{ \iota^{-1} \circ \chi_x \colon \ca X(X,\mathbb S) \to
    |\Tc| \mid x \in |X| \} \) has a cartesian lift \( (G(X),\gamma_X) \) for
    all \( X \in \ca X \).
\end{enumerate}

\begin{theorem}
  \label{t: natduality}
  Let \( |-| \colon \ca C \to \Set \) and \( |-| \colon \ca X \to \Set \) be
  faithful functors, and let \( (\Tc,\mathbb S,\iota) \) be a dualizing
  object. We have that
  \begin{itemize}
    \item
      the cartesian lifts \( (F(C),\phi_C) \) provided by \eqref{i: cartlift1}
      for each \( C \in \ca C \) are the underlying object part of a functor
      \( F \colon \ca C \to \opp{\ca X} \) and a natural isomorphism \( \phi
      \colon |F| \cong \ca C(-,\Tc) \), 
    \item
      the cartesian lifts \( (G(X),\gamma_X) \) provided by \eqref{i:
      cartlift2} for each \( X \in \ca X \) are the underlying object part 
      of a functor \( G \colon \opp{\ca X} \to \ca C \) and a natural
      isomorphism \( \psi \colon |G| \cong \ca X(-,\mathbb S) \).
  \end{itemize}
  Moreover, \( F \) and \( G \) yield a dual adjunction \eqref{e: dual adj},
  represented by the pair \( (\Tc, \mathbb S) \), and this dual adjunction is
  natural.
\end{theorem}

The action of the functors \( F \) and \( G \) on morphisms, as well as the
components of the units, are obtained via Remark~\ref{r: morphism-lifting}.

The following elementary observation will also prove to be helpful.

\begin{lemma}
  \label{l: mutually inverse frame isos that lift}
  Let $U : \ca C\to \ca S$ be a faithful functor and let $C,D\in \ca C$. If
  $i:U(C)\to U(D)$ is an isomorphism such that \( i \) and \( i^{-1} \) lift
  to morphisms \( \overline i \) and \( \overline j \) respectively, then we
  have \( \overline i^{-1} = \overline j \).
\end{lemma}
\begin{proof}
  Indeed, since \( i = U(\overline{i}) \) and \( i^{-1} = U(\overline{j})
  \), it follows that $U(\overline{j} \overline{i})$ and $U(\overline{i}
  \overline{j})$ are respectively the identities on \( U(C) \) and \( U(D) \).
  By faithfulness, we deduce that $\overline{j}{} \overline{i}$ is the
  identity on $C$, and similarly that $\overline{i}{} \overline{j}$ is the
  identity on $D$.
\end{proof}

\subsection{Morphisms of adjunctions}\label{ss: maps of adjunctions}

We recall the fundamentals. See \cite{mar65, lin65, kel69, dub70, pal71, ks74,
gra74}.

We consider categories and functors as in the following diagram:
\begin{equation}
  \label{e: fixed.adj.map 1}
  \begin{tikzcd}[row sep=huge, column sep=large]
    \ca A \ar[r, "F"{name=A,above},shift left=2]
           \ar[d,"U",swap]
    & \ca B \ar[l,"G"{name=B,below},shift left=2] 
           \ar[d,"V"] \\
    \ca C \ar[r, "P"{name=C,above},shift left=2]
    & \ca D \ar[l,"Q"{name=D,below},shift left=2]
    \ar[from=A,to=B,"\dashv"{rotate=-90,anchor=center},phantom]
    \ar[from=C,to=D,"\dashv"{rotate=-90,anchor=center},phantom]
  \end{tikzcd}
\end{equation}
where \( F \dashv G \) with unit \( \eta \colon \id \to GF \) and counit \(
\epsilon \colon FG \to \id \), and \( P \dashv Q \) with unit \( \nu \colon
\id \to QP \) and counit \( \delta \colon PQ \to \id \).

\begin{lemma}
  \label{l: mate-corresp}
  There is a bijection between the sets \( \mf{Nat}(UG,\,QV) \) and \(
  \mf{Nat}(PU,\, VF) \). 
\end{lemma}
\begin{proof}
  Since we have adjunctions \( - \circ G \dashv - \circ F \) and \( P \circ -
  \dashv Q \circ - \), we have a square of isomorphisms
  \begin{equation*}
    \begin{tikzcd}
      \mf{Nat}(UG,QV) \ar[r,"\cong"] \ar[d,swap,"\cong"]
        & \mf{Nat}(U,QVF) \ar[d,"\cong"] \\
      \mf{Nat}(PUG,V) \ar[r,swap,"\cong"]
        & \mf{Nat}(PU,VF)
    \end{tikzcd}
  \end{equation*}
  from which our claim follows.
\end{proof}

Pairs of natural transformations \( \zeta \colon UG \to QV \) and \( \xi
\colon PU \to VF \) that correspond to each other via Lemma~\ref{l:
mate-corresp} are said to be a \textit{mate-pair} with respect to \( F \dashv G
\) and \( P \dashv Q \). The explicit relationships between \( \zeta \) and \(
\xi \) obtained from the previous lemma may be given by the following pasting
diagrams:
\begin{equation}
  \label{e: mate-pair-eqs-1}
  \begin{tikzcd}
    \ca A \ar[d,"U",swap]
      \ar[rd,Rightarrow,"\zeta",shorten=4mm]
      & \ca B \ar[l,"G",swap] \ar[d,"V"] \\
    \ca C
      & \ca D \ar[l,"Q"] 
  \end{tikzcd}
    =
  \begin{tikzcd}
    \ca B \ar[d,"G",swap] \ar[rd,"\id"{name=A},bend left] \\
    \ca A \ar[d,"U",swap]
          \ar[r,"F"]
      & \ca B \ar[d,"V"] \\
    \ca C \ar[r,"P"]
          \ar[ru,"\xi",Rightarrow,shorten=4mm]
          \ar[rd,bend right,"\id"{name=B}",swap]
      & \ca D \ar[d,"Q"] \\
    & \ca D
    \ar[from=2-1,to=A,"\epsilon",Rightarrow,shorten=3mm]
    \ar[from=B,to=3-2,"\nu",Rightarrow,shorten=3.5mm]
  \end{tikzcd}
  \qquad \qquad
  \begin{tikzcd}
    & \ca A \ar[d,"F"] \ar[ld,bend right,"\id"{name=A},swap] \\
    \ca A \ar[d,"U",swap]
      \ar[rd,Rightarrow,"\zeta",shorten=4mm]
      & \ca B \ar[l,"G",swap] \ar[d,"V"] \\
    \ca C
      & \ca D \ar[l,"Q"] 
    \ar[from=A,to=2-2,"\eta",Rightarrow,shorten=3.5mm]
  \end{tikzcd}
    =
  \begin{tikzcd}
    \ca A \ar[d,"U",swap]
          \ar[r,"F"]
      & \ca B \ar[d,"V"] \\
    \ca C \ar[r,"P"]
          \ar[ru,"\xi",Rightarrow,shorten=4mm]
          \ar[rd,bend right,"\id"{name=B}",swap]
      & \ca D \ar[d,"Q"] \\
    & \ca D
    \ar[from=B,to=2-2,"\nu",Rightarrow,shorten=3.5mm]
  \end{tikzcd}
\end{equation}
\begin{equation}
  \label{e: mate-pair-eqs-2}
  \begin{tikzcd}
    \ca A \ar[d,"U",swap]
      \ar[rd,Rightarrow,"\zeta",shorten=4mm]
      & \ca B \ar[l,"G",swap] \ar[d,"V"] \\
    \ca C \ar[d,"P",swap]
      & \ca D \ar[l,"Q",swap] \ar[ld,bend left,"\id"{name=B}] \\
    \ca D
    \ar[from=2-1,to=B,"\delta",Rightarrow,shorten=3.5mm]
  \end{tikzcd}
    =
  \begin{tikzcd}
    \ca B \ar[d,"G",swap] \ar[rd,"\id"{name=A},bend left] \\
    \ca A \ar[d,"U",swap]
          \ar[r,"F"]
      & \ca B \ar[d,"V"] \\
    \ca C \ar[r,"P",swap]
          \ar[ru,"\xi",Rightarrow,shorten=4mm]
      & \ca D 
    \ar[from=2-1,to=A,"\epsilon",Rightarrow,shorten=3mm]
  \end{tikzcd}
  \qquad \qquad
  \begin{tikzcd}
    & \ca A \ar[d,"F"] \ar[ld,bend right,"\id"{name=A},swap] \\
    \ca A \ar[d,"U",swap]
      \ar[rd,Rightarrow,"\zeta",shorten=4mm]
      & \ca B \ar[l,"G",swap] \ar[d,"V"] \\
    \ca C \ar[d,"P",swap]
      & \ca D \ar[l,"Q",swap] \ar[ld,bend left,"\id"{name=B}] \\
    \ca D
    \ar[from=A,to=2-2,"\eta",Rightarrow,shorten=3.5mm]
    \ar[from=3-1,to=B,"\delta",Rightarrow,shorten=3.5mm]
  \end{tikzcd}
    =
  \begin{tikzcd}
    \ca A \ar[d,"U",swap]
          \ar[r,"F"]
      & \ca B \ar[d,"V"] \\
    \ca C \ar[r,"P",swap]
          \ar[ru,"\xi",Rightarrow,shorten=4mm]
      & \ca D 
  \end{tikzcd}
\end{equation}
The triangle identities guarantee that any one of the equations implies the
other three.

Functors \( U, V \) as in \eqref{e: fixed.adj.map 1} and a mate-pair \( \zeta
\colon UG \to QV \), \( \xi \colon PU \to VF \) constitute a \textit{right
morphism of adjunctions}, or a \textit{right adjunction morphism}, if \( \zeta
\) is invertible. Likewise, we say this data defines a \textit{strong
morphism} of adjunctions, or \textit{strong adjunction morphism} if \( \xi \)
is invertible as well. 

We note that there is some redundancy in the data for a right/strong map of
adjunctions, since \( \xi \) is determined by \( \zeta \) and vice-versa, any
of the equations in \eqref{e: mate-pair-eqs-1} and \eqref{e: mate-pair-eqs-2}.
We may omit one of the two mate-pairs when describing a right/strong morphism
of adjunctions.

If \( (U,V,\zeta,\xi) \) and \( (U',V',\zeta',\xi') \) are two right morphisms
of adjunctions \( (F\dashv G) \to (P\dashv Q) \), a 2-cell \( (U,V,\zeta) \to
(U',V',\zeta') \) consists of a pair of 2-cells \( \beta \colon U \to U' \)
and \( \gamma \colon V \to V' \) such that either (and therefore both) of the
following equations of pasting diagrams holds:
\begin{equation}
  \begin{tikzcd}[row sep=huge, column sep=large]
    \ca A \ar[d,"U"{name=W},swap,bend right]
          \ar[d,"U'"{name=X},bend left]
          \ar[rd,Rightarrow,shorten=8mm,"\zeta'",shift left=1]
    & \ca B \ar[l,"G",swap] \ar[d,"V'"] \\
    \ca C 
    & \ca D \ar[l,"Q"]
    \ar[from=W,to=X,Rightarrow,shorten=2mm,"\beta"]
  \end{tikzcd}
  = 
  \begin{tikzcd}[row sep=huge, column sep=large]
    \ca A \ar[d,"U",swap]
          \ar[rd,Rightarrow,shorten=8mm,"\zeta",swap,shift right=1]
    & \ca B \ar[d,"V"{name=Y},swap,bend right] 
            \ar[d,"V'"{name=Z},bend left] 
            \ar[l,"G",swap] \\
    \ca C & \ca D \ar[l,"Q"]
    \ar[from=Y,to=Z,Rightarrow,shorten=2mm,"\gamma"]
  \end{tikzcd}
  \qquad \qquad
  \begin{tikzcd}[row sep=huge, column sep=large]
    \ca A \ar[d,"U"{name=W},swap,bend right]
          \ar[d,"U'"{name=X},bend left]
          \ar[r,"F"]
    & \ca B \ar[d,"V'"] \\
    \ca C \ar[ru,Rightarrow,shorten=8mm,"\xi'",shift right=1,swap]
          \ar[r,"P"]
    & \ca D 
    \ar[from=W,to=X,Rightarrow,shorten=2mm,"\beta"]
  \end{tikzcd}
  = 
  \begin{tikzcd}[row sep=huge, column sep=large]
    \ca A \ar[d,"U",swap]
          \ar[r,"F"]
    & \ca B \ar[d,"V"{name=Y},swap,bend right] 
            \ar[d,"V'"{name=Z},bend left] \\
    \ca C \ar[ru,Rightarrow,shorten=8mm,"\xi",shift left=1]
          \ar[r,"P"]
    & \ca D 
    \ar[from=Y,to=Z,Rightarrow,shorten=2mm,"\gamma"]
  \end{tikzcd}
\end{equation}

We call $\bd{RADJ}$ the 2-category of adjunctions equipped with right
adjunction morphisms, and $\bd{SADJ}$ the 2-category of adjunctions with strong
adjunction morphisms.

The following result, merely stating that ``if a square of right adjoints
commutes, then the respective square of left adjoints also commutes'', is
nonetheless an important observation, which may be obtained as a consequence
of doctrinal adjunction \cite[Theorem 1.4]{kel74}.

\begin{lemma}
  \label{l: right adj comm implies left adj comm}
  Let \( (U,V,\zeta,\xi) \colon (F\dashv G) \to (P\dashv Q) \) and \(
  (L,K,\omega,\chi) \colon (P\dashv Q) \to (F\dashv G) \) be right adjunction
  morphisms, so that \( \zeta \colon UG \to QV \) and \( \omega \colon LQ \to
  GK \) are natural isomorphisms, and assume we have adjunctions \( L \dashv U
  \), \( K \dashv V \). This situation is depicted via the following diagram
  \begin{equation}
    \label{e: fixed.adj.map 3}
    \begin{tikzcd}[row sep=huge, column sep=huge]
      \ca A \ar[r, "F"{name=A,above},shift left=2]
             \ar[d,"U",shift left=2]
             \ar[d,phantom,"\dashv"]
      & \ca B \ar[l,"G"{name=B,below},shift left=2] 
             \ar[d,phantom,"\dashv"]
             \ar[d,"V",shift left=2] \\
      \ca C \ar[r, "P"{name=C,above},shift left=2]
            \ar[u,"L",shift left=2]
      & \ca D \ar[l,"Q"{name=D,below},shift left=2]
            \ar[u,"K",shift left=2]
      \ar[from=A,to=B,"\dashv"{rotate=-90,anchor=center},phantom]
      \ar[from=C,to=D,"\dashv"{rotate=-90,anchor=center},phantom]
    \end{tikzcd}
  \end{equation}

  The following are equivalent:
  \begin{itemize}
    \item
      The pair of units \( \id \to UL \), \( \id \to VK \) and the pair of
      counits \( LU \to \id \), \( KV \to \id \) of the adjunctions \( L
      \dashv U \) and \( K \dashv V \) define 2-cells in \( \bd{RADJ} \).
    \item
      We have that \( \chi \colon FL \to KP \) is invertible (that is, \(
      (L,K,\omega,\chi) \) is a strong adjunction morphism), and \( \chi^{-1}
      \) is the mate of \( \xi \colon PU \to VF \) with respect to \( L \dashv
      U \) and \( K \dashv V \).
  \end{itemize}

  When one (therefore all) of the above conditions holds, we have an
  adjunction \( (L,K,\omega) \dashv (U,V,\zeta) \) in the 2-category \(
  \bd{RADJ} \).
\end{lemma}

\subsection{Fibers of a functor}

Throughout this work, our main objects of study are categories equipped with a
faithful functor to the category \( \Frm \) of frames, and some of our results
are stated in terms of suitable adjunctions between these categories. Such
adjunctions naturally occur in an ambient 2-category, which we proceed to
describe.

We let \( \Con(\Frm) \) be the 2-category whose objects are
\textit{$\Frm$-concrete categories} \cite{ahs90,por94}, or
\emph{FC-categories} for short: pairs \( (\ca C, \Op_{\ca C}) \) where \( \ca
C \) is a category and \( \Op_{\ca C} \colon \ca C \to \Frm \) is a faithful
functor. For each pair of FC-categories \( (\ca C, \Op_{\ca C}) \) and \( (\ca
D, \Op_{\ca D}) \), the hom-category 
\begin{equation}
  \label{eq:confrm-homcat}
  \Con(\Frm)\big( (\ca C, \Op_{\ca C}), (\ca D, \Op_{\ca D}) \big)  
\end{equation}
consists of the following data: 
\begin{itemize}
  \item
    Objects are morphisms of FC-categories \( (\ca C, \Op_{\ca C}) \to (\ca D,
    \Op_{\ca D}) \), which we call \textit{FC-functors}, consisting of pairs
    \( (F,\phi) \) where \( F \colon \ca C \to \ca D \) is a functor and \(
    \phi \colon \Op_{\ca D}  F \to \Op_{\ca C} \) is a natural isomorphism;
  \item
    Morphisms \( (F,\phi) \to (G,\psi) \) are \textit{2-cells} of
    FC-categories, which consist of a natural transformation \( \phi \colon F
    \to G \) such that \( \psi \circ \Op_{\ca D}(\phi) = \phi \).
  \item
    Composition and identities are precisely the usual (vertical) composition
    and identities of natural transformations between functors.
\end{itemize}
If \( (F,\phi ) \colon (\ca C, \Op_{\ca C}) \to (\ca D, \Op_{\ca D}) \) and \(
(G,\psi) \colon (\ca D, \Op_{\ca D}) \to (\ca E, \Op_{\ca E}) \) are
FC-functors, their composite is given by the pair \( (GF, \phi \circ \psi_F)
\), and we likewise define horizontal composition of 2-cells. For an
FC-category \( \Op_{\ca C} \colon \ca C \to \Frm \), the identity FC-functor
is precisely the pair \( (\id_{\ca C}, \id_{\Op_{\ca C}} ) \).

We will often refer to the following observation.
\begin{lemma}\label{l: morphisms in fc are faithful}
    If \( (F,\phi) \colon (\ca C,\Op_{\ca C}) \to (\ca D,\Op_{\ca D}) \) is
    an FC-functor, \( F \) is faithful.
\end{lemma}
\begin{proof}
Follows from \( \Op_{\ca C}\) and \(\Op_{\ca D} \) being faithful, and there being a natural isomorphism \( \zeta \colon \Op_{\ca D}  F \to \Op_{\ca C} \).
\end{proof}

\begin{proposition}\label{p: hom cat is a preorder}
In the hom-category \(\Con(\Frm)\big( (\ca C, \Op_{\ca C}), (\ca D, \Op_{\ca D}) \big)\), if \( \chi,\omega \colon (F,\phi) \to (G,\psi) \) are 2-cells then \(
    \chi = \omega \).
\end{proposition}
\begin{proof}
By faithfulness of \( \Op_{\ca D} \), the components of \( \chi \) are monomorphisms and epimorphisms. Then, the components of \( \Op_{\ca D}(\chi) = \psi^{-1} \circ \phi \) are isomorphisms. Then, \( \Op_{\ca
    D}(\chi) = \Op_{\ca D}(\omega) \), hence $\chi=\omega$ by faithfulness of $\Op_{\ca D}$.   
\end{proof}

Proposition \ref{p: hom cat is a preorder} establishes the hom-category
\eqref{eq:confrm-homcat} to be a preorder, for all \( (\ca C, \Op_{\ca C}) \)
and \( (\ca D,\Op_{\ca D}) \).  Thus, \( \Con(\Frm) \) is an \textit{ordered}
category (a category enriched in preordered sets); hence a 2-cell \( (F,\phi)
\to (G,\psi) \) simply witnesses that a property holds. One may write \(
(F,\phi) \leq (G,\psi) \) to reflect this. 

A special class of these hom-preorders are of particular interest throughout
this work: the \textit{fibers} of an FC-category \( \Op \colon \ca C \to \Frm
\) at a frame \( L \). These can be succinctly defined as the hom-preorder
\begin{equation*}
  \Op^{-1}(L) = \Con(\Frm)\big( (1, L), (\ca C, \Op) \big)  
\end{equation*}
where \( L \colon 1 \to \Frm \) is the (faithful) functor from the terminal
category with value \( L \).

More explicitly, for an FC-category \( \Op \colon \ca C \to
\Frm \), we let the \textit{fiber of} \( \Op \) \textit{at} a frame $L$ be the
preorder whose objects are pairs $(C,\theta^C)$ where $C\in \ca C$ and
$\theta^C:\Op(C)\to X$ is an isomorphism, and whose morphisms $p:(C,\theta^C)\to
(D,\theta^D)$ are morphisms $p:C\to D$ in \( \ca C \) such that $\theta^D
\circ \Op(p) = \theta^C$, i.e. the following diagram commutes
\begin{equation*}
  \begin{tikzcd}[row sep=large,column sep=large]
    \Op(C)
    \ar[dr,"\theta^C",swap]
    \ar[rr,"\Op(p)"]
    && \Op(D)
    \ar[dl,"\theta^D"]\\
    & L
  \end{tikzcd},
\end{equation*}
and we denote this preorder by \( \Op^{-1}(L) \). 

Returning to the topic of adjunctions in \( \Con(\Frm) \) -- which we call
\textit{FC-adjunctions} -- we highlight that since this is an ordered
category, it follows that FC-adjunctions are simply Galois connections. To be
explicit, we let \( (F,\phi) \colon (\ca C, \Op_{\ca C}) \to (\ca D, \Op_{\ca
D}) \) and \( (U,\theta) \colon (\ca D, \Op_{\ca D}) \to (\ca C, \Op_{\ca C})
\) be FC-functors. An unraveling of the definitions yields the following
observation:

\begin{lemma}
  \label{l: confrm-adj}
  The following are equivalent:
  \begin{enumerate}[label=(\alph*)]
    \item
      We have an FC-adjunction \( (F,\phi) \dashv (U,\theta) \).
    \item
      We have \( (F,\phi) \circ (U,\theta) \leq \id_{(\ca D,\Op_{\ca D})} \)
      and \( \id_{(\ca C,\Op_{\ca C})} \leq (U,\theta) \circ (F,\phi) \).
    \item
      There exist natural transformations \( \eta \colon \id \to UF \) and \(
      \epsilon \colon FU \to \id \) such that \( \Op_{\ca C}(\eta) =
      \theta^{-1}_F \circ \phi^{-1} \) and \( \Op_{\ca D}(\epsilon) = \theta
      \circ \phi_U \).
  \end{enumerate}
  Moreover, the unit and counit are uniquely determined.
\end{lemma}

\begin{proof}
  The equivalence between (a) and (b) is simply the definition of adjunction
  in an ordered category, while the equivalence between (b) and (c) is merely
  restating the definition of 2-cell in \( \mathbf{Con}(\Frm) \).
\end{proof}

\begin{remark}
  \label{r: redundant info in FC-adj}
  We note that, for an FC-adjunction \( (F,\phi) \dashv (G,\theta) \), we
  necessarily have that \( \phi^{-1} \) and \( \theta \) define a mate-pair
  with respect to \( F \dashv G \) and \( \id_{\Frm} \dashv \id_{\Frm} \).
  Explicitly, we have the following equations:
  \begin{equation*}
    \phi^{-1} = \theta_F \circ \Op_{\ca C}(\eta)
      \qquad \theta = \Op_{\ca D}(\epsilon) \circ \phi^{-1}_G.
  \end{equation*}
  Indeed, this shows there is some redundant data and conditions in an
  FC-adjunction, an observation made precise with the following lemma.
\end{remark}

\begin{lemma}
  \label{l: when does fc-functor have adjoint}
  Let \( (F,\phi) \colon (\ca C, \Op_{\ca C}) \to (\ca D, \Op_{\ca D}) \) be
  an FC-functor. We have that:
  \begin{enumerate}[label=(\alph*)]
    \item
      \label{l: when has right adjoint}
      \( (F,\phi) \) has a left FC-adjoint \( (L,\lambda) \) if and only if we
      have an ordinary adjunction \( L \dashv F \) whose unit \( \eta \colon \id
      \to FL \) is such that \( \Op_{\ca D}(\eta) \) is invertible.
    \item
      \label{l: when has left adjoint}
      \( (F,\phi) \) has a right FC-adjoint \( (R,\rho) \) if and only if we
      have an ordinary adjunction \( F \dashv R \) whose counit \( \epsilon
      \colon FR \to \id \) is such that \( \Op_{\ca D}(\epsilon) \) is
      invertible.
  \end{enumerate}
\end{lemma}
\begin{proof}
  The direct implications of both (a) and (b) are immediate.

  Let \( L \dashv F \) be an adjunction with unit \( \eta \).  If \( \Op_{\ca
  D}(\eta) \) is invertible, we let \( \lambda = \Op_{\ca D}(\eta)^{-1} \circ
  \phi_L^{-1} \), so that \( (L,\lambda) \) defines an FC-functor. By
  definition, we have \( \id_{(\ca D,\Op_{\ca D})} \leq (F,\phi) \circ
  (L,\lambda) \), so we need only confirm that \( (L,\lambda) \circ (F,\phi)
  \leq \id_{(\ca C,\Op_{\ca C})} \), which is to say that there exists a
  natural transformation \( \delta \colon LF \to \id \) such that \( \Op_{\ca
  C}(\delta) \circ \lambda_F^{-1} = \phi \circ \lambda_F \). Indeed, as a
  consequence of a triangle identity, it follows that the counit of \( L
  \dashv F \) satisfies this.

  The converse implication of (b) follows similarly.
\end{proof}

Let \( (U,\theta) \colon (\ca D, \Op_{\ca D}) \to (\ca C, \Op_{\ca C}) \) be
an FC-functor. We say it \textit{preserves the cartesian lifting of a family} 
\( \{ f_i \colon X \to \Op_{\ca D}(D) \mid i \in I \} \), if \(
\big(U(\overline X), \omega \circ \theta_{\overline X}\big) \) is a cartesian
lift with respect to \( \Op_{\ca C} \) of the family \( \{ \theta_D^{-1} \circ
f_i \colon X \to \Op_{\ca C}(U(D)) \mid i \in I \} \) whenever \( (\overline
X, \omega) \) is the cartesian lift with respect to \( \Op_{\ca D} \) of \(
\{ f_i \mid i \in I \} \).

\begin{lemma}
  Let \( (U,\theta) \colon (\ca D, \Op_{\ca D}) \to (\ca C, \Op_{\ca C}) \) be
  an FC-functor. Given a family \( \{ f_i \colon X \to \Op_{\ca D}(D) \mid i
  \in I \} \) with a cartesian lifting \( (\overline X,\omega) \), and lifted
  family \( \{ \overline f_i \mid i \in I \} \), the following are equivalent:
  \begin{enumerate}[label=(\alph*)]
    \item
      \( (U,\theta) \) preserves the cartesian lifting of \( \{ f_i \mid i \in
      I \} \). 
    \item
      The family \( \{ F(\overline f_i) \mid i \in I \} \) is cartesian.
  \end{enumerate}
\end{lemma}
\begin{proof}
  We have that
  \begin{equation*}
    \Op_{\ca C}U(\overline f_i) = \theta^{-1}_D \circ f_i \circ \omega \circ
    \theta_{\overline X},
  \end{equation*}
  so it immediately follows that the family \( \{ F(\overline f_i) \mid i \in
  I \} \) is cartesian if and only if \( (F(\overline X),\omega \circ
  \theta_{\overline X}) \) is the cartesian lift of the family \( \{
  \theta^{-1}_D \circ f_i \mid i \in I \} \).
\end{proof}

\begin{lemma}
  \label{l: radj-prsv-carts}
  Right adjoint FC-functors preserve cartesian liftings of all families.
\end{lemma}
\begin{proof}
  We consider an FC-adjunction
  \begin{equation*}
    \begin{tikzcd}[column sep=large]
      (\ca C,\Op_{\ca C}) \ar[r,shift left=2.5,"{(F,\phi)}"]
               \ar[r,"\dashv"{rotate=-90},phantom]
      & (\ca D,\Op_{\ca D}), \ar[l,shift left=2.5,"{(U,\theta)}"]
    \end{tikzcd}
  \end{equation*}
  a family \( \{ f_i \colon X \to \Op_{\ca D}(D) \mid i \in I \} \) with a
  cartesian lifting \( (\overline X,\omega) \) with respect to \( \Op_{\ca D}
  \), and a morphism \( h \colon \Op_{\ca C}(C) \to \Op_{\ca C}(U(\overline
  X)) \). 

  Denoting the lifted family \( \{ \overline f_i \colon \overline X \to D \mid
  i \in I \} \), we observe that if the morphism \( \Op_{\ca C}U(\overline
  f_i) \circ h \colon \Op_{\ca C}(C) \to \Op_{\ca C} U(D) \) lifts to a
  morphism \( q_i \colon C \to U(D) \) for each \( i \in I \), then, letting
  \( q^\flat_i = \epsilon_D \circ F(q_i) \), we have
  \begin{equation*}
    \Op_{\ca D}(q^\flat_i) 
      = \Op_{\ca D}(\overline f_i) 
          \circ \theta_{\overline X} 
          \circ h
          \circ \phi_C^{-1}
  \end{equation*}
  for all \( i \in I \), meaning that \( q^\flat_i \) is a lift of the
  morphism on the right-hand side. Since \( \overline X \) is cartesian, it
  follows that \( \theta_{\overline X} \circ h \circ \phi_C^{-1} \) must have a
  lift, say, \( k \colon F(C) \to \overline X \). 

  Now, if we let \( k^\sharp = U(k) \circ \eta_C \), we observe that \(
  \Op_{\ca C}(k^\sharp) = h \). Thus, we conclude that \( (U(\overline X),
  \omega \circ \theta_{\overline X}) \) is a cartesian lifting of \( \{ f_i
  \mid i \in I \} \).
\end{proof}

\subsection{Frames and sublocales}
We give some background on the facts about frames and sublocales needed in this work. For more details, we refer the reader to \cite{picadopultr2012frames}. For every frame $L$, we call a subset $S\subseteq L$ a \emph{sublocale} if it is closed under all meets, and stable under the operation $a\to -$ for all $a\in L$. Sublocales are inclusions which are maps in the category $\bd{Loc}$, which is described in \cite{picadopultr2012frames} as having frames as objects, and right adjoints of frame maps as morphisms.

\begin{example}\label{e: open and closed}
    For every frame $L$ and $a\in L$, the inclusion $\up a\subseteq L$ is a sublocale, as well as $\{a\to x\mid x\in L\}\subseteq L$.
\end{example}

\begin{example}\label{e: point induces a sublocale}
 For a frame $L$ and a prime element $p\in L$, the inclusion $\{p,1\}\subseteq L$ is a sublocale. In particular, $\{X{\setminus}\overline{\{x\}},X\}\subseteq \Om X$ is a sublocale for every space $X$ and $x\in X$. 
\end{example}

The collection $\ca S L$ of all sublocales of $L$, ordered under set inclusion, is a coframe. The assignment $a\mapsto \up a$ (see Example \ref{e: open and closed}) is a frame embedding $e_L^{\ca S}:L\to \ca S L^{op}$. Every element of the form $e^{\ca S}_L(a)$ has a complement, given by the sublocale $\{a\to x\mid x\in L\}$ from Example \ref{e: open and closed}. Furthermore, it is proven in \cite{Joyal1984} that the embedding $e^{\ca S}_L$ enjoys the universal property that any frame map $g:L\to M$ such that all elements of $g[L]$ are complemented factors through it uniquely. The proof of the fact is also based on the following result.

\begin{lemma}\label{l: how L generates SL}
    In the coframe $\ca S L$, every element is a meet of elements of the form $e^{\ca S}_L(a)\vee \neg e^{\ca S}_L(b)$ for $a,b\in L$.
\end{lemma}

We will also need some facts about how subspaces relate to sublocales. For a space $X$, we call $b(x)$ the sublocale $\{X{\setminus}\overline{\{x\}},X\}\subseteq \Om X$ in Example \ref{e: point induces a sublocale}. For a space $X$, and a subspace $Y\subseteq X$, we say that the sublocale
\[
\bigvee \{b(x)\mid x\in Y\}\subseteq \Om X
\]
is \emph{induced} by $Y$. A sublocale is called \emph{induced} if it is induced by some subspace. The following result may be taken as motivation for the definition of sublocales of the form $e^{\ca S}_L(a)$ as \emph{closed} and those of the form $\neg e^{\ca S}_L(a)$ as \emph{open}.

\begin{lemma}\label{l: open and closed means induced}
Let $X$ be a space, and $U\in \Om X$.
\begin{enumerate}
    \item The sublocale $e^{\ca S}_L(U)$ is induced by $U^c$;
    \item The sublocale $\neg e^{\ca S}_L(U)$ is induced by $U$.
\end{enumerate}
Therefore, all open and closed sublocales of spatial frames are induced.
\end{lemma}

Every frame map $f:L\to M$ lifts to a coframe map $\ca S f:\ca S L\to \ca S M$. This is defined on generators as 
\begin{align*}
\ca S f(e^{\ca S}_L(a))=e^{\ca S}_L(f(a))&&\ca S f(\neg e^{\ca S}_L(a))=\neg e^{\ca S}_L(f(a)).
\end{align*}
As it is a coframe map, $\ca S f$ has a left adjoint, which we call $\ca S f^*$. 

\subsection{Exactness and strong exactness}
Throughout this work, for a frame $L$ we will regard the collection $\mathsf{Filt}(L)$ of filters of $L$ as ordered under reverse set inclusion. By analogy with $\bd{Loc}$, we call $\bd{CoLoc}$ the category of coframes equipped with left adjoints of coframe morphisms, and define subcolocales dually as sublocales. A meet $\bwe_i x_i\in L$ of a frame $L$ is called \emph{exact} if for all $y\in L$
\[
\bwe_i (x_i\vee y)=\bwe_i x_i \vee y.
\]
Filters closed under exact meets are called exact; we call their collection $\fe(L)$. The collection of exact filter has the following useful characterization (Lemma 3.7 of \cite{suarez25raney}).
\begin{lemma}\label{l: fe is smallest sl containing the principals}
    For a frame $L$, $\fe(L)\subseteq \mf{Filt}(L)$ is the smallest subcolocale containing all principal filters.
\end{lemma}

We will call a frame map $f:L\to M$ \emph{exact} if $f^{-1}(F)\in \fe(L)$ whenever $F\in \fe(M)$. For a more explicit definition and the proof of its equivalence with ours, we refer to Proposition 6.6 of \cite{suarez25raney}. We call $\bd{Frm}_{\ca{E}}$ the category of frames and exact frame maps. A frame map $f:L\to M$ is exact if and only if it extends to a coframe map $\fe(f):\fe(L)\to \fe(M)$. The notion of exactness is often studied with and compared to that of \emph{strong exactness}. A family $S$ of elements of a frame $L$ is called \emph{strongly exact} if whenever $y\in L$ and $s\to y=y$ for all $s\in S$, then also $\bwe S\to y=y$. Strongly exact meets are introduced in \cite{Wilson1994TheAT}, where they are called \emph{free} meets, and their description as those meets that are preserved by every frame map is used. Filters that are closed under strongly exact meets are called \emph{strongly exact}, and their collection, again ordered under reverse set inclusion, is denoted as $\fse(L)$. The following results are is Corollary 2.7 of \cite{ball20} and Corollary 3.6 of \cite{moshier20}.
\begin{theorem}
    For a frame $L$, there is a chain of subcolocale inclusions 
    \[
    \fe(L)\subseteq \fse(L)\subseteq \mathsf{Filt}(L).
    \]
\end{theorem}

\subsection{Exactness and essentiality}\label{ss: essentiality}

Exactness is the specialization to frames of the more general notion of \emph{admissibility}, i. e. distributivity over binary joins of a meet in a join-semilattice, with admissible joins defined dually. Admissibility is strongly intertwined with essentiality in categories of lattices, see for example \cite{BrunsLakser1970}, or \cite{ball19} for a specialization of these and related results in the context of frames. There, maximal essential extensions for frames are characterized; in particular, the maximal essential extension of a frame is proven to be a structure introduced by Funayama in \cite{funayama59} for a general lattice. For a frame $L$, the \emph{Funayama embedding} $e^{\ca F}_L:L\to \ca F L$ can be characterized as the embedding of $L$ into the MacNeille completion of its Boolean envelope; this characterization is proven in \cite{gr2003general} (Section II.4).  

The Funayama construction can also be realized, as done in \cite{johnstone82} (Section II.2), as a subcolocale $\ca F L\subseteq \ca S L$, the collection of all sublocales of $L$ which are joins of elements of the form $e^{\ca S}_L(a)\wedge \neg e^{\ca S}_L(b)$ for some $a,b\in L$ (in fact, in \cite{ball19}, this description is used). The two structures in \cite{gr2003general} and \cite{johnstone82} are shown to be isomorphic in \cite{funayamasrevisited}. Observe the parallel between the following result and Lemma \ref{l: fe is smallest sl containing the principals}. 

\begin{lemma}\label{l: funayama smallest subloc}
    For a frame $L$, $\ca F L\subseteq \ca S L $ is the smallest subcolocale containing $e^{\ca S}_L[L]$.
\end{lemma}
\begin{proof}
We denote reverse set inclusion as $\sqsubseteq$. Suppose that $\ca T\subseteq \ca S L$ is a subcolocale. For $a,b\in L$, in the frame $(\ca S L,\sqsubseteq)$ the Heyting implication $e^{\ca S}_L(a)\to e^{\ca S}_L(b)$ is evaluated as $\neg e^{\ca S}_L(a)\sqcup e^{\ca S}_L(b)$. Then, every element of this form is in $\ca T$, and sublocales are closed under arbitrary meets, $\ca F L\subseteq \ca T$. 
\end{proof}
 In \cite{moshier20}, the map 
\begin{align*}
\mathit{ker}_L:&\ca S L\to \mf{Filt}(L)\\
& S\mapsto \{x\in L\mid S\leq \neg e^{\ca S}_L(x)\}
\end{align*}
is defined. 
\begin{lemma}\label{l: the properties of ker}
For every frame $L$, the map $\mathit{ker}_L:\ca S L\to \mf{Filt}(L)$ is in $\bd{CoLoc}$. Its image concides with the strongly exact filters.    
\end{lemma}

\begin{example}\label{e: kernel of an open sublocale}
For every $a\in L$,
\begin{itemize}
    \item $\mathit{ker}_L(\neg e^{\ca S}_L(a))=\up a$.
    \item $\mathit{ker}_L(e^{\ca S}_L(a))=\{x\in L\mid x\vee a=1\}$.
\end{itemize}    
\end{example}
The following is Theorem 6.6 in \cite{JS25}.
\begin{lemma}\label{l: ker of SbL is fe}
    For a frame $L$, $\mathit{ker}_L[\ca F L]=\fe(L)$.
\end{lemma}

We call \emph{locally exact} those frame maps $f:L\to M$ such that they extend to maps of complete Boolean algebras $\ca F f:\ca F L\to \ca F M$. These maps were recently characterized explicitly in \cite{arrieta26}. However, as in this work we do not need their explicit description, we take the property of extending to the Funayama construction as a definition. We call $\bd{Frm}_{\ca{LE}}$ the category of frames and locally exact frame maps. The following is observed, for example, in Example 4.1 in \cite{arrieta26}). We give a proof based on our definitions.
\begin{lemma}\label{l: fe is subcoframe of F}
Every locally exact map is exact.
\end{lemma}
\begin{proof}
By Lemma \ref{l: ker of SbL is fe}, the restriction and co-restriction of $\mathit{ker}_L$, $\mathit{ker}'_L{}:\ca FL \to \fe(L)$, is a surjective map in $\bd{CoLoc}$. This means that its right adjoint is a subcoframe embedding of $\fe(L)$ into $\ca F$ (restricting as the identity on $L$), and so any frame map extending to $\ca F(-)$ also extends to $\fe(-)$. 
\end{proof}

Finally, we argue that $T_D$-spaces, first introduced in \cite{Aull62}, are topological counterparts of essential extensions of frames. Concretely, a space is
$T_D$ if and only if every point is the intersection of an open and a closed set; alternatively a space is $T_D$ when the topology on $X$ generated by the opens and their complement, called the \emph{Skula topology} and denoted by $\ca{SK}X$, is discrete. For a topological space $X$, $\ca U X$ will denote the collection of saturated sets, i. e. the intersections of open sets, ordered under set inclusion\footnote{The notation is motivated by the fact that saturated sets coincide with upper sets in the specialization preorder}.

\begin{theorem}\label{t: many chactacterizations of td spaces}
    The following are equivalent for a space $X$.
    \begin{enumerate}
        \item The space $X$ is $T_D$.
        \item The embedding $\Om X\subseteq \ca P X$ is, up to isomorphism, the Funayama embedding $e^{\ca F}_{\Om X}:\Om X\to \ca F \Om X$;

        \item The embedding $\Om X\subseteq \ca U X$ is, up to isomorphism, $\up:\Om X\to\fe(\Om X)$;

    \end{enumerate}
\end{theorem}
\begin{proof}
The equivalence between 1 and 2 is in 3.2 of \cite{ball19}, and that between 1 and 4 is in Theorem 5.7 of \cite{suarez25raney}.
\end{proof}

\subsection{Raney extensions}

In \cite{suarez25raney}, the category $\bd{Raney}$ of Raney extensions is introduced. A Raney extension can be identified, up to isomorphism, with a pair $(L,\ca F)$ where $L$ is a frame and $\ca F$ is a subcolocale of $\mathsf{Filt}(L)$ such that $\fe(L)\subseteq \ca F\subseteq \fse(L)$. These form the category $\bd{Raney}$, where a morphism $f:(L,\ca F)\to (M,\ca G)$ is a frame map $f:L\to M$ satisfying $f^{-1}(G)\in \ca F$ whenever $G\in \ca G$. The obvious forgetful functor $\Op_{\ca R}:\bd{Raney}\to \Frm$ is thus faithful, and so $(\Op_{\ca R},\bd{Raney})$ is a $\Frm$-concrete category. 
\begin{example}\label{e: re of strongly exact}
$(L,\fse(L))$ is a Raney extension for every frame $L$.
\end{example}
\begin{example}\label{e: re of exact}
 $(L,\fe(L))$ is a Raney extension for every frame $L$.
\end{example}

The assignment $L\mapsto \fse(L)$ from Example \ref{e: re of strongly exact} can be extended to a functor $\fse:\Frm\to \bd{Raney}$: using the characterization of strongly exact meets as those preserved by any frame morphism, it is easy to show that for every frame map preimages of strongly exact filters are strongly exact. The following is Theorem 6.3 in \cite{suarez25raney}.
\begin{lemma}\label{l: fse is left adjoint}
    The functor $\fse:\Frm\to \bf{Raney}$ is left adjoint to $\Or:\bd{Raney}\to \Frm$.
\end{lemma}
On the other hand, the assignment in Example \ref{e: re of exact} is not functorial (the maps that lift coincide with the exact ones, as seen in the preliminaries). We now describe the natural adjunction between Raney extensions and spaces. Because the embedding $e_X:\Om X\subseteq \ca U X$ is a map of distributive lattices with a coframe as the codomain, by the universal property of the ideal completion there is a coframe map $\overline{e_X}:\mf{Filt}(\Om X)\to \ca U X$, which then has a left adjoint $\overline{e_X}^*:\ca U X\to \mf{Filt}(\Om X)$, whose image we call $\ca U^* X$. All elements of $\ca U X$ are fixpoints of this adjunction, and so we get an order-isomorphism $\overline{e_X}|_{\ca U^* X} \overline{e_X}^*: \ca U X\to \ca U^*X$, explicitly given by
\[
S\mapsto \{U\in \Om X\mid S\subseteq U\}.
\]

\begin{example}\label{e: the raney extension UX}
 The pair $(\Om X,\ca U^*X)$ is a Raney extension for every space $X$. 
\end{example}
There is a contravariant functor $\ca U:\Topp\to \bd{Raney}$ mapping each space $X$ to the Raney extension $(\Om X,\ca U^*X)$ from Example \ref{e: the raney extension UX}. The functor $\ca U$ has an adjoint $\pt_{\ca R}:\bd{Raney}\to \Topp$ mapping a Raney extension $(L,\ca F)$ to $\bd{Raney}((L,\ca F),\bd 2_{\ca R})$, equipped with the subspace topology inherited from $\pt(L)$, that is, the one whose opens are sets of the form 
\[
\{f\in \bd{Raney}((L,\ca F),\bd 2_{\ca R})\mid \Op_{\ca S}f(a)=1\}
\]
for some $a\in L$.

\begin{theorem}
There is a dual adjunction $\ca{U}:\bd{Top}\leftrightarrows \bd{Raney}:\pt_{\ca R}$. The fixpoints in $\bd{Top}$ are the $T_0$-spaces.     
\end{theorem}

\subsection{Skula extensions}

In \cite{manuell15}, another category giving a pointfree description of $T_0$-spaces is introduced. This is the category of \emph{strictly zero-dimensional biframes}. We work with an equivalent category, whose objects we call \emph{Skula extensions}. The category of Skula extensions bears many analogies with the category of Raney extensions. A \emph{Skula extension} is a pair $(L,\ca D)$, where $L$ is a frame, and $\ca D$ is a subcolocale of $\ca S L$ such that $\ca F L\subseteq \ca D$. Skula extensions form the category $\bd{Skula}$, where a morphism $f:(L,\ca D)\to (M,\ca E)$ is a frame map $f:L\to M$ such that $\ca S f^*(E)\in \ca D$ whenever $E\in \ca E$. This is equivalent to $f$ extending to a coframe map $\ca D\to \ca E$, but this characterization is less useful for our ends. We immediately obtain the following examples.

\begin{example}\label{e: frame of congruences}
    $(L,\ca S L)$ is a Skula extension for every frame $L$.
\end{example}
\begin{example}\label{e: funayama}
   $(L,\ca F L)$ is a Skula extension for every frame $L$.
\end{example}
The assignment in Example \ref{e: frame of congruences} can be extended to a functor $\ca S:\Frm\to \bd{Skula}$. The following is Proposition 3.4 in \cite{manuell15}.

\begin{lemma}\label{l: congruence frame gives left adjoint}
    The functor $\ca S:\Frm\to \bd{Skula}$ is left adjoint to $\Op_{\ca S}:\bd{Skula}\to \Frm$.
\end{lemma}

On the other hand, the assignment in Example \ref{e: funayama} is not functorial; the maps that do lift are exactly the locally exact maps, as seen in the preliminaries.

For any topological space $X$, we let $\ca{SKC}X$ be the coframe of Skula-closed sets. The frame embedding $\Om X\subseteq \ca{SKC}X^{op}$ provides complements to all elements of $\Om X$, and so by the universal property of $\ca S L$ there is a coframe map $\ca S \Om X\to \ca{SKC}X$ extending the identity on $\Om X$, which is additionally a surjection by Lemma \ref{l: how L generates SL}. We call the corresponding subcolocale inclusion $\ca{SKC}^*X\subseteq \ca S \Om X$. Explicitly,
\[
\ca{SKC}^*X=\{\bigvee \{b(x)\mid x\in U\}\mid U\in \ca{SKC}X\}.
\]
Since $\ca{SKC}^*X\subseteq \ca S \Om X$ is an embedding, this collection is isomorphic to the collection $\ca{SKC}X$ of Skula-closed sets.
\begin{example}\label{e: skula of a space}
For a space $X$, the pair $(\Om X,\ca{SKC}^*X)$ is a Skula extension. To see this, we check that $\ca F \Om X\subseteq \ca{SKC}^*X$. By Lemma \ref{l: funayama smallest subloc}, for this it suffices to show that $e^{\ca S}_{\Om X}(U)\in \ca{SKC}^*X$ for all $U\in \Om X$. But, as $U$ is open, the sublocale $e^{\ca S}_{\Om X}(U)$ is induced (see Lemma \ref{l: open and closed means induced}), and the desired result follows.
\end{example}

What follows is in \cite{manuell15}, in particular the main duality result is Proposition 3.33. The assignment $X\mapsto (\Om X,\ca{SKC}^*X)$ is the object part of a contravariant functor $\ca{SK}:\Topp\to \bd{Skula}$. The adjoint $\pt_{\ca S}:\bd{Skula}\to \Topp$ maps a Skula extension $(L,\ca D)$ to $\bd{Skula}((L,\ca D),\bd 2_{\ca S})$, equipped with the subspace topology inherited from $\pt(L)$.

\begin{theorem}
There is a dual adjunction $\ca{SK}:\bd{Top}\leftrightarrows \bd{Skula}:\pt_{\ca S}$. The fixpoints in $\bd{Top}$ are the $T_0$-spaces.     
\end{theorem}

\subsection{McKinsey-Tarski algebras}
In \cite{bezhanishvili23}, the category $\bd{MT}$ of \emph{McKinsey-Tarski algebras}, or \emph{MT-algebras}, is introduced. These may be identified with pairs $(L,M)$ where $M$ is a complete Boolean algebra and $L$ is any subframe. Elements of $L$ are called \emph{open} elements. Morphisms between MT-algebras are maps of complete Boolean algebras that restrict to the subframes of open elements. 

\begin{example}\label{e: PX}
    For a space $X$, the pair $(\Om X,\ca P X)$ is an MT-algebra.
\end{example}
An MT-algebra is said to be $T_0$ (as defined in \cite{bezhanishvili23}) if every element of $M$ is a join of elements of the form $\bigwedge _i a_i\wedge \neg a$, with $a,a_i\in L$. We will assume all MT-algebras to be $T_0$, and in particular $\bd{MT}$ will denote the category of $T_0$-algebras. In particular (Theorem 5.8 in \cite{bezhanishvili23}), the functor $\Op_{\ca M}:\bd{MT}\to \Frm$ is faithful. So, $(\Op_{\ca M},\bd{MT})$ is an FC-category.

The assignment in Example \ref{e: PX} extends to a contravariant functor $\ca P:\bd{Top}\to \bd{MT}$, mapping each space to the MT-algebra $(\Om X,\ca P X)$. This has an adjoint $\pt_{\ca M}:\bd{MT}\to \Topp$, mapping an MT-algebra $(L,M)$ to the set $\pt_{\ca M}M$ of atoms of $M$, suitably topologized. The fixpoints of the adjunction are all $T_0$-spaces on the $\bd{Top}$ side (Theorem 3.17 in \cite{bezhanishvili23}). 

\begin{theorem}
There is a dual adjunction $\ca P:\bd{Top}\leftrightarrows \bd{MT}:\pt_{\ca M}$. The fixpoints in $\bd{Top}$ are the $T_0$-spaces.     
\end{theorem}

\section{SFC-categories}
\subsection{Motivating examples}\label{ss motivating examples}
The classical dual adjunction between frames and spaces
\begin{equation*}
  \begin{tikzcd}[column sep=large]
    \bd{Top} \ar[r,shift left=3,"\Om"]
             \ar[r,"\dashv"{rotate=-90},phantom]
    & \bd{Frm}^\bd{op} \ar[l,shift left=3,"\pt"]
  \end{tikzcd}
\end{equation*}
is an archetypal example of a \textit{natural} dual adjunction (in the sense
of \cite{porst91}), which is obtained via the dualizing object construction
due to Porst and Tholen. Its dualizing pair is $(\bd 2,\mathbb{S})$, where
$\bd 2$ is the two-element frame and $\mathbb{S}$ is the Sierpiński space.

Another natural dual adjunction between frames and spaces, where the fixpoints in $\bd{Top}$ are a category incomparable with $\bd{Sob}$, is studied in \cite{banaschewskitd}. A map $f:L\to M$ of frames is a \emph{D-morphism} if $f^{-1}(F)\in \fe(L)$ whenever $F\in \fcp(M)\cap \fe(M)$\footnote{Being a D-morphism, then, may be seen as a spatialized version of exactness.}. We call $\bd{Frm_D}$ the category of frames with D-morphisms, and we call $\bd{Top_D}$ the category of $T_D$-spaces and continuous maps. The duality result in \cite{banaschewskitd} can be phrased as follows.
\begin{theorem}
There is a dual adjunction $\Om:\bd{Top_D}\leftrightarrows \bd{Frm_D}:\pt_D$, coming from the dualizing object $(\bd 2,\mathbb{S})$, with $T_D$-spaces the fixpoints on the side of topological spaces.     
\end{theorem}

We note that the two natural adjunctions in this subsection are
related by a right map of adjunctions, in the sense of Subsection \ref{ss:
maps of adjunctions}, as shown below:

\begin{equation}\label{d: left map of adjunctions for TD}
    \begin{tikzcd}[row sep=huge,column sep=large]
      \bd{Frm_D}
            \ar[r,shift left=2,"\pt_D"{above}] 
            \ar[d,hookrightarrow]
        & \bd{Top_D}^{\mathsf{op}}
            \ar[l,shift left=2,"\Om"{below}]
            \ar[l,phantom,"\dashv"{rotate=-90,anchor=center}]
            \ar[d,hookrightarrow] \\
        \bd{Frm}
            \ar[r,shift left=2,"\pt"{above}]
        & \bd{Top}^{\mathsf{op}}
            \ar[l,shift left=2,"\Om"{below}]
            \ar[l,phantom,"\dashv"{rotate=-90,anchor=center}]
   \end{tikzcd}
\end{equation}

The natural transformations witnessing the two inclusions being a right map of
adjunctions, componentwise, are:
\begin{itemize}
    \item The identity $\Om X=\Om X$ for every $T_D$-space $X$;
    \item The assignment
    \begin{align*}\label{e: faithfulness becomes subspace inclusion}
       \xi_L^D:& \bd{Frm_D}(L,\bd 2)\to \bd{Frm}(L,\bd 2) \\
       & f\mapsto i_D(f)
    \end{align*}
    for every frame $L$.
\end{itemize}
If we replace $i_D:\bd{Frm_D}\subseteq \bd{Frm}$ with each of $\Op_{\ca R}:\bd{Raney}\to \Frm$, $\Op_{\ca S}:\bd{Skula}\to \Frm$ and $\Op_{\ca M}:\bd{MT}\to \Frm$, we witness a similar phenomenon. In all three cases, there are maps in $\bd{RADJ}$:
    \begin{equation}\label{d: maps of adjunctions for the three main}
\begin{minipage}{0.3\textwidth}
    \begin{tikzcd}[row sep=huge,column sep=large]
        \bd{Raney}
            \ar[r,shift left=2,"\pt_{\ca R}"{above}] \ar[d,"\Op_{\ca R}"]
        & \bd{Top}^{\mathsf{op}}
            \ar[l,shift left=2,"\ca U"{below}]
            \ar[l,phantom,"\dashv"{rotate=-90,anchor=center}]
            \ar[d,equal] \\
        \bd{Frm}
            \ar[r,shift left=2,"\pt"{above}]
        & \bd{Top}^{\mathsf{op}}
            \ar[l,shift left=2,"\Om"{below}]
            \ar[l,phantom,"\dashv"{rotate=-90,anchor=center}]
   \end{tikzcd}
\end{minipage}
\hfill
\begin{minipage}{0.3\textwidth}
    \begin{tikzcd}[row sep=huge,column sep=large]
              \bd{Skula}
            \ar[r,shift left=2,"\pt_{\ca S}"{above}] \ar[d,"\Op_{\ca S}"]
        & \bd{Top}^{\mathsf{op}}
            \ar[l,shift left=2,"\ca{SK}"{below}]
            \ar[l,phantom,"\dashv"{rotate=-90,anchor=center}]
            \ar[d,equal] \\
        \bd{Frm}
            \ar[r,shift left=2,"\pt"{above}]
        & \bd{Top}^{\mathsf{op}}
            \ar[l,shift left=2,"\Om"{below}]
            \ar[l,phantom,"\dashv"{rotate=-90,anchor=center}]
   \end{tikzcd}
\end{minipage}
\hfill
\begin{minipage}{0.3\textwidth}
    \begin{tikzcd}[row sep=huge,column sep=large]
                \bd{MT}
            \ar[r,shift left=2,"\pt_{\ca M}"{above}] \ar[d,"\Op_{\ca M}"]
        & \bd{Top}^{\mathsf{op}}
            \ar[l,shift left=2,"\ca{P}"{below}]
            \ar[l,phantom,"\dashv"{rotate=-90,anchor=center}]
            \ar[d,equal] \\
        \bd{Frm}
            \ar[r,shift left=2,"\pt"{above}]
        & \bd{Top}^{\mathsf{op}}
            \ar[l,shift left=2,"\Om"{below}]
            \ar[l,phantom,"\dashv"{rotate=-90,anchor=center}]
   \end{tikzcd}
\end{minipage}
\end{equation}

In all three cases, the natural transformations for each direction of the diagrams are given similarly as those for $i_D:\bd{Frm_D}\subseteq \Frm$. For each of the three $\Frm$-concrete categories $\Op:\ca C\to \Frm$, we let $\Omc:\ca C\leftrightarrows \Topp:\ptc$ be the upper adjunction in the diagrams.
\begin{itemize}
    \item From the definitions, it is immediate there is an equality $\Op \Omc X=\Om X$.
    \item As the underlying set of $\ptc (C)$ for $C\in \ca C$ is given by $\ca C(C,\Tc)$ for some $\Tc\in \ca C$ with $\Op \Tc= \bd 2$, and from the definition of its topology, there is a subspace inclusion given by:
\begin{align*}
   \xi^{\ca C}_C:&\ptc C\to \pt \Op C \\
   & f\mapsto  \Op f.
\end{align*}
\end{itemize}
As our last example, we consider the $\Frm$-concrete category $\Om:\Topz^{\mf{op}}\to \Frm$. There is a morphism in $\bd{RADJ}$ as below, where $r_0:\Topp\to \Topz$ denotes the $T_0$-reflection functor.

\begin{equation}\label{d: diagram for t0}
     \begin{tikzcd}[row sep=huge,column sep=large]
        \Topz^{\mf{op}}
            \ar[r,shift left=2,hookrightarrow] 
            \ar[d,"\Om"]
        & \bd{Top}^{\mathsf{op}}
            \ar[l,shift left=2,"r_0"{below}]
            \ar[l,phantom,"\dashv"{rotate=-90,anchor=center}]
            \ar[d,equal] \\
        \bd{Frm}
            \ar[r,shift left=2,"\pt"{above}]
        & \bd{Top}^{\mathsf{op}}
            \ar[l,shift left=2,"\Om"{below}]
            \ar[l,phantom,"\dashv"{rotate=-90,anchor=center}]
   \end{tikzcd}
\end{equation}

The natural transformations exhibiting this as a right map of adjunctions are:
\begin{itemize}
    \item The obvious isomorphism $\Om r_0 C\cong \Om X$ for every space $X$;
    \item The sobrification $\sigma_X:X\to \pt \Om X$ for every $T_0$ space $X$. Viewing the inclusion $\Topz\subseteq \Topp$ as a functor $\ptc$ with $|\ptc(X)|=\Topz(\{*\},X)$ this is described as:
    \begin{align*}
   \xi^{0}_X:& \ptc(X)\to \pt \Om X \\
   & f\mapsto  \Om f.
\end{align*}
\end{itemize}

\begin{remark}
    The functor $\Om:\Topp^{\mf{op}}\to \Frm$ is not faithful, while the restriction $\Om:\Topz^{\mf{op}}\to \Frm$ is. This provides some initial motivation for viewing $\Frm$-concrete categories as categories of pointfree $T_0$-spaces.
\end{remark}

These analogies suggest that the development of a general setting where we can
describe natural adjunctions between $\Frm$-concrete categories may prove
fruitful to study dualities between full subcategories of \( \Topp \) and
suitable categories of pointfree spaces.

\subsection{Definition and basic results}\label{ss: defn topc ptc omc}

Throughout, we fix an FC-category \( \Op \colon \ca C \to \Frm \), an object
\( \Tc \in \ca C \) and an isomorphism \( \iota \colon \Op(\Tc) \cong \bd 2 \). We will take the carrier set of an object \( C \in \ca C \) to be the set of
elements of the frame $\Op C$, denoted by $|\Op C|$, while we regard the
carrier of a space \( X \) in \( \Topp \) to be the set of its points \( |X|
\). 

We define $\Topc$ as the full subcategory of $\Topp$ consisting of those
spaces $X$ such that the family obtained by composing \( |\iota^{-1}| \) with
the characteristic functions of their points 
\begin{equation}\label{e: characteristic functions}
  \{ |\iota^{-1}| \circ \chi_{x}:\Topp(X,\mathbb S)\to |\Op(\Tc)| \mid x\in X\}
\end{equation}
has a cartesian lift \( (\Omc(X),\theta_X \colon |\Op\Omc(X)| \cong
\Topp(X,\mathbb S) ) \) in $\ca C$, and we denote the lifted family by
\begin{equation}
  \{\overline{\chi_{x}}: \Omc X \to \Tc \mid x\in X\}.
\end{equation}
We note that \( \Topc \) is a replete subcategory: if \( X \in \Topc \), and
if there is a homeomorphism \( X \cong Y \), then \( Y \in \Topc \). 

\begin{lemma}
  \label{l: pfc-cat implies zeta iso}
  Let \( (\ca C, \Op) \) be an FC-category, let \( \Tc \in \ca C \) be an
  object, and let \( \iota \colon \Op(\Tc) \cong \bd 2 \).be an isomorphism.
  If we take \( \Topc \) to be the category defined above, then \( \Omc \colon
  \opp{\Topc} \to \ca C \) defines a functor, and we have a natural
  isomorphism \( \zeta \colon \Op \Omc \cong \Om \).
\end{lemma}
\begin{proof}
  We take the pair \( (\Omc(X), \theta_X) \) to be the cartesian lift of the
  characteristic functions for each \( X \in \Topc \). By Remark~\ref{r:
  morphism-lifting}, we may let \( \Omc(f) \colon \Omc(Y) \to \Omc(X)  \) be
  the morphism induced by \( - \circ f \colon \Topp(Y,\mathbb S) \to
  \Topp(X,\mathbb S) \) for \( f \colon X \to Y \). It is routine to ensure
  that \( \Omc \) is functorial.

  Now, we observe that \( \theta_X \colon |\Op \Omc(X)| \to |\Om(X)| \) lifts
  to a frame morphism \( \zeta_X \colon \Op\Omc(X) \to \Om(X) \), since we
  have that \( \chi_x \circ \theta_X \) lifts to \( \iota \circ \Op(\overline
  \chi_x) \) for all \( x \). Frame maps whose underlying functions are
  bijective must be frame isomorphisms, so \( \zeta_X \) is an isomorphism for
  all \( X \). Naturality of \( \zeta \) is an immediate consequence of
  faithfulness of \( \Op \).
\end{proof}

For each object $C\in \ca C$, we define $\ptc(C)$ to be the space with carrier
set $\ca{C}(C,\Tc)$, and whose opens are sets of the form
\[
\{f\in \ca{C}(C,\bd{2}_{\ca{C}})\mid \Op f(a)=1\}
\]
for some $a\in \Op C$. We note that this is precisely the initial topology on
\( \ca C(C,\Tc) \) with respect to the evaluation maps $\{\eta_a:\ca
C(C,\Tc)\to 2\mid a\in \Op C\}$, defined by $\eta_a(f)=(\iota \circ \Op
f)(a)$.  Thus, \( (\ptc(C), \id) \) is a cartesian lift of this family.

We say that the quadruple $(\Op,\ca C,\Tc, \iota)$ is a \emph{spatializable
FC-category}, abbreviated as \emph{SFC-category}, if $\mathbb{S}\in \Topc$ and
$\ptc(C) \in \Topc$ for all \( C \in \ca C \). Thus, for an SFC-category, the
triple \( (\Tc, \mathbb S, |\iota| \colon |\Op(C)| \cong |\bf 2| = |\mathbb
S|) \) defines a dualizing object. As an immediate consequence of
Theorem~\ref{t: natduality}, we obtain the following:

\begin{theorem}
  \label{t: main theorem in porst tholen adapted}
  If $(\Op,\ca C,\Tc,\iota)$ is an SFC-category, then we have a natural dual
  adjunction
  \begin{equation}
    \label{e: dualadjC}
    \begin{tikzcd}[column sep=large]
      \ca C  \ar[r, "\ptc", shift left=2]
                  \ar[r,"\dashv"{rotate=-90},phantom]
        & \opp{\Topc} \ar[l, "\Omc", shift left=2]
    \end{tikzcd}
  \end{equation}
  obtained from the dualizing object $(\Tc,\mathbb{S},|i|)$.
\end{theorem}

We denote by $\mf{Dual}(\Op,\ca C,\Tc,\iota)$ the natural adjunction \eqref{e:
dualadjC} induced by $(\Op,\ca C,\Tc,\iota)$, and we denote the full
subcategory of its fixpoints in \( \Topc \) and in \( \ca C \) by
$\mf{Fix}(\Topc)$ and $\mf{Fix}(\ca D)$, respectively.

Explicit descriptions of the data of the natural adjunction in Theorem \ref{t:
main theorem in porst tholen adapted} are given below.
\begin{enumerate}[label=(NA\arabic*),ref=NA\arabic*]
  \item 
    \label{NA1} 
    The functor $\ptc:\ca C\to \opp{\Topc}$ acts on objects as $C\mapsto \ptc
    (C)$. On morphisms, it sends a map $f:C\to D$ to precomposition $-{\circ}
    f$.
  \item 
    \label{NA2} 
    The functor $\Omc:\Topc\to \ca C$ assigns each $X\in \Topc$ to the object
    $\Omc X\in \ca C$ which provides the cartesian lift of the family of
    characteristic functions. On morphisms, it sends a map $f:X\to Y$ to the
    lift of $-{\circ} f:\Topp(Y,\mathbb{S})\to \Topp(X,\mathbb{S})$.
  \item 
    \label{NA3}
    For $X\in \Topc$, the unit $\sigma^{\ca C}:X\to \ptc \Omc X$ maps $x\in X$
    to the lift $\overline{\chi_x}:\Omc X\to \Tc$ of its characteristic
    function. 

  \item 
    \label{NA4} 
    For $C\in \ca C$, as $\Op$ is faithful the unit $\eta^{\ca C}_C:C\to \Omc
    \ptc C$ is completely determined by $\Op (\eta^{\ca C}_{C})$.  This map
    takes $a\in \Op C$ to the map $\ptc C\to \mathbb{S}$ corresponding to the
    open
    \begin{equation}
      \label{e: stone like map}
      \{f\in \ptc C\mid \Op f(a)=1\}
    \end{equation}
    The topology of $\ptc C$ may then be described by treating $\Op (\eta^{\ca
    C}_C)$ as a Stone-like map, in the sense that the opens are the sets in
    the image of $\Op (\eta^{\ca C}_C)':\Op C\to \mathcal{P}(\ptc C)$, where
    $\Op (\eta^{\ca C}_C)'$ sends each $a\in \Op C$ to the open \eqref{e:
    stone like map}.
\end{enumerate}

Let us look at simple examples of the construction from Theorem \ref{t: main
theorem in porst tholen adapted}.

\begin{example}\label{e: trivial example}
  As a trivial example of SFC-category we have the quadruple
  $(\mathds{1}_{\Frm},\Frm,\bd 2,\id)$, which induces the classical
  contravariant adjunction between frames and spaces. Here, $\Topc=\Topp$, and
  $\mf{Fix}(\Topc)=\bd{Sob}$.
\end{example}

\begin{example}\label{e: td duality}
  In \cite{banaschewskitd}, the adjunction for $T_D$-spaces arises as in
  Theorem \ref{t: main theorem in porst tholen adapted} from the SFC-category
  $(i_D,\bd{Frm_D},\bd 2,\id)$, where $i_D$ is the subcategory inclusion
  $\bd{Frm_D}\subseteq \Frm$. The category $\Topc$ is the category of
  $T_D$-spaces. In fact, these are characterized as the spaces where every
  characteristic map $\widehat{\chi_x}:\Om X\to \bd 2$ is a morphism in
  $\bd{Frm_D}$. The $T_D$-spaces are also the fixpoints of the resulting
  adjunction. In this case, $\Topc=\mf{Fix}(\Topc)$.
\end{example}

\begin{example}\label{e: t0 spaces example}
  The quadruple $(\Om,\Topz^{\mathsf{op}},\{*\},\id)$, where $\{*\}$ is the
  one-point space, is an SFC-category. Here, $\Topc=\Topp$ and \(
  \mf{Fix}(\Topc)=\Topz \).

  For every $X\in \Topp$, we claim that its $T_0$-reflection $X_0$ provides a
  cartesian lift of the family of characteristic functions $\{ \chi_x:|\Om
  X|\to \bd 2 \mid x \in X \} $, whose lifted family is given by the
  inclusions $ \{ i_{r(x)}:\{*\}\to X_0 \mid x \in X \}$ with
  $i_{r(x)}(*)=r(x)$, where \( r \colon X \to X_0 \) is the \( T_0
  \)-reflection map.

  To see this, we let $Y$ be a $T_0$ space, and we let $f:|\Om Y|\to |\Om X|$
  be a function such that every composite $\chi_x\circ f$ lifts to some map \(
  \overline{f_x}:\{*\}\to Y \).  The desired lift $\overline{f}:X_0\to Y$ of
  $f$ is defined as $\overline{f}(r(x))=\overline{f_r(x)}(*)$ for every $x\in
  X$; this is well-defined as $Y$ is $T_0$. The functors obtained as in
  Theorem \ref{t: main theorem in porst tholen adapted} are the inclusion
  $\Topz\subseteq\Topp$ on one side, and the $T_0$-reflection of a space on
  the other. This implies that $\mf{Fix}(\Topc)=\Topz$.
\end{example}

\begin{example}
  \label{e: the sober example}
  The quadruple $(\Om,\bd{Sob}^{\mathsf{op}},\{*\}, \id)$ is an SFC-category.
  First, we show that $\Topc=\Topp$. For a topological space $X$, the object
  providing the cartesian lift of the family $\{\chi_x:|\Om X|\to 2\mid x\in
  X\}$ is the sobrification $\mf{sob}(X)$ of $X$. Calling $s:X\to \mf{sob}(X)$
  the sobrification map, the lift of the characteristic function of $x\in X$
  is given by the inclusion $i_{s(x)}:\{*\}\to \mf{sob}(X)$ defined as
  $i_{s(x)}(*)=s(x)$. To show that this is a cartesian lift, let $Y$ be a
  sober space and suppose that there is a function $f:|\Om Y|\to |\Om X|$
  whose composite $\chi_x\circ f$ lifts for every $x\in X$. By the universal
  property of $\Om X$, there is a frame map $\widehat{f}:\Om Y\to \Om X$
  lifting $f$, the desired lift is then, up to isomorphism, $\pt
  \widehat{f}:\mf{sob}X\to \mf{sob}Y$. The induced natural adjunction is
\begin{equation*}
\begin{tikzcd}[column sep=large,row sep=large]
     \Topp
            \ar[r,shift left=2,"\mf{sob}"{above}] 
        & \bd{Sob}
            \ar[l,hookrightarrow,{below},shift left=2]
            \ar[l,phantom,"\dashv"{rotate=-90,anchor=center}]
\end{tikzcd}
\end{equation*}
Then, $\mf{Fix}(\Topc)=\bd{Sob}$.
\end{example}

Whenever we are in the presence of an SFC-category \( (\ca C, \Op, \Tc, \iota)
\), the situation can be depicted as in the following diagram, which generalizes the diagrams in \ref{d: left map of adjunctions for TD},  \ref{d: maps of adjunctions for the three main}, and \ref{d: diagram for t0}. 

\begin{equation}
  \label{d: the basic picture of a T0 duality}
    \begin{tikzcd}[row sep=huge,column sep=large]
        \ca{C}
            \ar[r,shift left=2,"\ptc"{above}] \ar[d,"\Op",swap]
        & \bd{Top}_{\ca C}^{\mathsf{op}}
            \ar[l,shift left=2,"\Omc"{below}]
            \ar[l,phantom,"\dashv"{rotate=-90,anchor=center}]
            \ar[d,hookrightarrow] \\
        \bd{Frm}
            \ar[r,shift left=2,"\pt"{above}]
        & \bd{Top}^{\mathsf{op}}
            \ar[l,shift left=2,"\Om"{below}]
            \ar[l,phantom,"\dashv"{rotate=-90,anchor=center}]
   \end{tikzcd}
\end{equation}

Our next aim is to show that the vertical morphisms above constitute a right
map of adjunctions. For every $X \in \Topc$ and $C\in \ca C$, we define
\begin{align*}
  \zeta_X:\Op \Omc X\to \Om X && \xi_C:\ptc C\to \pt\Op C 
\end{align*}
via their actions on the carriers:
\begin{align*}
  \theta_X \colon |\Op \Omc(X)| &\to |\Om(X)| = \Topp(X,\mathbb{S})
    & \ca C(C,\Tc)&\to \Frm(\Op C,\bd 2)
    \\
   a &\mapsto \theta_{X,a},
    & f &\mapsto \iota \circ \Op f 
\end{align*}
so that $\zeta_X$ and \( \xi_C \) are the lifts of the above functions,
respectively, coming from the cartesian liftng properties of $\Om X$ in the
category of frames and \( \pt\Op(C) \) in the category of spaces; indeed, for
the former, we note that \( \chi_x \circ \theta_X = |\iota \circ \Op
\Omc(\overline{\chi}_x)| \) for all \( x \in X \).

\begin{proposition}
  \label{p: the zeta isomorphism and xi mate}
  For an SFC-category $(\Op,\ca C,2_{\ca{C}},\iota)$, the data depicted in
  \eqref{d: the basic picture of a T0 duality}, together with the natural
  transformations \( \zeta, \xi \), defines a right adjunction morphism.
\end{proposition}
\begin{proof}
  We already confirmed that \( \zeta \colon \Omc \Op \to \Om \) is a natural
  isomorphism in Lemma~\ref{l: pfc-cat implies zeta iso}.

  It remains to confirm that \( \zeta \) and \( \xi \) constitute a mate-pair, 
  and by Lemma~\ref{l: mate-corresp}, it is enough to prove that
  \begin{equation}
    \label{e: one of the mate relations}
    \begin{tikzcd}
      & \ca C \ar[d,"\ptc"] \ar[ld,bend right,"\id"{name=A},swap] \\
      \ca C \ar[d,"\Op",swap]
        \ar[rd,Rightarrow,"\zeta",shorten=6mm]
        & \opp{\Topc} \ar[l,"\Omc",swap] \ar[d,hookrightarrow] \\
      \Frm
        & \opp{\Topp} \ar[l,"\Omega"] 
      \ar[from=A,to=2-2,"\eta^{\ca C}",Rightarrow,shorten=4.5mm]
    \end{tikzcd}
      =
    \begin{tikzcd}
      \ca C \ar[d,"\Op",swap]
            \ar[r,"\ptc"]
        & \opp{\Topc} \ar[d,hookrightarrow] \\
      \Frm \ar[r,"\pt"]
            \ar[ru,"\xi",Rightarrow,shorten=5mm]
            \ar[rd,bend right,"\id"{name=B}",swap]
        & \opp{\Topp} \ar[d,"\Omega"] \\
      & \Frm
      \ar[from=B,to=2-2,"\eta",Rightarrow,shorten=4mm]
    \end{tikzcd}
  \end{equation}
  holds. Indeed, we simply need to observe that both \( |\zeta_{\ptc(C)} \circ
  \Op(\eta^\ca C_C)| \) and \( |\Omega(\xi_C) \circ \eta_{\Op(C)}| \) must be
  the function
  \begin{align*}
    |\Op(C)| &\to \Topp(\ptc(C),\mathbb S) \\
    a &\mapsto (f \mapsto \iota \circ \Op(f)(a)),
  \end{align*}
  so the condition \eqref{e: one of the mate relations} holds by faithfulness
  of \( |-| \colon \Frm \to \Set \), concluding the proof.
\end{proof}

We highlight that since \( \zeta \) and \( \xi \) are mates, we also have the
following equality of pasting diagrams
\begin{equation}
  \label{e: the second of the mate relations}
  \begin{tikzcd}
    \ca C \ar[d,"\Op",swap]
      \ar[rd,Rightarrow,"\zeta",shorten=6mm]
      & \opp{\Topc} \ar[l,"\Omc",swap] \ar[d,hookrightarrow] \\
    \Frm \ar[d,"\pt",swap]
      & \opp{\Topp} \ar[l,"\Omega",swap] \ar[ld,bend left,"\id"{name=B}] \\
    \opp{\Topp}
    \ar[from=2-1,to=B,"\sigma",Rightarrow,shorten=4.5mm]
  \end{tikzcd}
    =
  \begin{tikzcd}
    \opp{\Topc} \ar[d,"\Omc",swap] \ar[rd,"\id"{name=A},bend left] \\
    \ca C \ar[d,"\Op",swap]
          \ar[r,"\ptc"]
      & \opp{\Topc} \ar[d,hookrightarrow] \\
    \Frm \ar[r,"\pt",swap]
          \ar[ru,"\xi",Rightarrow,shorten=5mm]
      & \opp{\Topp} 
    \ar[from=2-1,to=A,"\sigma^{\ca C}",Rightarrow,shorten=6mm]
  \end{tikzcd}
\end{equation}

\begin{lemma}
  \label{l: xi is a subspace embedding and basic facts about the units}
  Let $(\Op,\ca C,2_{\ca{C}},\iota)$ be an SFC-category. The following hold:
  \begin{enumerate}[label=(\alph*)]
    \item 
      \label{xi subspace emb} 
      $\xi_C:\ptc C\to \pt \Op C$ is a subspace embedding for all $C\in \ca C$.
    \item
      \label{basic unit 1}
      $\Op(\eta^{\ca C}_C)$ is a frame surjection for all $C\in \ca C$.
  \end{enumerate}
\end{lemma}
\begin{proof}
  By faithfulness of $\Op$, the map $\xi_C:\ca C(C,\Tc)\to \Frm(\Op C,\bd 2)$
  is an injection. The topology on $\ca C(C,\Tc)$, as seen in \eqref{NA4}, is
  the initial one with respect to this injection, and so $\xi_C$ is a subspace
  embedding, confirming \ref{xi subspace emb}.

  Moreover, it follows that \( \Omega(\xi_C) \) must be a frame surjection for
  all \( C \in \ca C \), as is \( \eta_{\Op(C)} \colon \Op(C) \to \Omega \pt
  \Op(C) \). Then, by \eqref{e: one of the mate relations}, we have that
  \begin{equation*}
    \zeta_{\ptc(C)} \circ \Op(\eta^{\ca C}_C)
      = \Omega(\xi_C) \circ \eta_{\Op(C)} 
  \end{equation*}
  is a frame surjection. Since \( \zeta \) is a natural isomorphism, we
  conclude that \ref{basic unit 1} must hold.
\end{proof}

Natural dual adjunctions induced by SFC-categories are idempotent.
Idempotence ascends from the base adjunction between frames and spaces: in the
proof below, we note the use of Item (a) of Lemma \ref{l: xi is a subspace
embedding and basic facts about the units}, as well as the idempotence of
$\Om\dashv \pt$.

\begin{theorem}
  \label{t: idempotence}
  If $(\Op,\ca C,\Tc,\iota)$ is an SFC-category, the adjunction
  $\mf{Dual}(\Op,\ca C,\Tc,\iota)$ is idempotent.
\end{theorem}
\begin{proof}
  We claim it is enough to confirm that \( \eta^{\ca C}_{\Omc(X)} \) is an
  epimorphism for all \( X \). Indeed, the triangle identity \(
  \Omc(\sigma^{\ca C}_X) \circ \eta^{\ca C}_{\Omc(X)} = \id \) tells us that
  \( \eta^{\ca C}_{\Omc(X)} \) is a split monomorphism, so our claim follows.

  First, we observe that we have
  \begin{equation*}
    \begin{tikzcd}
      \Op \Omc(X) \ar[r,"\zeta_X"] \ar[d,swap,"\eta_{\Op \Omc(X)}"]
        & \Omega(X) \ar[d,"\eta_{\Omega(X)}"] \\
      \Omega \pt \Op \Omc(X) \ar[r,swap,"\zeta_{\Op \Omc(X)}"]
        & \Omega \pt \Omega(X)
    \end{tikzcd}
  \end{equation*}
  by naturality of \( \zeta \). Since \( \pt \dashv \Omega \) is idempotent, we
  have that \( \eta_{\Omega(X)} \) is an isomorphism, and therefore so is \(
  \eta_{\Op \Omc(X)} \).

  Now, evaluating \eqref{e: one of the mate relations} at \( \Omc(X) \), we
  obtain the following commutative square
  \begin{equation*}
    \begin{tikzcd}
      \Op \Omc(X) \ar[r,"\Op(\eta^{\ca C}_{\Omc(X)})"]
                  \ar[d,"\eta_{\Op\Omc(X)}",swap]
        & \Op \Omc \ptc \Omc(X) \ar[d,"\zeta_{\ptc \Omc(X)}"] \\
      \Omega \pt \Op \Omc(X) \ar[r,swap,"\Omega(\xi_{\Omc(X)})"]
        & \Omega \ptc \Omc(X)
    \end{tikzcd}
  \end{equation*}
  from where we deduce that, since \( \Omega(\xi_{\Omc(X)}) \) is a frame
  surjection, we must have that \( \Op(\eta_{\Omc(X)}) \) is a frame
  surjection as well. In particular, it is an epimorphism.

  Since \( \Op \) reflects epimorphisms, our result now follows.
\end{proof}

\begin{corollary}
  \label{c: duality result}
  If \( (\Op, \ca C, \Tc, \iota) \) is an SFC-category, then the natural dual
  adjunction \( \mathsf{Dual}(\Op, \ca C, \Tc, \iota) \) restricts to a dual
  equivalence
  \begin{equation*}
    \begin{tikzcd}[column sep=large]
      \mathsf{Fix}(\ca C) \ar[r,"\ptc",shift left=2]
                          \ar[r,"\dashv"{rotate=-90},phantom]
        & \mathsf{Fix}(\Topc)
            \ar[l,"\Omc",shift left=2]
    \end{tikzcd}
  \end{equation*}
  between the reflective subcategories of fixpoints of \( \ptc \dashv \Omc \).
\end{corollary}

\begin{remark}
  Idempotence of \( \ptc \dashv \Omc \) allows us to obtain further conditions
  on the units. Indeed, we obtain at once that 
  \begin{equation}
    \Om(\sigma_{X}^{\ca C}), \qquad
    \Om(\xi_{\Omc X}), \quad
    \text{and} \quad \Op(\eta^{\ca C}_{\Omc X})
  \end{equation}
  are isomorphisms for all \( X \in \Topc \).
\end{remark}

We say that an object $C\in \ca C$ is \emph{sober} if every point of the frame
$\Op C$ lifts. We observe that this is equivalent to $\xi_C$ being surjective,
hence a homeomorphism, by Lemma \ref{l: xi is a subspace embedding and basic
facts about the units} above.

\begin{proposition}
  \label{p: the zeta isomorphism and the subspace inclusion ptc to ptO}
  If $(\Op,\ca C,\Tc,i)$ be an SFC-category, then the data depicted
  on \eqref{d: the basic picture of a T0 duality} together with the mate-pair
  \( \zeta, \xi \) defines a strong adjunction morphism if and only if every
  object of $\ca C$ is sober.
\end{proposition}

\begin{proof}
  By Lemma \ref{l: xi is a subspace embedding and basic facts about the
  units}, $\xi_C$ is a homeomorphism precisely when it is surjective. This,
  indeed, is equivalent to sobriety of every $C\in\ca C$.
\end{proof}

We now look at the relation between sober objects in $\ca C$ and sobriety for
topological spaces.
\begin{lemma}
  \label{l: fibers of spatial frames whose spec is in topc}
    Let $(\Op,\ca C,\Tc,\iota_{\ca C})$ be an SFC-category. If \( L \) is a
    spatial frame, then the following are equivalent:
    \begin{enumerate}[label=(\roman*)]
      \item $\pt(L) \in \Topc$.
      \item $\pt(L) \in \mathsf{Fix}(\Topc)$.
      \item The fiber of $ L $ contains a sober object.
      \item The fiber of $L$ contains a spatial sober object.
    \end{enumerate}
    In case the conditions are verified, $\Omc \pt (L)$ is, up to isomorphism,
    the  unique spatial sober object which is both in $\Op^{-1}(L)$ and
    $\mf{Fix}(\ca C)$.
\end{lemma}
\begin{proof}
  It is clear that (ii) implies (i) and that (iv) implies (iii).

  If \( \pt(L) \) is in \( \Topc \), then we can evaluate \eqref{e: the second
  of the mate relations} at \( \pt(L) \), yielding a commutative square
  \begin{equation}
    \label{e: omc pt L sober}
    \begin{tikzcd}[column sep=large,row sep=large]
      \pt (L) \ar[r,"\sigma^{\ca C}_{\pt(L)}"] 
                 \ar[d,"\sigma_{\pt(L)}",swap]
        & \ptc \Omc \pt (L) \ar[d,"\xi_{\Omc\pt(L)}"] \\
      \pt \Om \pt(L) \ar[r,swap,"\pt \zeta_{\pt (L)}"]
        & \pt \Op \Omc \pt (L)
    \end{tikzcd}
  \end{equation}
  We note that \( \pt(\zeta_{\pt(L)}) \) is invertible, and since \( \Om
  \dashv \pt \) is idempotent, \( \sigma_{\pt(L)} \) is invertible as well. It
  follows that \( \xi_{\Omc\pt(L)} \) is a split epimorphism, and therefore
  invertible by Lemma~\ref{l: xi is a subspace embedding and basic facts about
  the units}.\ref{xi subspace emb}. Moreover, it also follows that \(
  \sigma^{\ca C}_{\pt(L)} \) is invertible.

  Thus, we obtain that the spatial object \( \Omc\pt(L) \) is sober as well,
  and that \( \pt(L) \in \mathsf{Fix}(\Topc) \). We also have a string of
  isomorphisms
  \begin{equation*}
    \begin{tikzcd}
      \Op \Omc \pt(L) \ar[r,"\zeta_{\pt(L)}"]
        & \Om \pt(L) \ar[r,"\eta^{-1}_L"]
        & L
    \end{tikzcd}
  \end{equation*}
  which ensures that \( \Omc \pt(L) \) is a sober object in the fiber
  of \( L \), which is necessarily spatial. Thus, we have proved that (i)
  implies (ii), (iii) and (iv).

  Now, to show (iii) implies (i), we let \( S \) be a sober object and 
  \( \theta \colon \Op(S) \cong L \) be an isomorphism, then we consider the
  composite 
  \begin{equation*}
    \begin{tikzcd}
      \ptc(S) \ar[r,"\xi_S"]
        & \pt \Op(S) \ar[r,"\pt(\theta^{-1})"]
        & \pt (L) 
    \end{tikzcd}
  \end{equation*}
  which witnesses that \( \pt(L) \in \Topc \). 

  Finally, to prove uniqueness, we let \( S \) be a spatial, sober object and
  let \( \theta \colon \Op(S) \cong L \) be an isomorphism. We have the
  following string of isomorphisms
  \begin{equation*}
    \begin{tikzcd}
      S \ar[r,"\eta^{\ca C}_S"]
        & \Omc\ptc(S) \ar[r,"\Omc(\xi^{-1}_S)"]
        & \Omc\pt \Op(S) \ar[r,"\Omc \pt(\theta)"]
        & \Omc\pt L,
    \end{tikzcd}
  \end{equation*}
  confirming our claim.
\end{proof}

\begin{remark}
  \label{r: omc pt L sober}
  We observe that, even if \( L \) is not spatial, having \( \pt(L) \) in \(
  \Topc \) still ensures that \( \pt(L) \in \Topc \) and that \( \Omc(\pt(L))
  \) is sober, by commutativity of \eqref{e: omc pt L sober}.
\end{remark}

We denote by $\ca C_\bd{Sob}$ the full subcategory of $\ca C$ consisting of
the sober objects.
\begin{corollary}
 The dual equivalence
 \begin{equation*}
    \begin{tikzcd}[column sep=large]
      \mathsf{Fix}(\ca C) \ar[r,"\ptc",shift left=2]
                          \ar[r,"\dashv"{rotate=-90},phantom]
        & \opp{\mathsf{Fix}(\Topc)}
            \ar[l,"\Omc",shift left=2]
    \end{tikzcd}
  \end{equation*}
  restricts to a dual equivalence
   \begin{equation*}
    \begin{tikzcd}[column sep=large]
      \mathsf{Fix}(\ca C)\cap \ca C_\bd{Sob}
      \ar[r,"\ptc",shift left=2]
                          \ar[r,"\dashv"{rotate=-90},phantom]
        & \opp{\mathsf{Fix}(\Topc)} \cap \opp{\bd{Sob}}
            \ar[l,"\Omc",shift left=2]
    \end{tikzcd}
  \end{equation*}
\end{corollary}
\begin{proof}
  By Remark~\ref{r: omc pt L sober}, if \( X \cong \pt(L) \) is a sober space
  in \( \mathsf{Fix}(\Topc) \), we have that \( \Omc(X) \) is a spatial, sober
  object. Conversely, if $C\in \mathsf{Fix}(\ca C)\cap \ca C_\bd{Sob}$, then
  the homeomorphism $\xi_C \colon \ptc (C) \to \pt \Op(C) $ witnesses that \(
  \ptc(C) \) is a sober space.
\end{proof}

\subsection{The ordered category of SFC-categories}\label{ss the ordered category sfc}

Naturally, SFC-categories can be equipped with a suitable notion of morphism,
elevating them to the status of a category. In fact, we will view this a
suitable ordered sub-category of the ordered category of \textit{pointed
FC-categories}, abbreviated as \emph{PFC-categories}. In turn, these are an
ordered sub-category of the ordered category of FC-categories. More precisely,
we will have inclusions given by faithful functors
\begin{equation*}
  \begin{tikzcd}[column sep=large]
    \mathbf{SFC} \ar[r,hookrightarrow]
      & \mathbf{PFC} \ar[r,hookrightarrow]
      & \mathbf{Con}(\Frm)
  \end{tikzcd}
\end{equation*}
whose underlying functions on the hom-preorders are monotone.

A \textit{pointed FC-category}, or \textit{PFC-category}, is a quadruple \(
(\Op_{\ca C}, \ca C, \Tc, \iota) \), where \( (\ca C, \Op_{\ca C}) \) is a
FC-category, $\Tc\in \ca C$ is an object and \( \iota \colon \Op(\Tc) \cong
\bd 2 \) is an isomorphism. 

If 
\begin{align}\label{e: c and d}
  (\Op_{\ca C},\ca C,\Tc,\iota_{\ca C}) && \text{ and } 
    && (\Op_{\ca D},\ca D,\bd 2_{\ca D}, \iota_{\ca D})
\end{align}
are PFC-categories, a \textit{PFC-functor} \( (H,\alpha,i) \colon (\Op_{\ca
C},\ca C,\Tc,\iota_{\ca C}) \to (\Op_{\ca D},\ca D,\bd 2_{\ca D}, \iota_{\ca
D}) \) is a functor \( H \colon \ca C \to \ca D \), together with 
\begin{enumerate}[label=(M\arabic*),ref=M\arabic*]
  \item 
    \label{Mxx} 
    a natural isomorphism $\alpha:\Op_{\ca D}{} H \to \Op_{\ca C}$,
    so that \( (H,\alpha) \) is a FC-functor,
  \item 
    \label{M2}
    and an isomorphism $i:H(\Tc) \to \bf 2_{\ca D}$ such that \( \iota_{\ca D}
    \circ \Op_{\ca D}(i) = \iota_{\ca C} \circ \alpha_{\Tc} \).
\end{enumerate}

The composite of a string of PFC-functors
\begin{equation*}
  \begin{tikzcd}
    (\ca C, \Op_{\ca C}, \Tc, \iota_{\ca C})
      \ar[r,"{(H,\alpha,i)}"]
    & (\ca D, \Op_{\ca D}, \bf 2_{\ca D}, \iota_{\ca D})
      \ar[r,"{(G,\gamma,k)}"]
    & (\ca E, \Op_{\ca E}, \bf 2_{\ca E}, \iota_{\ca E})
  \end{tikzcd}
\end{equation*}
is given by the PFC-functor \( (GH, \alpha \circ \gamma_H, k \circ G(i)) \),
and the identity PFC-functor on a PFC-category \( (\ca C, \Op_{\ca C}, \Tc,
\iota_{\ca C}) \) is simply \( (\id_{\ca C}, \id_{\Op_{\ca C}}, \id_{\Tc}) \).

For PFC-functors \( (H,\alpha,i),\, (K,\beta,j) \colon (\Op_{\ca C}, \ca C,
\Tc, \iota_{\ca C}) \to (\Op_{\ca D}, \ca D, \bd 2_{\ca D}, \iota_{\ca D}) \),
we have \( (H,\alpha,i) \leq (K,\beta,j) \) in \( \PFC \) if and only if we
have \( (H,\alpha) \leq (K,\beta) \) in \( \Con(\Frm) \), and, letting \( \phi
\colon H \to K \) be the underlying natural transformation, the following
triangle must be commutative:
\begin{equation*}
  \begin{tikzcd}[column sep=small]
    H(\Tc) \ar[rr,"\phi_{\Tc}"]
           \ar[rd,"i",swap]
      && K(\Tc) \ar[ld,"j"] \\
    & 2_{\ca D} 
  \end{tikzcd}
\end{equation*}
So, in particular, \( \phi_{\Tc} \) must be invertible.

For a PFC-functor
\begin{equation*} 
  (H,\alpha,i) \colon (\mathcal O_{\ca C},\ca C, \bd 2_{\ca C},\iota_{\ca C}) 
    \to (\mathcal O_{\ca D},\ca D, \bd 2_{\ca D},\iota_{\ca D} ),
\end{equation*}
we introduce the following condition:
\begin{enumerate}[label=(M\arabic*),ref=M\arabic*]\setcounter{enumi}{2}
  \item 
    \label{M3}
    For each $X\in \Topc$, the underlying FC-functor \( (H,\alpha) \)
    preserves the cartesian lifting w.r.t. \( |\Op_{\ca C}(-)| \) of the
    family \( \{|\iota^{-1}_{\ca C}| \circ \chi_x:\Topp(X,\mathbb S)\to
    |\Op_{\ca C}(\Tc)| \mid x\in X\} \).  
\end{enumerate}

We define $\SFC$ as the ordered sub-category of $\PFC$ whose objects are
SFC-categories and whose morphisms are the PFC-functors which satisfy
\eqref{M3} -- we call these \textit{SFC-functors}. Moreover, for SFC-functors 
\( (H,\alpha,i), (K,\beta,j) \colon (\Op_{\ca C}, \ca C, \Tc, \iota_{\ca C})
\to (\Op_{\ca D}, \ca D, \bd 2_{\ca D}, \iota_{\ca D}) \), we define \(
(H,\alpha,i) \leq (K,\beta,j) \) if and only if this is so in \( \PFC \), so
that the inclusion \( \SFC \to \PFC \) is homwise an order embedding.

Property \eqref{M3} is motivated by the fact that it induces a right
adjunction morphism between the induced natural dual adjunctions. In order to
prove this, we first verify that:
\begin{lemma}
  \label{l: m3 implies all topc are topd}
  If \( (H,\alpha,i) \colon (\Op_{\ca C}, \ca C, \Tc, \iota_{\ca C}) \to
  (\Op_{\ca D}, \ca D, \bf 2_{\ca D}, \iota_{\ca D}) \) is a PFC-functor
  satisfying \eqref{M3}, then we have \( \Topc \subseteq \Topp_{\ca D} \).
\end{lemma}
\begin{proof}
  If \( X \in \Topc \), then the family \( \{ |\iota^{-1}_{\ca C}| \circ
  \chi_x \colon \Topp(X,\mathbb S) \to |\Op_{\ca C}(\Tc)| \mid x \in X \} \)
  has a cartesian lift \( (\Omega_{\ca C}(X), \theta_{\ca C, X}) \), which is
  preserved by \( H \). 

  Letting \( \{ \overline \chi_x \colon \Omc(X) \to \Tc \mid x \in X \} \) be
  the lifted family, it follows that \( \{ H(\overline \chi_x) \mid x \in X \} \)
  must be a cartesian family. Since we have
  \begin{align*}
    |\Op_{\ca D} H(\overline \chi_x)| 
      &= |\alpha_{\Tc}^{-1} \circ \iota_{\ca C}^{-1}| 
            \circ \chi_x \circ \theta_X \circ \alpha_{\Omc(X)} \\
      &= |\Op_{\ca D}(i^{-1}) \circ \iota_{\ca D}^{-1}| 
            \circ \chi_x \circ \theta_X \circ \alpha_{\Omc(X)}, 
  \end{align*}
  we find that \( \{ |\iota_{\ca D}^{-1}| \circ \chi_x \mid x \in X \} \) has
  a cartesian lift. Thus, we must have \( X \in \Topp_{\ca D} \).
\end{proof}

Now, we may define 
\begin{align*}
  \zeta^H_X: H\Omc X\to \Om_{\ca D} X 
    && \xi^H_C:\ptc C\to \pt_{\ca D}H C
\end{align*}
by defining their action on the carriers:
\begin{align*}
  |\Op _{\ca D}H\Om_{\ca C}X|&\to |\Op_{\ca D}\Om_{\ca D}X|
    & \ca C(C,\Tc)&\to \ca D(HC,\bd 2_{\ca D}) \\
   a &\mapsto \theta_{\ca D,X}^{-1} \theta_{\ca C,X} |\alpha_{\Omc(X)}|(a)
    & f &\mapsto i \circ H (f) 
\end{align*}
where we take \( (\Omega_{\ca D}(X),\theta_{\ca D,X}) \) to be the cartesian
lift of \( \{ |\iota^{-1}_{\ca D}| \circ \chi_x \mid x \in X \} \).

\begin{proposition}
  \label{p: dual on morphisms}
  If \( (H,\alpha,i) \colon (\ca C, \Op_{\ca C}, \Tc, \iota_{\ca C})
  \to (\ca D, \Op_{\ca D}, \bf 2_{\ca D}, \iota_{\ca D}) \) is an SFC-functor,
  then the following diagram
  \begin{equation}
    \label{d: morphisms of t0 dualities give morphisms between the adjunctions}
    \begin{tikzcd}[row sep=huge,column sep=large]
      \ca{C}
        \ar[r,shift left=2,"\ptc"{above}] 
        \ar[d,swap,"H"]
      & \bd{Top}_{\ca C}^{\mathsf{op}}
        \ar[l,shift left=2,"\Omc"{below}]
        \ar[l,phantom,"\dashv"{rotate=-90,anchor=center}]
        \ar[d,hookrightarrow] \\
      \ca D
        \ar[r,shift left=2,"\pt_{\ca D}"{above}]
      & \bd{Top}_{\ca D}^{\mathsf{op}}
        \ar[l,shift left=2,"\Om_{\ca D}"{below}]
        \ar[l,phantom,"\dashv"{rotate=-90,anchor=center}]
    \end{tikzcd}
  \end{equation}
  together with the natural transformations \( \zeta^H, \xi^H \), yields the
  data for a right adjunction morphism \( \mathsf{Dual}(\ca C, \Op_{\ca C},
  \Tc, \iota_{\ca C}) \to \mathsf{Dual}(\ca D, \Op_{\ca D}, \bf 2_{\ca D},
  \iota_{\ca D}) \).
\end{proposition}
\begin{proof}
  First, we note that since \( H\Omega_{\ca C}(X) \) underlies a cartesian
  lift, it follows that the inverse of \( |\Op_{\ca D}(\zeta^H_X)| \) lifts as
  well, so \( \zeta^H_X \) must be invertible by Lemma~\ref{l: mutually
  inverse frame isos that lift}.

  Next, to prove that \( \zeta^H \) and \( \xi^H \) are mates, we note that
  the functions \( |\Op_{\ca D}(\zeta^H_{\ptc(C)} \circ H(\eta^{\ca C}_C))| \)
  and \( |\Op_{\ca D}(\Omega_{\ca D}(\xi^H_C) \circ \eta^{\ca D}_{H(C)})| \)
  both must be equal to
  \begin{align*}
    \theta_{\ca D,\ptc(C)}^{-1} \circ \theta_{\ca C,\ptc(C)} 
        \circ |\alpha_{\Omc\ptc(C)} \circ |\Op_{\ca D}H(\eta^{\ca C}_C)| 
    &= \theta_{\ca D,\ptc(C)}^{-1} \circ \theta_{\ca C,\ptc(C)} 
        \circ |\Op_{\ca C}(\eta^{\ca C}_C)| \circ |\alpha_C| \\
    &= \theta_{\ca D,\ptc(C)}^{-1} \circ \nu_{\Op_{\ca C},C} \circ |\alpha_C|
    \\
    &= \theta_{\ca D,\ptc(C)}^{-1} \circ \Topp(|\xi_C^H|,\mathbb S) 
                                   \circ \nu_{\Op_{\ca D},H(C)},
  \end{align*}
  where the last equality is a consequence of the following calculation:
  \begin{align*}
    \nu_{\ca C,|\alpha_C|(a)}(f)
      &= |\iota_{\ca C} \circ \Op_{\ca C}(f) \circ \alpha_C|(a) \\
      &= |\iota_{\ca C} \circ \alpha_{\Tc} \circ \Op_{\ca D}H(f)|(a) \\
      &= |\iota_{\ca D} \circ \Op_{\ca D}(i \circ H(f))|(a)
  \end{align*}
  Thus, our result follows by faithfulness.
\end{proof}

For an SFC-functor \( (H,\alpha,i) \), we call $\mathsf{Dual}(H,\alpha,i)$ the
right adjunction morphism
\[
  (H:\ca C\to \ca D,\Topc\subseteq \bd{Top}_{\ca D},\zeta,\xi)
\]
as described by Proposition~\ref{p: dual on morphisms}. We highlight that \(
\mathsf{Dual} \colon \SFC \to \bf{RADJ} \) defines a 2-functor; indeed, if \(
(H,\alpha,i) \leq (K,\beta,j) \), then there exists a (unique) natural
transformation \( \lambda \colon H \to K \) such that \( \beta \circ \Op_{\ca
D}(\lambda) = \alpha \). By inspecting the underlying functions on carrier
sets of \( \zeta^H \) and \( \zeta^K \), one is led to conclude that \(
\zeta^K \circ \lambda_{\Omega_{\ca C}} = \zeta^H \). Thus, the pair
\begin{equation*}
  (\lambda, \id_{\Topc \subseteq \Topp_{\ca D}}) 
    \colon \mathsf{Dual}(H,\alpha,i) \to \mathsf{Dual}(K,\beta,j)
\end{equation*}
defines a 2-cell in \( \bf{RADJ} \), where \( {\id}_{\Topc \subseteq
\Topp_{\ca D}} \) is the identity natural isomorphism on the inclusion functor
\( \Topc \subseteq \Topp_{\ca D} \).

\begin{lemma}\label{l: thr initial topology from xi}
  If $(H,\alpha,i):(\Op_{\ca C},\ca C,\Tc,\iota_{\ca C})\to (\Op_{\ca D},\ca
  D,\bd 2_{\ca D},\iota_{\ca D})$ is a SFC-functor, then the components of
  $\xi^H$ are subspace embeddings.
\end{lemma}
\begin{proof}
  The topology on $\ptc C$ is the initial one induced by $\xi^H_C$. Then,
  the map is a subspace embedding exactly when it is injective. This holds
  for all $C\in \ca C$, since $H$ is faithful (see Lemma \ref{l: morphisms in fc are faithful}). 
\end{proof}

\begin{proposition}\label{p: the zeta generalized}
  Let $(H,\alpha,i):(\Op_{\ca C},\ca C,\Tc,\iota_{\ca C})\to (\Op_{\ca D},\ca
  D,\bd 2_{\ca D},\iota_{\ca D})$ be a morphism in $\SFC$.  If $H$ is full,
  $\mf{Dual}(H,\alpha,i)$ is a strong adjunction morphism.
\end{proposition}
\begin{proof}
  If $H$ is full, by Lemma \ref{l: thr initial topology from xi}, $\xi^H_C$ is
  a subspace embedding which is surjective, hence a homeomorphism.
\end{proof}

\begin{lemma}
  In the ordered category of SFC-categories, the terminal object is
  \[
    (\mathds{1}_{\Frm},\Frm,\bd 2,\id_{\bd{2}}),
  \]
  and the morphism from an SFC-category $(\Op,\ca C,\Tc,\iota_{\ca C})$ to the
  terminal object is $(\Op,\id_{\Op},\iota_{\ca{C}})$.
\end{lemma}
\begin{proof}
  We only need to confirm that \( (\Op, \id_{\Op}, \iota_{\ca C}) \) is indeed
  an SFC-functor; that is, that this PFC-functor satisfies \eqref{M3}. This
  follows by invertibility of \( \zeta \).
\end{proof}

Thus, we obtain that diagram \eqref{d: the basic picture of a T0 duality} is
an instance of diagram \eqref{d: morphisms of t0 dualities give morphisms
between the adjunctions}, in the case where $(H,\alpha,i)$ a morphism to the
terminal object.

Since \( \PFC \) is an ordered category, we may consider the notion of
adjunction for PFC-functors; given PFC-functors 
\begin{equation}
  \label{e: lr-pfc-functors}
  \begin{tikzcd}
    (\ca C, \Op_{\ca C},\Tc,\iota_{\ca C}) 
      \ar[r,shift right=2,swap,"{(H.\alpha,i)}"]
    & (\ca D, \Op_{\ca D},\bd{2_{\ca D}},\iota_{\ca D}),
      \ar[l,shift right=2,swap,"{(K,\beta,j)}"]
  \end{tikzcd}
\end{equation}
we have a \textit{PFC-adjunction} \( (K,\beta,j) \dashv (H,\alpha,i) \) if and
only if we have
\begin{equation*}
  (K,\beta,j) \circ (H,\alpha,i) 
    \leq \id_{(\ca C, \Op_{\ca C}, \Tc, \iota_{\ca C})}
    \qquad \text{and} \qquad
  \id_{(\ca D, \Op_{\ca D}, \bf 2_{\ca D}, \iota_{\ca D})}
    \leq (H,\alpha,i) \circ (K,\beta,j).
\end{equation*}
in \( \PFC \).

\begin{remark}
  \label{r: pfc-adj redunancy}
  We highlight that PFC-functors as in \eqref{e: lr-pfc-functors} form a
  PFC-adjunction \( (K,\beta,j) \dashv (H,\alpha,i) \) if and only if we have
  an FC-adjunction \( (K,\beta) \dashv (H,\alpha) \) and the following
  equations hold
  \begin{equation*}
    j \circ K(i) = \epsilon_{\Tc}, \qquad
    i \circ H(j) \circ \eta_{\bf 2_{\ca D}} = \id,
  \end{equation*}
  where \( \epsilon \) and \( \eta \) are the underlying unit and counit.

  Besides the redundant isomorphisms and conditions present in an
  FC-adjunction as pointed out in Remark~\ref{r: redundant info in FC-adj}, we
  note that \(i\) and \(j\) determine each other as well, provided that either
  one of \( \epsilon_{\Tc} \), \( \eta_{\bf 2_{\ca D}} \) is invertible.
\end{remark}

Building on Lemma~\ref{l: when does fc-functor have adjoint}, we obtain the
following analogue for PFC-functors:

\begin{lemma}
  \label{l: when does pfc-functor have adjoint}
  Let \( (H,\alpha,i) \colon (\ca C, \Op_{\ca C},\Tc, \iota_{\ca C}) \to (\ca
  D, \Op_{\ca D}, \bf 2_{\ca D}, \iota_{\ca D}) \) be a PFC-functor. We have
  that:
  \begin{enumerate}[label=(\alph*)]
    \item
      \label{l: when has right pfc-adjoint}
      \( (H,\alpha,i) \) has a left PFC-adjoint \( (L,\lambda,j) \) if and only
      if we have an ordinary adjunction \( L \dashv F \) whose unit \( \eta
      \colon \id \to HL \) is such that \( \Op_{\ca D}(\eta) \) is invertible,
      and whose counit is invertible at \( \Tc \).
    \item
      \label{l: when has left pfc-adjoint}
      \( (H,\alpha,i) \) has a right PFC-adjoint \( (R,\rho,k) \) if and only
      if we have an ordinary adjunction \( H \dashv R \) whose counit \(
      \epsilon \colon HR \to \id \) is such that \( \Op_{\ca D}(\epsilon) \)
      is invertible, and whose unit is invertible at \( \bf 2_{\ca D} \). 
  \end{enumerate}
\end{lemma}
\begin{proof}
  Indeed, this follows immediately by Lemma~\ref{l: when does fc-functor have
  adjoint} and Remark~\ref{r: pfc-adj redunancy}. We only check the argument
  for \ref{l: when has right pfc-adjoint}, as the other is analogous.

  The given data guarantees an FC-adjunction \( (L,\lambda) \dashv (H,\alpha)
  \), and if we let \( \delta \) be the counit of \( L \dashv H \), we define \(
  j = \delta_{\Tc} \circ K(i^{-1}) \). Then we note that
  \begin{align*}
    H(j^{-1}) \circ i^{-1} 
      &= HK(i) \circ H\delta^{-1}_{\Tc} \circ i^{-1} \\
      &= HK(i) \circ \eta_{H(\Tc)} \circ i^{-1} = \eta_{\bf 2_{\ca D}},
  \end{align*}
  so we indeed obtain a PFC-adjunction.
\end{proof}

We say that an adjoint functor is \emph{PFC-adjoint} if it is part of a
PFC-adjunction. We next show that all right PFC-adjoints satisfy \eqref{M3}.

\begin{lemma}
  \label{l: right pfc adjoint preserve cartesian lifts}
  If we have a PFC-adjunction
  \begin{equation*}
    \begin{tikzcd}
      (\ca C, \Op_{\ca C},\Tc,\iota_{\ca C}) 
        \ar[r,shift right=2,swap,"{(R.\alpha,i)}"]
        \ar[r,phantom,"\dashv"{rotate=-90}]
      & (\ca D, \Op_{\ca D},\bd{2_{\ca D}},\iota_{\ca D}),
        \ar[l,shift right=2,swap,"{(L,\lambda,j)}"]
    \end{tikzcd}
  \end{equation*}
  then \( (R,\alpha,i) \) satisfies \eqref{M3}. 
\end{lemma}
\begin{proof}
  Indeed, we have an FC-adjunction \( (L,\lambda) \dashv (R,\alpha) \), as noted
  in the previous remark, so the result is a consequence of Lemma~\ref{l:
  radj-prsv-carts}.
\end{proof}

\begin{lemma}
  \label{l: induced spec iso}
  If we have a PFC-adjunction
  \begin{equation*}
    \begin{tikzcd}
      (\ca C, \Op_{\ca C},\Tc,\iota_{\ca C}) 
        \ar[r,shift right=2,swap,"{(H.\alpha,i)}"]
        \ar[r,phantom,"\dashv"{rotate=-90}]
      & (\ca D, \Op_{\ca D},\bd{2_{\ca D}},\iota_{\ca D}),
        \ar[l,shift right=2,swap,"{(L,\lambda,j)}"]
    \end{tikzcd}
  \end{equation*}
  then we have a natural isomorphism \( \omega \colon \pt_{\ca D} \to \ptc L
  \).
\end{lemma}
\begin{proof}
  We begin by noting that we have a natural isomorphism on the underlying
  carrier sets, given by
  \begin{equation*}
    \begin{tikzcd}
      {|\pt_{\ca D}|} \ar[r,equal]
        & \ca D(-,\bf 2_{\ca D}) \ar[r,"i^{-1} \circ -"]
        & \ca D(-,H(\Tc)) \ar[r,"\cong"]
        & \ca C(L(-),\Tc) \ar[r,equal]
        & {|\ptc L|}
    \end{tikzcd}
  \end{equation*}
  which we denote by \( \omega \). We also obtain a commutative square
  \begin{equation*}
    \begin{tikzcd}
      \Op_{\ca C}L \ar[r,"\eta^{\ca C}"]
                   \ar[d,"\lambda",swap]
        & \Set\big(\ca C(L(-),\Tc),2\big) \ar[d,"- \circ \omega"] \\
      \Op_{\ca D} \ar[r,swap,"\eta^{\ca D}"]
        & \Set\big(\ca D(-,\bf 2_{\ca D}),2\big)
    \end{tikzcd}
  \end{equation*}
  via the following calculation:
  \begin{align*}
    \iota_{\ca C} \circ \Op_{\ca C}(\epsilon_{\Tc} \circ L(i^{-1} \circ f))
      &= \iota_{\ca C} \circ \Op_{\ca C}(j \circ L(f)) \\
      &= \iota_{\ca C} \circ \Op_{\ca C}(j) \circ \Op_{\ca C}L(f) \\
      &= \iota_{\ca D} \circ \lambda_{\bf 2_{\ca D}} \circ \Op_{\ca C}L(f) \\
      &= \iota_{\ca D} \circ R(f) \circ \lambda_D.
  \end{align*}
  Since both vertical transformations are isomorphisms, it follows that
  the components of \( \omega \) lift to homeomorphisms, and thereby we obtain
  a natural isomorphism \( \pt_{\ca D} \to \ptc L \).
\end{proof}
 
\begin{theorem}
  \label{t: right pfc adjoints are continuous and they also give cartesian lifts}
  Suppose that $(R,\alpha,i):(\Op_{\ca C},\ca C,\Tc,\iota_{\ca C}) \to
  (\Op_{\ca D},\ca D,\bd 2_{\ca D},\iota_{\ca D})$ is a right PFC-adjoint.

  If $(\Op_{\ca C},\ca C,\Tc,\iota_{\ca C})$ is an SFC-category, then so is
  $(\Op_{\ca D},\ca D,\bd 2_{\ca D},\iota_{\ca D})$, and \( (R,\alpha,i) \) is
  an SFC-functor.
\end{theorem}
\begin{proof}
  By Lemma~\ref{l: right pfc adjoint preserve cartesian lifts}, we have that
  right PFC-adjoints satisfy \eqref{M3}, so it follows that \( \Topc \subseteq
  \Topp_{\ca D} \), by Lemma~\ref{l: m3 implies all topc are topd}. 

  Thus, if \( (\ca C, \Op_{\ca C}, \Tc, \iota_{\ca C}) \) is an SFC-category,
  we get \( \mathbb S \in \Topp_{\ca D} \), and Lemma~\ref{l: induced spec
  iso} guarantees that \( \pt_{\ca D}(D) \in \Topp_{\ca C} \) for all \( D \in
  \ca D \). We conclude that \( (\ca D, \Op_{\ca D}, \bf 2_{\ca D}, \iota_{\ca
  D}) \) must be an SFC-category as well, and that \( (R,\alpha,i) \) is an
  SFC-functor.
\end{proof}

We now specialize our study to subcategory inclusions. For an SFC-category $(\Op,\ca C,\Tc,\iota_{\ca C})$, and for any subcategory inclusion $I:\ca D\subseteq \ca C$, 
\[
(I,\id_{\Op {I}}):(\Op I,\ca D)\to (\Op,\ca C)
\]
is an FC-functor. So, it suffices for there to be an object $\bd 2_{\ca D}\in
\ca D$ and an isomorphism $i:I\bd 2_{\ca D}\cong \Tc$ to obtain a PFC-functor
\[
(I,\id_{\Op{I}},i):(\Op I,\ca D,\bd 2_{\ca D},\iota_{\ca C}\circ\Op i)
    \to (\Op,\ca C,\Tc,\iota_{\ca C}).
\]
Even in the case where $(\Op I,\ca D,\bd 2_{\ca D},\iota_{\ca C}\circ\Op i)\in
\SFC$, this need not be an SFC-functor. We provide examples for both cases.

\begin{example}
  For the $T_D$-duality from \cite{banaschewskitd} (see Example \ref{e: td
  duality}), the inclusion $i_D:\bd{Frm_D}\subseteq \Frm$ gives a morphism of
  SFC-categories $(i_D,\id_{\Op i_D},\id_{\bd{2}})$, and the right map of adjunctions $\mf{Dual}(i_D,\id_{\Op i_D},\id_{\bd{2}})$ is depicted in Diagram \eqref{d: left map of adjunctions for
  TD}.
\end{example}

\begin{example}\label{e: the inclusion of sobers is not continuous}
 Recall the SFC-categories
\begin{align*}
& (\Om,\Topz^{\mathsf{op}},\{*\},\id_{\bd{2}})  & & (\Om,\bd{Sob}^{\mathsf{op}},\{*\},\id_{\bd{2}})
\end{align*}
from Examples \ref{e: t0 spaces example} and \ref{e: the sober example}.
Consider the inclusion $i_S:\bd{Sob}\subseteq \Topz$. The triple
$(i_S,\id_{\Op i_S},\id_{\{*\}})$ is a PFC-functor, but it does not
satisfy \eqref{M3}: for a space $X$, the object with the cartesian lift property is $X_0$ in the category $\Topz$ (see Example \ref{e: t0 spaces example}), and $\mf{sob}X$ in the category $\bd{Sob}$.
\end{example}

Let \( (\ca C, \Op) \) be an FC-category, and let \( I \colon \ca D \to \ca C
\) be a coreflective subcategory with coreflector \( R \colon \ca C \to \ca D
\), and counit \( \lambda \). We say that \( I \) is \textit{FC-coreflective}
if \( \Op(\lambda) \) is invertible. Indeed, this will imply that we have an
FC-adjunction
\begin{equation*}
  \begin{tikzcd}
    (\ca D, \Op I) \ar[r,shift left=2,"{(I,\id)}"]
                   \ar[r,"\dashv"{rotate=-90},phantom]
    & (\ca C, \Op) \ar[l,shift left=2,"{(R,\Op(\lambda))}"]
  \end{tikzcd}
\end{equation*}
by Lemma~\ref{l: when does fc-functor have adjoint}. In fact, an analogous
result also holds for PFC-categories:

\begin{proposition}
  \label{p: coreflective stable}
  Let \( (\ca C, \Op, \Tc, \iota) \) be a PFC-category, and let \( I \colon
  \ca D \to \ca C \) be an FC-coreflective subcategory. Writng \( R \colon \ca
  C \to \ca D \) for the coreflector and \( \lambda \) for the counit, we
  claim that the following hold:
  \begin{enumerate}[label=(\alph*)]
    \item
      \( \big(\ca D, \Op I, R(\Tc), \iota \circ \Op(\lambda_{\Tc})\big) \) is
      a PFC-category,
    \item
      \( (R,\Op(\lambda),\id) \) is a right PFC-functor, whose left
      PFC-adjoint is \( (I,\id,\lambda_{\Tc}) \),
    \item
      If \( (\ca C, \Op, \Tc, \iota) \) is an SFC-category, then so is \( (\ca
      D, \Op I, R(\Tc), \iota \circ \Op(\lambda_{\Tc})) \), and \(
      (R,\Op(\lambda),\id) \) is an SFC-functor.
  \end{enumerate}
\end{proposition}
\begin{proof}
  For (a), we merely need to observe that
  \begin{equation*}
    \begin{tikzcd}
      \Op IR(\Tc) \ar[r,"\Op(\lambda_{\Tc})"]
        & \Op(\Tc) \ar[r,"\iota"]
        & \bf 2
    \end{tikzcd}
  \end{equation*}
  is invertible, since \( \lambda_{\Tc} \) is invertible.

  We obtain (b) by applying Lemma~\ref{l: when does pfc-functor have
  adjoint}.\ref{l: when has left pfc-adjoint}, highlighting that the unit of
  \( I \dashv R \) is invertible at all components.

  Finally, (c) is an application of Theorem~\ref{t: right pfc adjoints are
  continuous and they also give cartesian lifts}.
\end{proof}

\begin{remark}[SFC-adjunctions]
  If we have an \textit{SFC-adjunction} -- a PFC-adjunction such that the left
  adjoint also satisfies \eqref{M3} -- as in the following diagram
  \begin{equation*}
    \begin{tikzcd}
      (\ca C, \Op_{\ca C},\Tc,\iota_{\ca C}) 
        \ar[r,shift right=2,swap,"{(R.\alpha,i)}"]
        \ar[r,phantom,"\dashv"{rotate=-90}]
      & (\ca D, \Op_{\ca D},\bd{2_{\ca D}},\iota_{\ca D}),
        \ar[l,shift right=2,swap,"{(L,\lambda,j)}"]
    \end{tikzcd}
  \end{equation*}
  then we obtain the following results:
  \begin{itemize}
    \item
      By Lemma~\ref{l: m3 implies all topc are topd}, we conclude that \(
      \Topc = \Topp_{\ca D} \). 
    \item
      Since \( \mathsf{Dual} \) is a 2-functor, from Lemma~\ref{l: right adj
      comm implies left adj comm} we deduce that \( \xi^D \colon \pt_{\ca D}
      \to \ptc \ca F \) is invertible.
    \item
      By Proposition~\ref{p: dual on morphisms}, we obtain that \( \Omc \cong
      \ca F \Om_{\ca D} \).
  \end{itemize}

  We expect the third condition will seldom hold in practice (see, for
  instance Remark~\ref{r: why Ci into C does not satisfy M3}), 
  analogue of the second condition already holds for left PFC-adjoints
  (see Corollary~\ref{c: if frames are c, then sobers are topc} for a
  particular case), via an arguments similar to Lemma~\ref{l: right adj comm
  implies left adj comm}.
\end{remark}

\subsection{Pointfree sobrification}\label{ss: pointfree sobrification}

We now work towards constructing a pointfree version of sobrification. For an object $C\in \ca C$, we call $s_C:\mf{sob}(C)\to C$ the \emph{sobrification} of $C$ if it is a sober coreflection, and if that coreflection is an FC-adjunction. We introduce the property of $\Op:\ca C\to \Frm$ having a left FC-adjoint. We observe that if $\ca F:\Frm\to \ca C$ is a left FC-adjoint of $\Op$,
$\ca F[\Frm]$ is coreflective in $\ca C$, with the composition $\ca F \Op:\ca C\to
\ca F[\Frm]$ giving the coreflector.
\begin{example}
The $\Frm$-concrete categories $\Op_{\ca R}:\bd{Raney}\to \Frm$ and $\Op_{\ca S}:\bd{Skula}\to \Frm$ both satisfy this property, by Lemmas \ref{l: fse is left adjoint} and \ref{l: congruence frame gives left adjoint}.
\end{example}

\begin{lemma}
  \label{l: ptc of the initial of a spatial is sober}
Suppose that $(\Op,\ca C,\Tc,\iota_{\ca C})$ is a SFC-category where $\Op$ has a left FC-adjoint $\ca F$. For every $C\in \ca C$, the map 
    \[
    \xi_{\ca F\ca O C}:\ptc \ca F\Op C\to \pt\Op C
    \]
    is a homeomorphism. Hence, $\ca F\Op C$ is sober.
\end{lemma}
\begin{proof}
By FC-adjointness, the map is a bijection, and as $\xi$ is a subspace embedding at all components (\ref{xi subspace emb}), it is a homeomorphism.
\end{proof}

\begin{corollary}
  \label{c: if frames are c, then sobers are topc}
  If \( \Op \) has a left FC-adjoint, then \( \mathbf{Sob} \subseteq \Topc \).
\end{corollary}
\begin{proof}
  We have \( \nu_L \colon L \cong \Op \ca F(L) \), so we have an isomorphism
  \begin{equation*}
    \begin{tikzcd}
      \ptc\ca F(L) \ar[r,"\xi_{\ca F(L)}"]
        & \pt \Op \ca F(L) \ar[r,"\pt(\nu_L^{-1})"]
        & \pt (L),
    \end{tikzcd}
  \end{equation*}
  highlighting that \( \xi_{\ca F(L)} \) and \( \xi_{\ca F\Op\ca F(L)} \) are
  part of a naturality square of isos, for all frames \( L \).
\end{proof}

\begin{lemma}\label{l: sober object lifts the frame spatialization}
  Suppose that $(\Op,\ca C,\Tc,\iota_{\ca C})$ is an SFC-category, and let \(
  C \in \ca C \) be such that \( \pt \Op(C) \in \Topc \). Then the following
  are equivalent:

  \begin{enumerate}[label=(\roman*)]
    \item
      There is a lift of the composition 
        \[
        \begin{tikzcd}[column sep=large]
            \Op C
            \ar[r,"\eta_{\Op C}"]
            & \Om \pt \Op C.
            \ar[r,"\zeta^{-1}_{\pt \Op C}"]
            & \Op \Omc \pt \Op C
        \end{tikzcd}
        \]
    \item
      $C$ is sober.
  \end{enumerate}
\end{lemma}
\begin{proof}
  It is obvious that sobriety implies that a lift is given by \( \Omega_{\ca
  C}(\xi_C^{-1}) \circ \eta^{\ca C}_C \). For the other direction, suppose \(
  \Op(g) = \zeta^{-1}_{\pt \Op(C)} \circ \eta_{\Op(S)} \), and \( \pt\Op C \in
  \mf{Fix}(\Topc)\). For every point $f:\Om \pt \Op C\to \bd 2$ there is
  $\overline{f}:\Omc \pt \Op C\to \Tc$ lifting it, as $\pt \Op C\in \Topc$
  implies that \( \Omc \pt \Op (C) \) is sober. Every point $h:\Op C\to \bd 2$
  is $h_s\circ \eta_{\Op C}$ for some point $h_s:\Om \pt \Op C\to \bd{2}$.
  Then, $\overline{h_s}\circ g$ is a lift of $h$, making $\xi_C$ surjective,
  hence a homeomorphism.
\end{proof}

For a sober object $C\in \ca C$, we will denote by $\eta^{\sigma}_{\Op C}$ the
lift of $\zeta^{-1}_{\pt \Op(C)} \circ \eta_{\Op(C)}$ given by the lemma
above.
\begin{lemma}\label{l: map from sober induces a commutative square}
  Suppose that $(\Op,\ca C,\Tc,\iota_{\ca C})$ is an SFC-category, and that
  $f:C\to D$ is in $\ca C$, and $D$ is sober. Then, there is a commuting
  square as follows.
\begin{equation}\label{d: pullback of this is sobrification}
    \begin{tikzcd}[row sep=large,column sep=large]
    D
    \ar[r,"f"]
    \ar[d,"\Omc \pt \Op f\circ \eta^{\sigma}_{\Op D}",swap]
    & C
    \ar[d,"\eta^{\ca C}_C"]\\
    \Omc \pt \Op C
    \ar[r,"\Omc(\xi_C)",swap]
    & \Omc \ptc C
    \end{tikzcd}
\end{equation}
\end{lemma}
\begin{proof}
We have:
    \begin{align*}
        \Omega_{\ca C}(\xi_C) \circ \Omc \pt \Op f \circ \eta^{\sigma}_{\Op D}&=\Omc \ptc f\circ \Omc(\xi_D)\circ \eta^{\sigma}_D &\text{ naturality of $\xi$}\\
        &=\Omc \ptc f\circ \eta^{\ca C}_D &\text{ Lemma \ref{l: sober
        object lifts the frame spatialization}}\\
        &=\eta_C^{\ca C}\circ f. & \text{ naturality of $\eta^{\ca C}$}
    \end{align*}
\end{proof}

\begin{proposition}\label{p: sobrification as a pullback}
Suppose that $(\Op,\ca C,\Tc,\iota_{\ca C})$ is an SFC-category where $\ca O$ is a right FC-adjoint. If the pullback of the following diagram exists:
\[
\begin{tikzcd}[row sep=large,column sep=large]
   & C  
   \ar[d,"\eta^{\ca C}_C"]\\
     \Omc \pt \Op C
     \ar[r,"\Omc (\xi_C)",swap]
     & \Omc \ptc C
\end{tikzcd}
\]
then the horizontal arrow is the sobrification of $C$.
    
\end{proposition}
\begin{proof}
    Suppose that the pullback is given by \begin{equation*}
    \begin{tikzcd}[row sep=large,column sep=large]
    D
    \ar[r,"h_1"]
    \ar[d,"h_2",swap]
    & C
    \ar[d,"\eta^{\ca C}_C"]\\
    \Omc \pt \Op C
    \ar[r,"\Omc(\xi_C)",swap]
    & \Omc \ptc C
    \end{tikzcd}
\end{equation*}

\begin{claim}\label{claim 2}
$\Op h_1$ is a frame isomorphism.
\end{claim}
\begin{proof}(Of Claim 1).
Because $\ca F\Op C$ is sober (Lemma \ref{l: ptc of the initial of a spatial is sober}), it induces a commutative square as in \ref{l: map from sober induces a commutative square}, and using the pullback universal property we have this diagram:
\begin{equation*}
 \begin{tikzcd}
   \ca F \ca O C 
   \ar[dr,"u"]
    \ar[drr,"\lambda_C",bend left=20]
     \ar[ddr,"\Omc \pt \Op \lambda_C\circ \eta^{\sigma}_{\Op \ca F \Op C}",swap,bend right=20]
    \\& D
    \ar[r,"h_1"]
    \ar[d,"h_2",swap]
    & C
    \ar[d,"\eta^{\ca C}_C"]\\
    & \Omc \pt \Op C
    \ar[r,"\Omc(\xi_C)",swap]
    & \Omc \ptc C
    \end{tikzcd}
\end{equation*}
Since $\Op \lambda_{C}$ is a frame isomorphism, $\Op h_1$ is a frame surjection. Using adjointness $\ca F\dashv \Op$, we get some map $\Tilde{u}:\ca F\ca O C\to \ca F \Op D$ such that $\lambda_D\circ \Tilde{u}=u$, and as $\Op \lambda_D$ is an isomorphism, it is a surjective frame map, which gives that $\Op u$ is surjective, but as it is split mono ($\Op h_1\circ \Op u$ is an isomorphism), then it must be a frame isomorphism. Then, $\Op h_1$, too, is an isomorphism. 
\end{proof}

\begin{claim}\label{claim 1}
$D$ is sober.
\end{claim}

\begin{proof}(Of Claim 2). We use the characterization in \ref{l: sober object lifts the frame spatialization}. Indeed, $\pt \Op D$ must be in $\Topc$, as it is in $\mf{Fix}(\Topc)$ by Lemma \ref{l: ptc of the initial of a spatial is sober}. We claim that the required lift is $h_2$ followed by the isomorphism $\Omc\pt\Op h_1^{-1}$ (we know $\Op h_1$ to be invertible by \ref{claim 1}). We have 
\begin{align*}
    \Op \Omc \pt \Op h_1^{-1}\circ \Op h_2\circ \Op u
    &= \Op \Omc \pt \Op h_1^{-1}\circ \Op \Omc\pt\Op \lambda_C\circ \Op \eta^{\sigma}_{\Op \ca F\Op C} & \text{ left triangle }\\
    &= \Op \Omc \pt \Op h_1^{-1}\circ \Op \Omc\pt\Op \lambda_C\circ \zeta^{-1}_{\pt \Op \ca F \Op C}\circ \eta_{\Op \ca F \Op C}& \text{ definition of $\eta^{\sigma}$}\\
    &= \Op \Omc \pt \Op u\circ \zeta^{-1}_{\pt \Op \ca F \Op C}\circ \eta_{\Op \ca F \Op C}&\text{ right triangle}\\
    &=\zeta^{-1}_{\pt \Op D}\circ \Om \pt \Op u\circ \eta_{\Op \ca F \Op C}& \text{ naturality of $\zeta^{-1}$}\\
     &=\zeta^{-1}_{\pt \Op D}\circ \eta_{\Op D}\circ \Op u & \text{ naturality of $\eta$}
\end{align*}
Since $\Op u$ is an isomorphism, the desired equality holds.
\end{proof}
Then $h_1:D\to C$ is a map from a sober object. Lemma \ref{l: map from sober induces a commutative square} gives universality and Claim 1 gives that the associated adjunction is an FC-adjunction. 
\end{proof}

\begin{corollary}\label{c: sufficient condition for sobrification to exist}
    If $(\Op,\ca C,\Tc,\iota_{\ca C})$ is an SFC-category where $\ca O$ is a right FC-adjoint, and if $\ca C$ has all pullbacks, sober objects are a coreflective subcategory of $\ca C$.
\end{corollary}
We recall the following result in point-set topology (\cite{suarez22}, Lemma 4.7).
\begin{theorem}
    If $X$ is a sober space, and $Y\subseteq X$ is a subspace, the sobrification of $Y$ is given by the inclusion of $Y$ into
    \[
    \bigcap \{Z\subseteq X\mid Y\subseteq Z,X\text{ is sober}\}
    \]
\end{theorem}

We now prove a pointfree, categorical version of the theorem. 

\begin{lemma}
  \label{l: sober interior well defined}
  Let $(\Op,\ca C,\Tc,\iota_{\ca C})$ be an SFC-category. The collection of
  sober objects in a fiber is downclosed. 

  Furthermore, let $(D_i,\theta^i)\in \Op^{-1}(\Op C)$ be a family such that
  each $D_i$ is sober. If the join $\bigvee_i (D_i,\theta^i)$ exists, then it
  is sober. 
\end{lemma}
\begin{proof}
  Suppose that $(D,\theta^D),(E,\theta^E)\in \Op^{-1}(\Op C)$, that $E$ is
  sober, and that there is a morphism $f:D\to E$ in the fiber. If $p:\Op C\to
  \bd{2}$ is a point, we call $\overline{p}:E\to \Tc$ the lift of $p$ given by
  sobriety of $E$; then $\overline{p}\circ f:D\to \Tc$ is a lift. 

  For a family $(D_i,\theta^i)\in \Op^{-1}(\Op C)$ where each $D_i$ is sober,
  if $\bigvee_i (D_i,\theta^i)$ exists, then the lift of a point $p:\Op C\to
  \bd{2}$ is given by the universal property of the colimit, as the map
  $\overline{p}:\bigvee_i D_i\to \Tc$ corresponding to the collection of lifts
  $\overline{p_i}:D_i\to \Tc$. 
\end{proof}

\begin{theorem}\label{t: sobrification as a sober interior}
Let $(\Op,\ca C,\Tc,\iota_{\ca C})$ be an SFC-category where $\Op$ is a right PCF-adjoint. If the sobrification of $C\in \ca C$ exists, it is given by: 
\[
\mathit{int}_{\mf{sob}}(C):=\bve \{D\in \Op^{-1}(\Op C)\mid D\text{ is sober}\}
\]
\end{theorem}
\begin{proof}

Let $s_C:\mf{sob}(C)\to C$ be the sobrification. As sobrification is a FC-adjoint, $(\mf{sob}(C),\Op(s_C))\in \Op^{-1}(\Op C)$ and as $\mf{sob}(C)$ is sober we must have $\mf{sob}(C)\leq \mathit{int}_{\mf{sob}}(C)$. On the other hand, the universal map
\[
f:\mathit{int}_{\mf{sob}}(C)\to C
\]
is a map to $C$ from a sober object, hence it factors through the sobrification:
\begin{equation*}
    \begin{tikzcd}
& \mf{sob}(C)
\ar[dr,"s_C"]\\
\mathit{int}_{\mf{sob}}(C)
\ar[rr,"f"]
\ar[ur,"\Tilde{f}"]
&& C.
    \end{tikzcd}
\end{equation*}
By our assumption, $\Op (s_C)$ and $\Op(f)$ are frame isomorphism; hence, so is $\Op (\Tilde{f})$. We then get the reverse inequality $\mathit{int}_{\mf{sob}}(C)\leq \mf{sob}(C)$ in $\Op^{-1}(\Op C)$.
\end{proof}

\subsection{Concrete examples}\label{ss concrete examples 1}
We prove the main natural adjunction in \cite{suarez25raney} between Raney extensions and topological spaces as a special case of Theorem \ref{t: main theorem in porst tholen adapted}.

\begin{theorem}\label{t: duality for raney extensions}
    \label{l: raney extensions are a quasi t0 duality}
The PFC-category $(\Op_{\ca R},\bd{Raney},\bd 2_{\ca R},\id_{\bd{2}})$ is spatializable. The adjunction 
\[
\mf{Dual}(\Op_{\ca R},\bd{Raney},\bd 2_{\ca R},\id_{\bd{2}})
\]
is such that $\Topc=\Topp$ and $\mf{Fix}(\Topc)=\Topz$.
\end{theorem}
\begin{proof}

We will show that $(\Op_{\ca R},\bd{Raney},\bd 2_{\ca R},\id_{\bd{2}})$ is spatializable, with $\Topc=\Topp$, by showing that for every space $X\in \Topp$, the Raney extension $(\Om X,\mathcal{U}^*X)$ provides the cartesian lift for the family of characteristic functions. Let $(L,\ca F)$ be a Raney extension and $f:|L|\to |\Om X|$ be a function. Suppose that for all $x\in X$ the function $\chi_x{\circ} f$ lifts to a map 
\[
\chi_x\circ f:(L,\ca F)\to \bd 2_{\ca R}
\]
of Raney extensions. This means $(\chi_x\circ f)^{-1}(1)\in \ca F$. In particular, this gives lifts in $\bd{Frm}$ of each $\chi_x\circ f$, and so, by the cartesian lift property of $\Om X$ for frames, the map $f$ is a frame map. To show the desired result, it suffices to show that it lifts to a map $f:(L,\ca F)\to (\Om X,\ca U ^*X)$ in $\bd{Raney}$. Let $S$ be a saturated set. For the desired lift to exist, we have to show that
\[
f^{-1}(\{U\in \Om X\mid S\subseteq U\})\in \ca F.
\]
For $U\in \Om X$, $S\subseteq U$ if and only if $\chi_x(U)=1$ for all $x\in S$, and so:
\[
f^{-1}(\{U\in \Om X\mid S\subseteq U\})=f^{-1}(\bigcap_{x\in S}\chi_x^{-1}(1))=\bigcap_{x\in S}(\chi_x\circ f)^{-1}(1)
\]
Since by our initial assumption $(\chi_x\circ f)^{-1}(1)\in \ca F$ for each $x\in X$, and as $\ca F$ is closed under arbitrary intersections (as it is a subcolocale).
\end{proof}

Similarly as in Theorem \ref{t: duality for raney extensions}, one can prove the adjunctions in \cite{bezhanishvili23} and that in \cite{manuell15}, too, as special cases of Theorem \ref{t: main theorem in porst tholen adapted}. We provide sketches for both proofs.

\begin{theorem}\label{t: duality for skula extensions}\label{l: szd are a quasi t0 duality}
  The PFC-category $(\Op_{\ca S},\bd{Skula},\bd 2_{\ca S},\id_{\bd{2}})$ is spatializable. The adjunction 
  \[
   \mf{Dual}(\Op_{\ca S},\bd{Skula},\bd 2_{\ca S},\id_{\bd{2}})
  \]
  is such that $\Topc=\Topp$ and $\mf{Fix}(\Topc)=\Topz$. 
\end{theorem}
\begin{proof}(Sketch). 
    We show that for every topological space $X$ the Skula extension $(\Om X,\ca{SKC}^* X)$ provides the required cartesian lift. For Skula extension $(L,\ca D)$, we assume there is a function $f:|L|\to |\Om X|$ where every composite $\chi_x\circ f$ lifts. We observe that $\ca S\chi_x^*(\{1\})=b(x)$ for every $x\in X$, and so this assumption means 
    \begin{equation}\label{e: Sf^* of the points}
        \text{$\ca S f^*b(x)\in \ca D$ for every $x\in X$.}
    \end{equation}
 The function $f$ is a frame map, by the cartesian lift property of $\Om X$. We show that $f$ is a map in $\bd{Skula}$; for this it suffices to show $\ca S f^*(S)\in \ca D$ for every $S\in \ca{SKC}^*X$. If $S\in \ca{SKC}^*X$, then there is some Skula-closed set $Y\subseteq X$ such that
    \[
    S=\bigvee \{b(x)\mid x\in Y\}
    \]
    As $\ca S f^*$ is a left adjoin, it preserves joins, and so 
    \[
    \ca S f^*(S)=\bigvee \{\ca S f^*(b(x))\mid x\in Y\}.
    \]
   By \ref{e: Sf^* of the points}, $\ca S f^*(b(x))\in \ca D$ for all $x\in X$, and as $\ca D$ is a subcolocale it is closed under all joins, so $S\in \ca D$, as desired.
\end{proof}

\begin{theorem}\label{t: duality for mt}
    \label{l: mt algebras are a quasi t0 duality}
 The PFC-category $(\Op_{\ca M},\bd{MT},\bd 2_{\ca M},\id_{\bd{2}})$ is spatializable. The adjunction 
 \[
 \mf{Dual}(\Op_{\ca M},\bd{MT},\bd 2_{\ca M},\id_{\bd{2}})
 \]
 is such that $\Topc=\Topp$ and $\mf{Fix}(\Topc)=\Topz$.
\end{theorem}
\begin{proof}(Sketch). We show that for every topological space $X$ the MT-algebra $(\Om X,\ca P X)$ provides the required cartesian lift. For an MT-algebra $(L,M)$, we assume there is a function $f:|L|\to |\Om X|$ where every composite $\chi_x\circ f$ extends to a map of Boolean algebras $\overline{f_x}:M\to \bd 2$. By the cartesian lift property of $\Om X$, $f$ is a frame map. We define the desired extension of $f$ as
\[
\overline{f}(m)=\{x\in X\mid \overline{f_x}(m)=1\}
\]
for all $m\in M$. Let us show this is a map of Boolean algebras. To see $\overline{f}(\bve_i m_i)\subseteq \bigcup_i \overline{f}(m_i)$ for $m_i\in M$, we notice that if for $x\in X$ we have $\overline{f_x}(\bve_i m_i)=1$ then we must also have $\overline{f_x}(m_i)=1$ for some $i\in I$, as $\overline{f_x}$ preserves all joins. Preservation of meets is proved analogously.
\end{proof}

We aim to define a functor $\ca S_{\ca S}:(\Op_{\ca S},\bd{Skula},\bd{2}_{\ca S},\id_{\bd{2}})\to (\Op_{\ca R},\bd{Raney},\bd{2}_{\ca R},\id_{\bd{2}})$. The assignment we will describe is studied in \cite{suarez26zero}, but the correspondence is heavily based on the description of Raney extensions as collections of sublocales. We adapt the relevant results to our setting. The definition of the functor relies on the map $\mathit{ker}:\ca S L\to \mf{Filt}(L)$ from \ref{l: the properties of ker}. The assignments $L\mapsto \mf{Filt}(L)$ and $L\mapsto \ca S L$ are both functorial; for a frame map $f:L\to M$, we note that the left adjoint of the unique coframe map $\mf{Filt}(f)$ extending $f$ is $f^{-1}$. Then, the assignments $f\mapsto \ca S f^*$ and $f\mapsto f^{-1}$ extend the assignments $\ca S(-)$ and $\mf{Filt}(-)$ to functors:
\begin{align*}
  \ca S:\bd{Frm}\to \bd{CoLoc}  &&   \mf{Filt}:\bd{Frm}\to \bd{CoLoc} 
\end{align*}
 We now strengthen the result in \cite{moshier20} and show that the map $\mathit{ker}:\ca S L\to \mf{Filt}(L)$ is a natural transformation $\mathit{ker}:\ca S\implies \mf{Filt}$.

\begin{proposition}\label{p: ker is a natural transformation}
For every frame map $f:L\to M$, the following square commutes in $\bd{CoLoc}$. 
\[
\begin{tikzcd}
\ca S L
\ar[r,"\mathit{ker}_L"]
& \mf{Filt}(L)
\\
\ca S M
\ar[u,"\ca S f^*"]
\ar[r,"\mathit{ker}_M"]
& \mf{Filt}(M)
\ar[u,"f^{-1}"]
\end{tikzcd}
\]
Then, there is a natural transformation $\ca S\implies \mf{Filt}$ evaluated as $\mathit{ker}_L:\ca S L\to \mf{Filt}(L)$ at each $L\in \Frm$.
\end{proposition}
\begin{proof}
 For a sublocale $S\in \ca S M$, we have 
 \begin{align*}
     f^{-1}(\mathit{ker}_M(S)))&=f^{-1}(\{x\in L\mid S\subseteq \neg e^{\ca S}_L(x) \})\\
     & =\{x\in L\mid S\subseteq \neg e^{\ca S}_L(f(x)) \}\\
     & =\{x\in L\mid S\subseteq \ca S f(\neg e^{\ca S}_L(x))\}& \text{ by definition of $\ca S f$}\\
     & =\{x\in L\mid \ca S f^*(S)\subseteq \neg e^{\ca S}_L(x)\}& \text{ by adjointness}\\
     & =\mathit{ker}_L(\ca S f^*(S)).\qedhere
 \end{align*}

\end{proof}

\begin{lemma}\label{l: assignment from skula to raney is well defined}
    If $L$ is a frame and $\ca D\subseteq \ca S L$ is a subcolocale with $\ca F L\subseteq \ca D$, 
    \[
    \fe(L)\subseteq \mathit{ker}_L[\ca D].
    \]
    Then, $(L,\mathit{ker}_L[\ca D])$ is a Raney extension.
\end{lemma}
\begin{proof}
Since $\ca F L$ contains $\neg e_L^{\ca S}(a)\in \ca D$ for all $a\in L$, by Example \ref{e: kernel of an open sublocale} $\mathit{ker}_L[\ca D]$ contains at least the principal filters. As it is a coframe map, its image is a subcolocale of $\mf{Filt}(L)$, and so by Lemma \ref{l: fe is smallest sl containing the principals} it must contain $\fe(L)$.  
\end{proof}
\begin{example}\label{e: saturated of a szdbifrm}
   We claim that, for a space $X$, the Raney extension $\ca S_{\ca S}(\Om X,\ca{SKC}^*X)$ is $(\Om X,\ca U^*X)$. An arbitrary element of $\mathit{ker}_{\Om X}[\ca{SKC}^*X]$ can be written as follows for some $Y\subseteq X$:
   \begin{align*}
       \mathit{ker}_{\Om X}(\bigvee \{b(x)\mid x\in Y\})&=\bigcap\{\mathit{ker}_{\Om X}(b(x))\mid x\in Y\}\\
       &=\bigcap_{x\in Y}\{U\in \Om X\mid b(x)\subseteq \neg e^{\ca S}_{\Om X}(U)\}\\
       &= \bigcap_{x\in Y}\{U\in \Om X\mid x\in U\}\\
       &=\{U\in \Om X\mid Y\subseteq U\},
   \end{align*}
which is in $\ca U^*X$; furthermore, all elements of $\ca U^*X$ are of this form for some $Y\subseteq X$.

\end{example}
 We now show that the assignment from Lemma \ref{l: assignment from skula to raney is well defined} can be extended to a functor.

\begin{lemma}\label{l: assignment from skula to raney is well defined 2}
    If $f:(L,\ca D)\to (L,\ca E)$ is a map of Skula extensions, $f:(L,\mathit{ker}_L[\ca D])\to (M,\mathit{ker}_M[\ca E])$ is a map of Raney extensions.
\end{lemma}
\begin{proof}
 Suppose that $E\in \ca E$. By Proposition \ref{p: ker is a natural transformation}, $f^{-1}(\mathit{ker}_M(E))=\mathit{ker}_L(\ca S f^*(E))$. Since $f$ is a morphism in $\bd{Skula}$, $\ca S f^*(E)\in \ca D$. The desired result follows.
\end{proof}

Lemmas \ref{l: assignment from skula to raney is well defined} and \ref{l: assignment from skula to raney is well defined 2} show that there exist a functor as desired, which we call $\ca S_{\ca S}:\bd{Skula}\to \bd{Raney}$.

\begin{proposition}\label{p: SS is morphism in SFC}
 The triple $(\ca S_{\ca S},\id_{\Op_{\ca S}},\id_{\bd{2}})$ is a morphism of SFC-categories. Furthermore, the map $\mf{Dual}(\ca S_{\ca S},\id_{\Op_{\ca S}},\id_{\bd{2}})$: 
      \begin{equation*}
        \begin{tikzcd}[row sep=huge,column sep=large]
          \bd{Skula}
            \ar[r,shift left=2,"\pt_{\ca S}"{above}] 
            \ar[d,"\ca S_{\ca S}"]
          & \bd{Top}^{\mathsf{op}}
            \ar[l,shift left=2,"\ca{SKC}"{below}]
            \ar[l,phantom,"\dashv"{rotate=-90,anchor=center}]
            \ar[d,equal] \\
          \bd{Raney}
            \ar[r,shift left=2,"\pt_{\ca R}"{above}]
          & \bd{Top}^{\mathsf{op}}
            \ar[l,shift left=2,"\ca U"{below}]
            \ar[l,phantom,"\dashv"{rotate=-90,anchor=center}]
       \end{tikzcd}
      \end{equation*}
      is in $\bd{SADJ}$.

\end{proposition}
\begin{proof}
    Conditions \ref{Mxx} and \ref{M2} are immediate. For \ref{M3}, note that, as seen in Example \ref{e: saturated of a szdbifrm}, for every space $X$ we have $\ca S_{\ca S}(\Om X,\ca{SK}^*(X))\cong(\Om X,\ca U^*X)$. 
\end{proof}

The connection between MT-algebras and Raney extensions has been explored in \cite{bezhanishvili26raneyandmt}. For an MT-algebra $(L,M)$ we call $\ca{S}_{\ca M}(M)$ the subcoframe of $M$ obtained by closing the collection $L$ under all meets. By an argument similar to that used for the isomorphism $\ca{U}^*(X)\cong \ca{U}(X)$, one shows that there is a collection of filters of $L$ isomorphic to $\ca S_{\ca M}(M)$, which we call $\ca{S}_{\ca M}^*(M)$. In \cite{bezhanishvili26raneyandmt}, a functor $\ca{S}_{\ca M}:\bd{MT_R}\to \bd{Raney}$ is thus defined on objects, up to isomorphism, from a wide subcategory of $\bd{MT_R}$, which is then shown to be part of an equivalence. The definition of this functor can be extended to all $\bd{MT}$.
\begin{lemma}\label{l: functor sat}
The assignment $(L,M)\mapsto (L,\ca{S}^*_{\ca M}(M))$ extends to a functor $\ca{S}_{\ca M}:\bd{MT}\to \bd{Raney}$.
\end{lemma}

\begin{example}\label{e: Sm maps PX to UX}
    For a space $X$, the Raney extension $\ca S_{\ca M}(\Om X,\ca P X)$ is, up to isomorphism, $(\Om X,\ca U^ *X)$.
\end{example}

\begin{proposition}\label{p: SM is a morphism in SFC}
 The triple $(\ca S_{\ca M},\id_{\Op_{\ca M}},\id_{\bd{2}})$ is a morphism of SFC-categories. Furthermore, the morphism $\mf{Dual}(\ca S_{\ca M},\id_{\Op_{\ca M}},\id_{\bd{2}})$: 
\begin{equation*}
  \begin{tikzcd}[row sep=huge,column sep=large]
    \bd{MT}
      \ar[r,shift left=2,"\pt_{\ca M}"{above}] 
      \ar[d,"\ca S_{\ca M}"]
    & \bd{Top}^{\mathsf{op}}
      \ar[l,shift left=2,"\ca P"{below}]
      \ar[l,phantom,"\dashv"{rotate=-90,anchor=center}]
      \ar[d,equal] \\
    \bd{Raney}
      \ar[r,shift left=2,"\pt_{\ca R}"{above}]
    & \bd{Top}^{\mathsf{op}}
      \ar[l,shift left=2,"\ca U"{below}]
      \ar[l,phantom,"\dashv"{rotate=-90,anchor=center}]
 \end{tikzcd}
\end{equation*}
is in $\bd{SADJ}$.

\end{proposition}
\begin{proof}
Conditions \ref{Mxx} and \ref{M2} are immediate. Property \ref{M3} follows from Example \ref{e: Sm maps PX to UX}, as the example shows that $\ca{S}_{\ca M}$ preserves cartesian lifts of families of characteristic functions. Commutativity of the left adjoint direction of the diagram is the content of Lemma 5.5 of \cite{BM+25}.
\end{proof} 

\section{Fiber-initial and fiber-terminal objects}

For an FC-category \( \Op \colon \ca C \to \Frm \), we say an
object \( C \) in \( \ca C \) is \textit{fiber-initial} (respectively,
\textit{fiber-terminal}) if \( (C,\id) \) is a bottom (top) element in
$\Op^{-1}(\Op (C))$.

In the sequel, for every fiber-initial object \( I \), and every object \(
(C,\theta^C) \) in \( \Op^{-1}(\Op(I)) \), we denote by \( \iota_C \colon I
\to C \) the underlying morphism of the universal map \( (I,\id) \to
(C,\theta^C) \). Similarly, for every fiber-terminal object \( T \) and every
object \( (C,\theta^C) \) in \( \Op^{-1}(\Op(T)) \), the underlying morphism
of the universal map \( (C,\theta^C) \to (T,\id) \) is denoted by \( \tau_C
\colon C \to T \). If the object \( C \) is clear from context, we will omit
the subscripts.

\begin{lemma}\label{l: iota and tau mono and epi}
  If $I$ is fiber-initial, and $(C,\theta^C) \in \Op^{-1}(\Op(I))$, then
  $\iota:I\to C$ is a monomorphism and an epimorphism. The same holds for
  $\tau:C\to T$ when $T$ is fiber-terminal and \( (C,\theta^C) \in
  \Op^{-1}(\Op(T)) \).
\end{lemma}
\begin{proof}
  This follows at once by our general observations for \( \Frm \)-concrete
  categories.
\end{proof}

We recall the following characterizations of $T_D$- and sober spaces. This is a rephrasing of Proposition 4.3 in \cite{banaschewskitd}.

\begin{theorem}[\cite{banaschewskitd}, Proposition 4.3]\label{t: sober and td as fiber minimal and maximals}
Consider the FC-category $\Om:\Topz^{\mf{op}}\to \Frm$. 
\begin{itemize}
    \item The spaces which are minimal in their fiber coincide with the sober spaces. 
    \item The spaces which are maximal in their fiber coincide with the $T_D$-spaces. 
\end{itemize}

\end{theorem}

The result exhibits a symmetry between sobriety and the $T_D$-property. We now note that sober spaces are not just minimal, they are fiber-initial, with the sobrification map $\sigma_X:X\to \mf{sob}(X)$ of a $T_0$-space being the universal map witnessing this. \begin{proposition}\label{p: td points of a frame}
    For a frame $L$, we have 
    \[
    \bd{Frm_D}(L,\bd 2)=\bd{Frm}_{\ca E}(L,\bd 2)=\bd{Frm}_{\ca{LE}}(L,\bd 2).
    \]
\end{proposition}
\begin{proof}
The equality $\bd{Frm_D}(L,\bd 2)=\bd{Frm}_{\ca{LE}}(L,\bd 2)$ is given by Corollary 3.3 and 3.4 in \cite{arrieta20}. The equality $\bd{Frm_D}(L,\bd 2)=\bd{Frm}_{\ca{E}}(L,\bd 2)$ is given by Lemmas 7.10 and 7.11 of \cite{suarez25raney}. 
\end{proof}

The following is a well-known result in pointfree topology, we give a new proof based on exactness. The following is Lemma 4.11 in \cite{suarez25raney}.
\begin{lemma}\label{l: a particular exact meet}
    Let $L$ be a frame, $x\in L$. The meet
    \[
    \bigwedge\{y\in L\mid x\leq y,x\neq y\}
    \]
    is exact.
\end{lemma}

\begin{lemma}\label{l: exact cp filters are always induced.}
For a $T_0$-space $X$, if $f\in \bd{Frm_D}(L,\bd 2)$, there is some $x\in X$ with $f=\chi_x$.
\end{lemma}
\begin{proof}
    We use the equality $\bd{Frm_D}(L,\bd 2)=\bd{Frm}_{\ca E}(L,\bd 2)$ in Proposition \ref{p: td points of a frame}. Suppose that $f:L\to \bd 2$ is exact. We consider the subset of $X$
    \[
    S:=\bigcap\{U\in \Om X\mid f(U)=1\}\cap \bigcap \{U^c\mid U\in \Om X\mid f(U)=0\}.
    \]
   Towards contradiction, assume this is empty. Then
   \[
   \bigcap f^{-1}(1)\subseteq \bigcup f^{-1}(0).
   \]
   From which we obtain
   \[
   \bigcap \{\bigcup f^{-1}(0)\cup U\mid U\in f^{-1}(1)\}= \bigcup f^{-1}(0).
   \]
   The intersection on the left is an open set, so it is the meet in $\Om X$. Furthermore, the family whose meet is computed can be rewritten as
   \[
   \{U\in \Om X\mid \bigcup f^{-1}(0)\subseteq U, U\nsubseteq \bigcup f^{-1}(0)\}.
   \]
   So, by Lemma \ref{l: a particular exact meet}, the meet on the left is exact, and as $f$ is exact this means that it is in $f^{-1}(1)$, but this is a contradiction, as it equals $\bigcup f^{-1}(0)$. Let $x\in S$. We observe that whenever $x\in S$ we have $f\leq\chi_x$; conversely, if $f(U)=0$ for some open set $U$, then $x\in U^c$, and so $\chi_x(U)=0$. Then, $f=\chi_x$, as desired. We also observe that as $X$ is $T_0$ and $f=\chi_x$, $S$ must be a singleton.
\end{proof}

For every $T_0$ space $X$, then, there is a subspace inclusion $\tau_D:\pt_D \Om X\hookrightarrow X$. We may now show that $T_D$-spaces are fiber-terminal.
\begin{lemma}
For the FC-category $(\Om,\Topz^{\mf{op}})$, $T_D$-spaces are terminal in their fibers.
\end{lemma}
\begin{proof}
Let $X$ be a $T_D$-space; we identify it with $\pt_D \Om X$. We consider a $T_0$-space $Y$ such that there is an isomorphism $i:\Om Y\cong \Om X$. By Lemma \ref{l: exact cp filters are always induced.}, there is a subspace inclusion $\tau_Y:\pt_D \Om Y\hookrightarrow Y$. As $\Om Y$ is $T_D$-spatial, $\Om(\tau_Y)$ is a frame isomorphism, and so this corresponds to a morphism in the fiber of $\Om^{-1}(\Om Y)$. Precomposing this with the homeomorphism $\pt_D(i):\pt_D \Om X\cong \pt_D \Om Y$ yields a map as desired.
\end{proof}
 We can then formulate a strong version of Theorem \ref{t: sober and td as fiber minimal and maximals}.
 \begin{theorem}\label{t: sober and td as fiber initial and terminals}
     Consider the FC-category $\Om:\Topz^{\mf{op}}\to \Frm$. 
\begin{itemize}
    \item The spaces which are initial in their fiber coincide with the sober spaces. 
    \item The spaces which are terminal in their fiber coincide with the $T_D$-spaces. 
\end{itemize}
 \end{theorem}

The sobrification map is also a witness of every fiber having an initial object. The symmetry in Theorem \ref{t: sober and td as fiber initial and terminals} now breaks, since not all fibers have terminal objects, as the next example shows.

\begin{example}\label{e: some fibers of omega have no terminal objects}
Any spatial frame $L$ which is not $T_D$-spatial gives a counterexample. If the fiber of a spatial frame $L$ has a terminal object, then this is $T_D$ by Theorem \ref{t: sober and td as fiber initial and terminals}; by definition of fiber this is a space with $L$ as its frame of opens. We give a concrete example. Consider the frame $(\mathbb{N}\cup \{\infty\})^{\mf{op}}\times \bd 2$. All meets are exact, and so exact filters are the principal ones. In particular, the filter
\[
L{\setminus}\{(\infty,1),(\infty,0)\}
\]
is completely prime, but not exact. This is also the only completely prime filter containing $(1,1)$ and omitting $(\infty,1)$, showing that $L$ is not $T_D$-spatial.
\end{example}

\subsection{Natural adjunctions induced by fiber-initials and fiber-terminals}\label{ss nat adj for fiber initials and terminals}

We now want to consider dualities arising from the initial and the terminal
objects of fibers. We define $I_{\ca I}:\ca{C}_{\ca{I}}\to \ca C$ to be the full subcategory
inclusion of fiber-initial objects, and $I_{\ca T}:\Ct\to \ca C$ similarly for
fiber-terminal ones. For an SFC-category \( (\Op,\ca C, \Tc, \iota_{\ca C})
\), one can respectively view the terminal and initial object of $\Op^{-1}(L)$
(when they exist) as the smallest and largest pointfree spaces having $L$ as
its frame of opens, taking contravariance of \( \ptc \colon \ca C \to \Topp \)
into account. This intuition is made precise by the following result.

\begin{proposition}
  \label{p: boundaries for spectrum}
  Let \( (\Op,\ca C, \Tc, \iota_{\ca C})\) be an SFC-category. Let  $C\in \ca{C}$ be
  such that $\Op^{-1}(\Op(C))$ has both an initial object $I$ and a terminal
  object $T$. The maps
\begin{equation*}
  \begin{tikzcd}[row sep=large,column sep=large]
    \ptc(T) \ar[r,"\ptc \tau"]
      & \ptc(C) \ar[r,"\ptc \iota"]
      & \ptc(I),
  \end{tikzcd}
\end{equation*} 
are subspace embeddings.
\end{proposition}
\begin{proof}
  The functor $\ptc$ sends $\iota$ to the precomposition map 
  \[
    -{\circ} \iota:\ca C(C,\Tc)\to \ca C(I,\Tc). 
  \]
  This is injective, as $\iota$ is an epimorphism by Lemma \ref{l: iota and
  tau mono and epi}. By definition of the topologies on the two spaces, this
  is a subspace embedding. The argument for terminal objects is analogous.
\end{proof}
 
We now want to extend the restriction of $\ptc$ to the categories of
fiber-initial and of fiber-terminal objects to a natural adjunction, and view
them both as part of a broader setting. If $\Tc\in \Ct$, we call $\Topt$ the
category $\Topp_{\Ct}$ defined as in Subsection~\ref{ss: defn topc ptc omc}.
In case $\Tc\in \Ci$, we define $\Topp_{\ca I}$ similarly. 

For any SFC-category where $\Tc$ is terminal, and $X\in \Topp_{\ca T}$, we
call $\Omt X$ the object in $\Ct$ providing the cartesian lift for the family
of characteristic functions of points of $X$. We define $\Omi X$ similarly for
each $X\in \Topi$.

We define an SFC-category $(\Op,\ca C,\Tc,\iota_{\ca C})$ to be \emph{total}
if: 
\begin{enumerate}[label=(T\arabic*),ref=T\arabic*]
  \item 
    \label{T2} 
    $\Topc=\mf{Fix}(\Topc)=\Topz$;
  \item 
    \label{T1}
    The objects $\bd 2_{\ca C}$ and $\Omc \mathbb{S}$ are both fiber-initial
    and fiber-terminal (thus, the only objects in their fiber, up to
    isomorphism);
  \item 
    \label{T3}
    The underlying FC-functor of $(\Op,\id,\iota_C) \colon (\ca C,\Op_{\ca
    C},\Tc,\iota_C) \to (\Frm, \id, \bf 2, \id)$ has a left FC-adjoint;
  \item 
    \label{T4}
    $\ptc[\ca C_{\ca T}]\subseteq \Topt$.
\end{enumerate}

For the next result, we highlight that by Lemma~\ref{l: when does fc-functor
have adjoint}, \eqref{T3} yields an adjunction \( \ca F \dashv \Op \) whose
unit \( \nu \) is invertible. 

\begin{proposition}
  \label{p: fiber-initial fc-corefl}
  If \( (\ca C, \Op, \Tc, \iota_{\ca C}) \) is an SFC-category such that
  \eqref{T3} holds, then the PFC-functor 
  \begin{equation*}
    (\Op I_{\ca I},\id,\nu_{\bf 2}) \colon 
      (\Ci, \Op I_{\ca I}, \ca F(\bf 2), \nu_{\bf 2}) \simeq
      (\Frm, \id, \bf 2, \id)
  \end{equation*}
  is an equivalence of PFC-categories.
\end{proposition}

\begin{proof}
  First, we note that \( (\ca F(L),\nu_L^{-1}) \) is the initial object of \(
  \Op^{-1}(L) \) for every frame \( L \): for if \( C \) is an object in \(
  \ca C \) and \( \theta \colon \Op(C) \cong L \) is an isomorphism, then the
  the morphism corresponding to \( \theta^{-1} \) under \( \ca C(\ca F(L),C)
  \cong \Frm(L,\Op(C)) \) is precisely the universal map \( \ca F(L) \to C \). 

  In particular, it follows that \( \ca F(L) \) is fiber-initial for all
  frames \( L \), and for \( C \in \Ci \), we have that the counit \( \delta
  \) of \( \ca F \dashv \Op \) is invertible at \( C \). Thus, we obtain a
  natural isomorphism \( \lambda_{I_{\ca I}} \colon \ca F\Op I_{\ca I} \to
  I_{\ca I} \). Since \( \ca F \) and \( I_{\ca I} \) are fully faithful, it
  follows that \( \Op I_{\ca I} \) is fully faithful.

  We conclude our proof by noting the composite \( \Op I_{\ca I} \) is
  essentially surjective, a fact witnessed by the unit \( L \cong \Op(\ca
  F(L)) \), since \( \ca F(L) \) is fiber-initial. 
\end{proof}

From this result, we immediately deduce that 
\begin{theorem}
  \label{t: the initial duality}
  If \( (\ca C, \Op, \Tc, \iota_{\ca C}) \) is an SFC-category such that
  \eqref{T3} holds, then:
  \begin{itemize}
    \item
      \( (\Ci, \Op I_{\ca I}, \tilde{\ca F}(\bf 2), \nu_{\bf 2}) \) is an
      SFC-category.
    \item
      \( \Topi = \Topp \),
    \item
      \( \mathsf{Fix}_{\ca I}(\Topi) = \mathbf{Sob} \),
    \item
      \( \pt_{\ca I} \cong \ptc I_{\ca I} \cong \pt \Op I_{\ca I} \).
  \end{itemize}
\end{theorem}
\begin{proof}
  The first item follows by the fact that PFC-equivalences are
  SFC-equivalences, while the remaining three items are consequences of
  Lemmas~\ref{l: m3 implies all topc are topd} and~\ref{l: induced spec iso},
  respectively.
\end{proof}

We can, under certain conditions, still partially obtain the conclusions of
Theorem~\ref{t: the initial duality}, even when \eqref{M3} does not hold:

\begin{lemma}\label{l: definition omegai}
  Let $(\Op,\ca C, \Tc, \iota_{\ca C})$ be an SFC-category.  If $\Tc\in \Ci$,
  for all $X\in \Topi$ the fiber of $\Om X$ has an initial object, and this is
  $\Om_{\ca I} X$. Furthermore, $\pt_{\Ci} C=\ptc C$ for all $C\in \Ci$.
\end{lemma}
\begin{proof}
  For the first item, it suffices to observe that by Lemma \ref{l: pfc-cat
  implies zeta iso}, the object $\Omi X$ must be in the fiber of $\Om
  X$. Since fiber-initial objects form full subcategories of $\ca C$, for any
  object $C\in \Ci$ and whenever $\Tc$ is initial, the assignment $f\mapsto
  I_{\ca I}f$ is a bijection $\Ci(C,\Tc)\cong \ca C(I_{\ca I}C,I_{\ca I}\Tc)$. 
\end{proof}

Now, we describe the properties of terminal objects of the fibers.

\begin{lemma}\label{l: definition omegat}
  Let $(\Op,\ca C, \Tc, \iota_{\ca C})$ be an SFC-category. If $\Tc\in \Ct$,
  then $\Om_{\ca T} X$ is naturally isomorphic to $\Om_{\ca C}X$ for all $X\in
  \Topp_{\ca T}\cap \Topc$. Furthermore, $\pt_{\Ct} C=\ptc C$ for all $C\in
  \Ct$.
\end{lemma}

\begin{proof}
  For the second, suppose that $X\in \Topp_{\ca T}$. Once again, by
  Lemma~\ref{l: pfc-cat implies zeta iso}, $\Op \Om_{\ca T}X$ is isomorphic to
  $\Om X$. By Lemma \ref{l: mutually inverse frame isos that lift}, to prove
  the claim it is enough to show that the identity on $\Om X$ lifts to maps 
  \begin{align*}
    \tau:\Omc X\to \Om_{\ca T} X   && \tau': \Om_{\ca T} X\to \Omc X.
  \end{align*} 
  The existence of the first map follows from terminality of $\Omt X$.
  Unraveling the definition of cartesian lift in the case of an identity
  yields that the identity $\id_{|\Om X|}$ lifts to a map $\tau':\Omc X\to
  \Om_{\ca T} X$ whenever each characteristic function $\chi_x:|\Om X|\to 2$
  lifts to a map $\overline{\chi^T_x}:\Om_{\ca T} X\to \Tc$. This is ensured
  by the universal property of $\Om_{\ca T} X$. The proof that
  $\pt_{\Ct}=\ptc$ is analogous to the same statement for fiber-initials.
\end{proof}

With the above result, we are able to characterize the spaces in $\Topp_{\ca
T}$.
\begin{proposition}\label{p: spaces in topt}
  Let $(\Op,\ca C,\Tc,\iota_{\ca C})$ be an SFC-category where $\Tc$ is
  fiber-terminal. For $X\in \Topc$, the following are equivalent.
  \begin{enumerate}
    \item 
      \label{topt 1}
      $X\in \Topp_{\ca T}$.
    \item 
      \label{topt 2}
      $\Omc X$ is fiber-terminal.
    \item 
      \label{topt 3}The family $\{\chi_x:|\Om X|\to 2\mid x\in X\}$ has a
      lift in $\Ct$. 
    \end{enumerate}
\end{proposition}
\begin{proof}
   Item \ref{topt 1} implies Item \ref{topt 2} by Lemma \ref{l: definition omegat}. If Item \ref{topt 2} holds, then the lift $\overline{\chi_x}:\Omc X\to \Tc$ is in $\Ct$, and so Item \ref{topt 3} follows. Let \ref{topt 3} hold, and let $T\in \Ct$ be the object providing a lift
   \[
   \{\overline{\chi_x}^T:T\to \Tc\mid x\in X\}
   \]
   for the family of characteristic functions. Let us show it is cartesian. We take $U\in \Ct$, and assume that there is a function $f:|\Op U|\to |\Om X|$ such that $\chi_x{\circ} f$ lifts for every $x\in X$.
   By the universal property of $\Omc X$, there is $\overline{f}:U\to \Omc X$ lifting $f$. As $\Op (\tau)$ is an isomorphism, the map $\tau {\circ} \overline{f}$ is the required lift.
\end{proof}

We observe that by the characterization of spaces in $\Topt$ of Proposition \ref{p: spaces in topt} above, in the definition of total SFC-category, axiom \ref{T4} may be equivalently replaced with 
  
\begin{enumerate}[label=(T\arabic*)',ref=T\arabic*']\setcounter{enumi}{3}
  \item 
    $\Omc \ptc C\in \Ct$ whenever $C\in \Ct$.
\end{enumerate}

We turn to proving some basic consequences of totality.

\begin{proposition}\label{p: spaces in topi}
  If $(\Op,\ca C,\Tc,\iota_{\ca C})$ be a total SFC-category, then,
  $\Topt=\mf{Fix}_{\ca T}(\Topt)$.
\end{proposition}
\begin{proof}
  Consider $X\in \Topt$. By \ref{T2}, the unit
  $\sigma_X^{\ca C}:X\to \ptc \Omc X$ is a homeomorphism. By Lemma \ref{l:
  definition omegat}, this is also the unit for $\mf{Dual}(\Op I_{\ca
  T},\Ct,\Tc,\iota_{\ca C})$, and so $X\in \mf{Fix}_{\ca T}(\Topt)$.
\end{proof}

\begin{example}
   The SFC-category $(\Topz^{\mathsf{op}},\Om,\{*\},\id_{\bd{2}})$ satisfies
   \eqref{T3}, with a left PFC-adjoint $\pt:\Frm\to \Topz^{\mf{op}}$. As seen
   in Example \ref{e: the inclusion of sobers is not continuous}, the
   inclusion $\bd{Sob}^{\mf{op}}\subseteq \Topz^{\mf{op}}$ is a PFC-functor
   which does not satisfy \ref{M3}. However, by Theorem \ref{t: the initial
   duality}, the composite $\pt\Om:\Topz^{\mf{op}}\to \bd{Sob}^{\mf{op}}$
   does, and so it is an SFC-functor.
\end{example}

We now study the natural adjunction for terminal objects.

\begin{theorem}\label{t: the terminal duality}
  Let $(\Op,\ca C,\Tc,\id_{\bd{2}})$ be a total SFC-category. Then:
  \begin{enumerate}
    \item 
      $(\Op I_{\ca T},\Ct,\Tc,\iota_{\ca C})\in \SFC$.
    \item 
      \label{terminal duality 2}
      $(I_{\ca T},\id_{\Op I_\ca T},\id_{\Tc}):(\Op I_{\ca T},\Ct,\Tc,\iota_{\ca C})\hookrightarrow (\Op,\ca C,\Tc,\iota_{\ca C})$ is a SFC-functor.
    \item 
      \label{terminal duality 3}
      The morphism $\bd{Dual}(I_{\ca T},\id_{\Op I_\ca T},\id_{\Tc})$:
      \begin{equation*}\label{d: morphism from CT to C dualities}
        \begin{tikzcd}[row sep=huge,column sep=large]
          \ca{C}_{\ca T}
              \ar[r,shift left=2,"\ptc"{above}] 
              \ar[d,"I_{\ca T}"]
          & \bd{Top}_{\ca T}^{\mathsf{op}}
              \ar[l,shift left=2,"\Omc"{below}]
              \ar[l,phantom,"\dashv"{rotate=-90,anchor=center}]
              \ar[d,hookrightarrow] \\
          \ca C
              \ar[r,shift left=2,"\ptc"{above}]
          & \bd{Top}^{\mathsf{op}}
              \ar[l,shift left=2,"\Omc"{below}]
              \ar[l,phantom,"\dashv"{rotate=-90,anchor=center}]
        \end{tikzcd}
      \end{equation*}
      is in $\bd{SADJ}$.
  \end{enumerate}
\end{theorem}

\begin{proof}
  By \ref{T1} and \ref{T4}, the triple is an SFC-category (in particular,
  \ref{T1} shows that $\mathbb{S}\in \Topt$ as it satisfies the
  characterization in Proposition \ref{p: spaces in topt}). the inclusion
  immediately satisfies \ref{M2} and \ref{Mxx}. For $X\in \Topt$, by \ref{T2}
  the family of its characteristic functions has a cartesian lift in $\ca C$;
  by Lemma \ref{l: definition omegat} this is the same as the cartesian lift
  in $\Ct$, and so the inclusion $I_{\ca T}$ satisfies \eqref{M3}. Item 2 is
  then proven. Finally, as $I_{\ca T}$ is full and faithful, Proposition
  \ref{p: the zeta generalized} establishes that the maps in the diagram yield
  a strong adjunction morphism. 
\end{proof} 

\begin{remark}\label{r: why Ci into C does not satisfy M3}
  We have seen that the composition $\ca F\Op:\ca C\to \Ci$ gives an
  SFC-functor. The condition that the inclusion $I_{\ca I}:\Ci\to \ca C$ is a morphism
  is somewhat unnatural: by \ref{M3}, and by Lemma \ref{l: definition omegai},
  this would imply that for $X\in \Topi$ the object $\Omc X$ is a free object,
  namely $\ca F(\Om X)$. But spatialization is often akin to a quotient (this
  is also confirmed by $\Op(\eta^{\ca C})$ being a frame surjection, see Lemma
  \ref{l: xi is a subspace embedding and basic facts about the units}), and so
  it will not in general be a free object. In contrast, we will see below that
  the inclusion $\Ct\to \ca C$ is always a SFC-functor. Theorem \ref{t: the
  initial duality} and \ref{t: the terminal duality}, then, together suggest
  that the SFC-category of fiber-initials is akin to a quotient of $(\Op,\ca
  C,\Tc,\iota_{\ca C})$, whereas that of fiber-terminals is akin to a
  subobject. 
\end{remark}

If $(\Op,\ca C,\Tc,\iota_{\ca C})$ is a total SFC-category, we define the
category $i_{\ca T}:\bd{Frm}_{\ca{T}}\subseteq \Frm$ as the inclusion into
$\Frm$ of the essential image of $\Op I_{\ca T}:\Ct\to \Frm$, and call $\Op
I_{\ca T}'$ the corestriction of this functor to its essential
image\footnote{The analogue for fiber-initial objects is simply all of $\Frm$,
by \ref{T3}, and so in that setting the coming discussion becomes trivial.}.
For a frame $L\in \Frm_{\ca T}$ we define $\pt_{\ca T}(L)$ to be the set
$\Frm_{\ca T}(L,\bd 2)$, topologized in the standard way. 
\begin{proposition}\label{p: adding frmt to the picture}
    If $(\Op,\ca C,\Tc,\iota_{\ca C})\in \SFC$ is total:
    \begin{enumerate}
        \item $(i_{\ca T},\Frm_{\ca T},\Tc,\iota_{\ca C})\in \SFC$;
        \item $(\Op I'_{\ca T},\id_{\Op I_{\ca T}'},\id_{\Tc}):(\Op I_{\ca T},\Ct,\Tc,\iota_{\ca C})\cong (i_{\ca T},\Frm_{\ca T},\Tc,\iota_{\ca C})$ is a morphism in $\SFC$;
        \item There is a diagram in $\bd{SADJ}$:
\[
\begin{tikzcd}[row sep=large]
     \mf{Dual}(i_{\ca T},\Frm_{\ca T},\mathbf{2},\iota_{\ca C})
     \\
     \mf{Dual}(\Op I_{\ca T},\Ct,\mathbf{2},\iota_{\ca C})
     \ar[u,"\mf{Dual}(\Op I_{\ca T}'{,}\id_{\Op I_{\ca T}'}{,}\id_{\Tc})"]
      \ar[d,"\mf{Dual}(I_{\ca T}{,}\id_{\Op I_{\ca T}}{,}\id_{\Tc})"]
     \\
     \mf{Dual}(\Op,\ca C,\mathbf{2},\iota_{\ca C})
\end{tikzcd}
\]      
\end{enumerate}
\end{proposition}
\begin{proof}
 $\Op I_{\ca T}':\Ct\to \Frm_{\ca T}$ is an equivalence of categories. Then, as $(\Op I_{\ca T},\Ct,\Tc,\iota_{\ca C})$ is in $\SFC$ by Theorem \ref{t: the terminal duality}, $(i_{\ca T},\Frm_{\ca T},\Tc,\iota_{\ca C})\in \SFC$. The inclusion $i_{\ca T}$ is, up to isomorphism, the inclusion $\Op I_{\ca T}:\Ct\to \Frm$, and this is a morphism in $\SFC$ by Theorem \ref{t: the terminal duality}. Since $\Op I_{\ca T}'$ is full, by Proposition \ref{p: the zeta generalized} the map $\mf{Dual}(\Op I_{\ca T}'{,}\id_{\Op I_{\ca T}'}{,}\id_{\Tc})$ is both left and right; while $\mf{Dual}(I_{\ca T}{,}\id_{\Op I_{\ca T}}{,}\id_{\Tc})$ is both left and right by Theorem \ref{t: the terminal duality}.
\end{proof}

We expand the diagram in Proposition \ref{p: adding frmt to the picture} above in the case where every fiber has a terminal object (hence $i_{\ca T}$ is wide), adding also the adjunction $\mf{Dual}(\mathds{1}_{\Frm},\Frm,\bd 2,\id_{\bd{2}})$ to the picture. One obtains a faithful pointfree representation of $\Topp_{\ca T}$ simultaneously a full subcategory of $\ca C$ and a wide subcategory of $\Frm$. Note that $\mf{Dual}(i_{\ca T},\id_{i_{\ca{T}}},\id_{\bd{2}})$ is in $\bd{RADJ}$, but not in general in $\bd{SADJ}$.
\begin{figure}[htbp]\label{d: duality with frmt}
    \centering\
    \begin{tikzcd}[row sep=huge, column sep=large]
        \ca{C}
            \ar[r,shift left=2,"\ptc"{above}]
            \ar[r,phantom,"\dashv"{rotate=-90,anchor=center}]
        & \bd{Top}_{\ca C}^{\mathsf{op}}
            \ar[l,shift left=2,"\Omc"{below}] \\
        \ca{C}_{\ca{T}}
            \ar[u,hookrightarrow,red,dashed]
            \ar[r,shift left=2,"\ptc"{above}]
            \ar[r,phantom,"\dashv"{rotate=-90,anchor=center}]
            \ar[d,"\Op{} I_{\ca T}'",no head,dotted,swap]
        & \bd{Top}_{\ca{T}}^{\mathsf{op}}
            \ar[u,hookrightarrow,red,dashed]
            \ar[l,shift left=2,"\Omc"{below}]
             \ar[d, equal] \\
        \bd{Frm}_{\ca{T}}
            \ar[r,shift left=2,"\ptt"{above}]
            \ar[r,phantom,"\dashv"{rotate=-90,anchor=center}]
            \ar[d,hookrightarrow,blue,squiggly]
        & \bd{Top}_{\ca{T}}^{\mathsf{op}}
            \ar[l,shift left=2,"\Om"{below}]
            \ar[d,hookrightarrow,red,dashed] \\
        \Frm
            \ar[r,shift left=2,"\pt"{above}]
            \ar[r,phantom,"\dashv"{rotate=-90,anchor=center}]
        & \bd{Top}^{\mathsf{op}}
            \ar[l,shift left=2,"\Om"{below}]
    \end{tikzcd}
\caption{Dashed arrows represent full subcategory inclusions, squiggly ones represent wide inclusions, and dotted ones represent equivalences. }
\end{figure}

 We will see that the $\bd{Raney}$, $\bd{Skula}$, and $\bd{MT}$ settings each provide an instance of this phenomenon. Motivated by Theorem \ref{t: sober and td as fiber initial and terminals} for topological spaces, we study the relation between fiber-initiality and sobriety, and fiber-terminality and the $T_D$-property in our general setting.

\begin{lemma}\label{l: if it is in TopT it is TD}
If $(\Op,\ca C,\Tc,\iota_{\ca C})$ is an SFC-category satisfying \ref{T2}, \ref{T1}, and \ref{T4}. Then, $\ptc[\Ct]\subseteq \bd{Top_D}$, and so:
\[
\mf{Fix}_{\ca T}(\Topt)\subseteq \bd{Top_D}.
\]
\end{lemma}

\begin{proof}
We use the characterization of $T_D$-spaces in Theorem \ref{t: sober and td as fiber minimal and maximals}. Let $T\in \Ct$ and set $\ptc T=X$. Let $i: Y
  \to X$ be a map of $T_0$-spaces such that $\Omega(i)$
  is a frame isomorphism, in particular $i$ is a subspace inclusion. We show that $i$ is a homeomorphism. By Proposition \ref{p: the zeta isomorphism and the subspace inclusion ptc to ptO}, $\Om$ is isomorphic to $\Op\Omc$, and so $\Op\Omc(i)$ is an isomorphism. This means that $(\Omc(Y),\Op\Omc(i)^{-1})$ is an object in $\Op^{-1}(\Om(X))$. By Lemma \ref{l: definition omegat}, $(\Omc(X),\id_{\Om (X)})$ is a terminal object in the fiber. Let $\tau:\Omega_{\mathcal C}(Y) \to \Omega_{\mathcal C}(X)$, then, be the universal map to it. By definition of morphisms in this category, $\Op(\tau)=\Op\Omega_{\mathcal C}(i)^{-1}$. By Lemma \ref{l: mutually inverse frame isos that lift}, $\Omc(i)$ and $\tau$ are mutual inverses. Thus, $\ptc\Omc i:\ptc \Omc Y\to \ptc \ \Omc X$ is a homeomorphism. As by our assumption $\mf{Fix}(\Topc)=\Topz$, this is, up to isomorphism, the map $i:Y\to X$. Then, $i$ is an isomorphism, as desired.
\end{proof}

\begin{theorem}\label{t: main initial terminal theorem}
Let $(\Op,\ca C,\Tc,\iota_{\ca C})$ be a total SFC-category. 
\begin{enumerate}
    \item \label{main it 3}$\mf{Fix}_{\ca I}(\Topi)=\bd{Sob}$. 
    \item \label{main it 4}$\mf{Fix}_{\ca T}(\Topt)\subseteq \bd{Top_D}$.
\end{enumerate}
\end{theorem}

\begin{proof}
  The first item follows by Proposition~\ref{t: the initial duality}

  For \eqref{main it 4}, we consider $\ptc(\ca F \Op C)$ for some $C\in \ca
  C$. By Lemma \ref{l: ptc of the initial of a spatial is sober},
  \[
    \xi_{\ca F\ca O C}:\ptc \ca F\ca O C\to \pt \Op C
  \]
  is a homeomorphism. Hence $\ptc \ca F\ca O C$ is sober. The inclusion
  $\mf{Fix}_{\ca T}(\Topt)\subseteq\bd{Top_D}$ is the content of Lemma \ref{l:
  if it is in TopT it is TD}.
\end{proof}

Property \ref{T3} is not necessary for connecting sobriety and
fiber-initiality, as the next results show.

\begin{lemma}\label{l: if it is sober it is in Topi}
  Let $(\Op,\ca C,\Tc,\iota_{\ca C})$ be an SFC-category satisfying \ref{T2}
  and where $\Tc$ is fiber-initial. If $X$ is a sober space, and $I$ is
  initial in the fiber $\Op^{-1}(\Om (X))$, then $X$ is homeomorphic to
  $\ptc(I)$ and to $\pt \Op I$.
\end{lemma}
\begin{proof}
  This follows by the proof of Lemma~\ref{l: fibers of spatial frames whose
  spec is in topc}, since \( X \cong \pt \Om (X) \).
\end{proof}

\begin{lemma}\label{l: ptc of the initial of a spatial is sober-2}
  Suppose that $(\Op,\ca C,\Tc,\iota_{\ca C})$ is an SFC-category, and let \(
  L \) be a spatial frame such that \( \pt L \in \Topc \). If $(I,\theta)$ is
  initial in $\Op^{-1}(L)$, then \( I \) is sober.
\end{lemma}
\begin{proof}
  Immediate consequence of Lemma~\ref{l: sober interior well defined}
\end{proof}

\begin{theorem}
    If $(\Op,\ca C,\Tc,\iota_{\ca C})$ is an SFC-category satisfying \ref{T2} and \ref{T1}, $\ptc[\Ci]\subseteq \bd{Sob}$. If for every spatial frame $\Om X$ the fiber of $\Om X$ has an initial object, $\ptc[\Ci]=\bd{Sob}$.
\end{theorem}
\begin{proof}
    By Lemmas \ref{l: if it is sober it is in Topi} and \ref{l: ptc of the
    initial of a spatial is sober-2}.
\end{proof}

\subsection{Concrete examples}\label{ss concrete examples 2}

In the coming subsection, when the rest of the data is clear from the context we will sometimes identify morphisms $(H,\alpha,i)$ in $\bd{SFC}$ with the functor $H$, a similarly we will identify an SFC-category $(\Op,\ca C,\Tc,\iota_{\ca C})$ with $\ca C$. We show that the SFC-categories of Raney extensions, Skula extensions, and MT-algebras provide instances of several of our results for the natural adjunctions of fiber-initial and fiber-terminal objects. In particular, we will see instances of diagram \eqref{d: duality with frmt} for Raney extensions, Skula extensions, and MT-algebras.

For the functor $\Or:\bd{Raney}\to \Frm$, fibers have both initial and terminal objects, which we can characterize explicitly. The following result is essentially Theorem 3.9 in \cite{suarez25raney}; it is now a direct consequence of our definition of objects and morphisms in $\bd{Raney}$.

\begin{lemma}\label{l: exact filters are terminal and strongly exact initial}
A Raney extension $(L,\ca F)$ is fiber-initial if and only if $\ca F=\fse(L)$, and it is fiber-terminal if and only if $\ca F=\fe(L)$.
\end{lemma}

\begin{lemma}\label{l: topt is topd for raney}
    For the SFC-category $(\Op_{\ca R},\bd{Raney},\bd 2_{\ca R},\id_{\bd{2}})$, $\Topt=\bd{Top_D}$.

\end{lemma}
\begin{proof}
    By Proposition \ref{p: spaces in topt}, $X\in \Topt$ if and only if $\Om_{\ca R}(X)$ is fiber-terminal. By Lemma \ref{l: exact filters are terminal and strongly exact initial}, the fiber-terminal objects are those of the form $\fe(L)$ for some frame $L$. The result follows by Theorem \ref{t: many chactacterizations of td spaces}.
\end{proof}

\begin{lemma}\label{l: raney is total}
  The SFC-category $(\Op_{\ca R},\bd{Raney},\bd 2_{\ca R},\id_{\bd{2}})$ is total.
\end{lemma}
\begin{proof}
Property \ref{T2} follows from Theorem \ref{t: duality for raney extensions}. The fiber $\bd 2$ only contains $\bd 2_{\ca R}$, up to isomorphism, and similarly for $\Om \mathbb{S}$, so \ref{T1} is satisfied. Property \ref{T3} follows by $\fse$ being left adjoint to $\Op$; this is immediately seen to be stable as $\Op_{\ca R}(L,\fse(L))=L$. For item \ref{T4}, we recall that spectra of fiber-terminal objects are $T_D$ spaces, by Proposition \ref{p: td points of a frame}, and so they are in $\Topt$ by Lemma \ref{l: topt is topd for raney}.
\end{proof}

In the sequel, $\pt_{\ca E}:\Frm_{\ca E}\to \bd{Top_D}$ is the functor $\ptt:\Frm_{\ca T}\to \Topp_{\ca T}$.
\begin{theorem}\label{t: terminal duality for raney}
There are morphisms in $\SFC$:
\[
\begin{tikzcd}[row sep=large,column sep=large]
    \bd{Raney}_{\ca T}
    \ar[rr,"I_{\ca T}",hookrightarrow]
    && \bd{Raney}
    \ar[rr,"\fse\Op_{\ca R}",twoheadrightarrow]
    &&  \bd{Raney}_{\ca I}.
\end{tikzcd}
\]
Applying the $\mf{Dual}$ functor yields the following diagram in $\bd{RADJ}$, where $\mf{Dual}(\Op_{\ca R}I_{\ca T})$ and $\mf{Dual}(I_{\ca T})$ are also in $\bd{SADJ}$.
\begin{equation*}
    \begin{tikzcd}[row sep=huge, column sep=large]
        \bd{Raney}_{\ca I}
            \ar[r,shift left=2,"\pt_{\ca R}"{above}]
            \ar[r,phantom,"\dashv"{rotate=-90,anchor=center}]
        & \bd{Top}^{\mathsf{op}}          
            \ar[l,shift left=2,"\fse\Op_{\ca R}\ca{U}"{below}]
         \\
        \bd{Raney}
            \ar[u,"\fse \Op_{\ca R}",twoheadrightarrow]
            \ar[r,shift left=2,"\pt_{\ca R}"{above}]
            \ar[r,phantom,"\dashv"{rotate=-90,anchor=center}]
        & \bd{Top}^{\mathsf{op}}
            \ar[l,shift left=2,"\ca{U}"{below}] 
            \ar[u,equal]\\
        \bd{Raney}_{\ca T}
            \ar[u,hookrightarrow,red,dashed,"I_{\ca T}"]
            \ar[r,shift left=2,"\pt_{\ca R}"{above}]
            \ar[r,phantom,"\dashv"{rotate=-90,anchor=center}]
            \ar[d,"\ca O_{\ca R} I_{\ca T}",dotted,no head,swap]
        & \bd{Top_D}^{\mathsf{op}}
            \ar[u,hookrightarrow,red,dashed]
            \ar[l,shift left=2,"\ca{U}"{below}]
             \ar[d,equal] \\
        \bd{Frm}_{\ca{E}}
            \ar[r,shift left=2,"\pt_{\ca{E}}"{above}]
            \ar[r,phantom,"\dashv"{rotate=-90,anchor=center}]
            \ar[d,hookrightarrow,blue,squiggly]
        & \bd{Top_D}^{\mathsf{op}}
            \ar[l,shift left=2,"\Om"{below}]
            \ar[d,hookrightarrow,red,dashed] \\
        \Frm
            \ar[r,shift left=2,"\pt"{above}]
            \ar[r,phantom,"\dashv"{rotate=-90,anchor=center}]
        & \bd{Top}^{\mathsf{op}}
            \ar[l,shift left=2,"\Om"{below}]
    \end{tikzcd}
\end{equation*}
\end{theorem}
\begin{proof}
  By Lemma \ref{l: raney is total}, the SFC-category is total, so Theorems \ref{t: the initial duality} and \ref{t: the terminal duality}, as well as Proposition \ref{p: adding frmt to the picture}, apply. By Lemma \ref{l: topt is topd for raney}, $\Topt=\bd{Top_D}$, and $\Frm_{\ca T}=\Frm_{\ca E}$ by our definition of exact frame map.
\end{proof}

We obtain a pointfree version of Theorem \ref{t: sober and td as fiber initial and terminals} as a corollary.

\begin{corollary}\label{c: sober and td as fiber minimal and maximals FOR RANEY}
For the FC-category $(\Op_{\ca R},\bd{Raney})$:
\begin{itemize}
    \item A space $X$ is sober if and only if it is the spectrum of some fiber-initial object.
    \item A space $X$ is $T_D$ if and only if it is the spectrum of some fiber-terminal object.
\end{itemize}
\end{corollary}
\begin{proof}
The results follows from the duality for fiber-initial and fiber-terminal objects in \ref{t: terminal duality for raney}, and from having $\mf{Fix}_{\ca I}(\Topp)=\bd{Sob}$ and $\mf{Fix}(\Topt)=\Topt$ by \ref{t: main initial terminal theorem} and $\Topt=\bd{Top_D}$ (Lemma \ref{l: topt is topd for raney}).
\end{proof}

The following is Corollary 3.8 in \cite{manuell15}.

\begin{lemma}\label{l: fiber initials and terminals for biframes}
 A Skula extension $(L,\ca D)$ is fiber-initial if and only if $\ca D=\ca S L$. It is fiber-terminal if and only if $\ca D=\ca F L$.
\end{lemma}

\begin{lemma}\label{l: topt is topd for skula}
    For the SFC-category $(\Op_{\ca S},\bd{Skula},\bd 2_{\ca S},\id_{\bd{2}})$, $\Topt=\bd{Top_D}$.

\end{lemma}
\begin{proof}
    By Proposition \ref{p: spaces in topt}, $X\in \Topt$ if and only if $\ca{SK}(X)$ is fiber-terminal. By Lemma \ref{l: fiber initials and terminals for biframes}, the fiber-terminal objects are those of the form $(L,\ca F L)$ for some frame $L$. The result follows by Theorem \ref{t: many chactacterizations of td spaces}.
\end{proof}

\begin{lemma}\label{l: skula is total}
    The SFC-category $(\Op_{\ca S},\bd{Skula},\bd 2_{\ca S},\id_{\bd{2}})$ is total.
\end{lemma}
\begin{proof}
Property \ref{T2} follows from Theorem \ref{t: duality for skula extensions}. The fiber $\bd 2$ only contains $\bd 2_{\ca S}$, up to isomorphism, and similarly for $\Om \mathbb{S}$ so \ref{T1} is satisfied. Property \ref{T3} follows by existence of the left adjoint $\ca C:\bd{Frm}\to \bd{Skula}$ (Lemma \ref{l: congruence frame gives left adjoint}), and its stability, which we observe is ensured by $\Op_{\ca S}(L,\ca C L)=L$. For item \ref{T4}, we recall that spectra of fiber-terminal objects are $T_D$ spaces, by Proposition \ref{p: td points of a frame}, and so they are in $\Topt$ by Lemma \ref{l: topt is topd for skula}.
\end{proof}

\begin{theorem}\label{t: terminal duality for skula}
There are morphisms in $\SFC$:
\[
\begin{tikzcd}[row sep=large,column sep=large]
    \bd{Skula}_{\ca T}
    \ar[rr,"I_{\ca T}",hookrightarrow]
    && \bd{Skula}
    \ar[rr,"\ca S \Op_{\ca S}",twoheadrightarrow]
    &&  \bd{Skula}_{\ca I}.
\end{tikzcd}
\]
Applying the $\mf{Dual}$ functor yields the following diagram in $\bd{RADJ}$, where $\mf{Dual}(\Op_{\ca S}I_{\ca T})$ and $\mf{Dual}(I_{\ca T})$ are also in $\bd{SADJ}$.
\begin{equation*}
    \begin{tikzcd}[row sep=huge, column sep=large]
        \bd{Skula}_{\ca I}
            \ar[r,shift left=2,"\pt_{\ca S}"{above}]
            \ar[r,phantom,"\dashv"{rotate=-90,anchor=center}]
        & \bd{Top}^{\mathsf{op}}          
            \ar[l,shift left=2,"\ca S \Op_{\ca S}\ca{SK}"{below}]
         \\
        \bd{Skula}
            \ar[u,"\ca S\Op_{\ca S}",twoheadrightarrow]
            \ar[r,shift left=2,"\pt_{\ca S}"{above}]
            \ar[r,phantom,"\dashv"{rotate=-90,anchor=center}]
        & \bd{Top}^{\mathsf{op}}
            \ar[l,shift left=2,"\ca{SK}"{below}] 
            \ar[u,equal]\\
        \bd{Skula}_{\ca T}
            \ar[u,hookrightarrow,red,dashed,"I_{\ca T}"]
            \ar[r,shift left=2,"\pt_{\ca S}"{above}]
            \ar[r,phantom,"\dashv"{rotate=-90,anchor=center}]
            \ar[d,"\ca O_{\ca S} I_{\ca T}",dotted,no head,swap]
        & \bd{Top_D}^{\mathsf{op}}
            \ar[u,hookrightarrow,red,dashed]
            \ar[l,shift left=2,"\ca{SK}"{below}]
             \ar[d,equal] \\
        \bd{Frm}_{\ca{LE}}
            \ar[r,shift left=2,"\pt_{\ca{LE}}"{above}]
            \ar[r,phantom,"\dashv"{rotate=-90,anchor=center}]
            \ar[d,hookrightarrow,blue,squiggly]
        & \bd{Top_D}^{\mathsf{op}}
            \ar[l,shift left=2,"\Om"{below}]
            \ar[d,hookrightarrow,red,dashed] \\
        \Frm
            \ar[r,shift left=2,"\pt"{above}]
            \ar[r,phantom,"\dashv"{rotate=-90,anchor=center}]
        & \bd{Top}^{\mathsf{op}}
            \ar[l,shift left=2,"\Om"{below}]
    \end{tikzcd}
\end{equation*}
\end{theorem}

\begin{proof}
 By Lemma \ref{l: skula is total}, the SFC-category is total, so Theorems \ref{t: the initial duality} and \ref{t: the terminal duality}, as well as Proposition \ref{p: adding frmt to the picture}, apply. By Lemma \ref{l: topt is topd for skula}, $\Topt=\bd{Top_D}$, and $\Frm_{\ca T}=\Frm_{\ca{LE}}$ by our definition of locally exact frame map. 
\end{proof}

We have then obtained a new natural duality for $T_D$-spaces.

\begin{corollary}
    There is a natural contravariant adjunction 
    \[
    \pt_{\ca{LE}}:\bd{Frm}_{\ca{LE}}\leftrightarrows \bd{Top_D}:\Om
    \]
    for which the fixpoints are all $T_D$-spaces. 
\end{corollary}

We will now characterize the spectra of fiber-initial and fiber-terminal objects for $(\Op_{\ca S},\bd{Skula})$. 

\begin{corollary}\label{c: sober and td as fiber minimal and maximals FOR SKULA}
    For the FC-category $(\Op_{\ca S},\bd{Skula})$:
\begin{itemize}
    \item A space is sober if and only if it is the spectrum of some fiber-initial object.
    \item A space is $T_D$ if and only if it is the spectrum of some fiber-terminal object.
\end{itemize}
\end{corollary}
\begin{proof}
Follows from the dualities in \ref{t: terminal duality for raney}, and from having $\mf{Fix}_{\ca I}(\Topp)=\bd{Sob}$ and $\mf{Fix}(\Topt)=\Topt$ (Theorem \ref{t: main initial terminal theorem}), and $\Topt=\bd{Top_D}$ (Lemma \ref{l: topt is topd for skula}).
\end{proof}

Fiber-initial and fiber-terminal objects for MT-algebras have not been explicitly studied. We now give some preliminary results. The initial object of a fiber $\Op^{-1}_{\ca M}(L)$ does not exist in general: this would amount to a free complete Boolean algebra on $L$. We now study terminal objects.

\begin{lemma}\label{l: fiber terminal implies funayama}
    If $M$ is an MT-algebra which is terminal in $\Op_{\ca M}^{-1}(\Op_{\ca M}M)$, then it is isomorphic to $\ca F\Op_{\ca M}M$.
\end{lemma}
\begin{proof}
Suppose that $M$ is a fiber-terminal MT-algebra. The MT-algebra $(\Op_{\ca M}M,\ca F\Op_{\ca M}M)$ is in the same fiber, and so by our assumption there exists a map of MT-algebras $\tau:\ca F\Op_{\ca M}M\to M$, which restricts to an isomorphism between the frames of opens. In particular, as $M$ is $T_0$, it is a surjection. As $\tau$ is injective when restricted to the opens, by essentiality of the Funayama embedding, it must be injective on all of $\ca F\Op_{\ca M}M$. Then, $\tau$ is a morphism of complete Boolean algebras which is also a bijection. Hence, it is an isomorphism.
\end{proof}

The converse of Lemma \ref{l: fiber terminal implies funayama} above does not hold in general.
\begin{example}\label{e: a funayama which is not terminal}
The complete Boolean algebra $\ca F L$ is atomic if and only if $L$ is $T_D$-spatial (Theorem 3.4 in \cite{arrieta20}). Consider any spatial frame $\Om X$ which is not $T_D$-spatial. If its fiber had a terminal object, this would be isomorphic to $\ca F \Om X$, by Lemma \ref{l: fiber terminal implies funayama}. Then, the identity on $\Om X$ would extend to a morphism of complete Boolean algebras $\overline{\id_{\Om X}}:\ca P X\to \ca F \Om X$, which would then be surjective. But since images of complete maps of atomic complete Boolean algebras are atomic, and $\ca F \Om X$ is not, we would reach a contradiction. Then, $\Op^{-1}_{\ca M}(\Om X)$ has no terminal object.
\end{example}
We proceed to characterize $T_D$-spaces as the spectra of fiber-terminal MT-algebras. For an MT-algebra $M$, we define $\pt^D_{\ca M}M$ as the collection of all \emph{locally closed} atoms, that is, those of the form $a\wedge \neg b$ for $a,b\in \Op_{\ca M}M$. The following is Theorem 5.5 of \cite{bezhanishvili25td}.
\begin{lemma}\label{l: ptdOM is the spectrum of td atoms}
    There is a homeomorphism $\vartheta: \pt^D_{\ca M}M\cong \pt_D\Op_{\ca M}M$, given by $x\mapsto \up x\cap \Op_{\ca M}(M)$.
\end{lemma}
In \cite{bezhanishvili25td}, a map $\chi_M:M\to \ca P\pt^D_{\ca M}M$ is defined as $\chi_M(x)=\{y\in \pt^D_{\ca M}(M)\mid y\leq x \}$, and this is shown to be a map in $\bd{MT}$. The Boolean algebra $\ca P\pt^D_{\ca M}M$ is then regarded as an MT-algebra, with opens given by the direct image of $\Op_{\ca M}M$ under $\chi_M$.
\begin{lemma}\label{l: if X td then PX mt terminal}
If $X$ is a $T_D$ space, $\ca P X$ is terminal in the fiber of $\Om X$. 
\end{lemma}
\begin{proof}
Let $X$ be a $T_D$-space, and suppose that $(\Op_{\ca M}M,M)$ is in the fiber of $\Om X$. Consider the MT-algebra $\ca P\pt^D_{\ca M}M$. By Lemma \ref{l: ptdOM is the spectrum of td atoms}, this is isomorphic to the MT-algebra $\ca P\pt_D\Om X$, using the hypothesis that $M$ is in the fiber of $\Om X$. In turn, this is isomorphic to $\ca P X$, as $X$ is a $T_D$-space. Hence, for our claim it suffices to find a morphism from $M$ to $\ca P\pt^D_{\ca M}M$ in the fiber of $\Om X$. Consider the map $\chi_M:M\to \ca P\pt^D_{\ca M}M$. By its definition, $\Op_{\ca M}(\chi_M)$ is surjective, and by $T_D$-spatiality of $\Op_{\ca M}(M)$, as well as the isomorphism in Lemma \ref{l: ptdOM is the spectrum of td atoms}, it is injective. This means that it is a frame isomorphism, and so $\chi_M:M\to \ca P\pt^D_{\ca M}M$ is in the fiber $\Op_{\ca M}^{-1}(\Om X)$; hence a map as required.
\end{proof}

\begin{proposition}\label{p: td are the spectra of mt terminals}
For the FC-category $(\Op_{\ca M},\bd{MT})$, the spectra of fiber-terminal objects coincide with the $T_D$-spaces.
\end{proposition}
\begin{proof}
If $(L,M)\in \bd{MT}$ is fiber-terminal, then by Lemma \ref{l: fiber terminal implies funayama} it is isomorphic to $(L,\ca F L)$ for some frame $L$. By Lemma \ref{l: ptdOM is the spectrum of td atoms}, $\pt_{\ca M}(L,M)$ is then homeomorphic to $\pt_D\Op_{\ca M}M$. Conversely, if $X$ is a $T_D$-space, then $(\Om X,\ca P X)$ is fiber-terminal, by Lemma \ref{l: if X td then PX mt terminal}. By duality, $X$ is homeomorphic to its spectrum, hence the spectrum of a fiber-terminal object. 
\end{proof}

\begin{lemma}\label{l: mt algebras are B1 B2 and B4}
    The SFC-category $(\Op_{\ca M},\bd{MT},\bd 2_{\ca M},\id_{\bd{2}})$ satisfies \ref{T2}, \ref{T1}, and \ref{T4}. Therefore, 
    \[
    (\Op_{\ca M}I_{\ca T},\bd{MT}_{\ca T},\bd{2}_{\ca M},\id_{\bd{2}})
    \]
    is an SFC-category.
\end{lemma}
\begin{proof}
    The fibers of $\bd 2$ and of $\ca P \mathbb{S}$ only contains one object, up to isomorphism, and so \ref{T1} holds. Property \ref{T2} holds by Theorem \ref{l: mt algebras are a quasi t0 duality}. If $M\in \bd{MT}$ is fiber-terminal, by Proposition \ref{p: td are the spectra of mt terminals} its spectrum $\pt_{\ca M}(M)$ is a $T_D$-space. By Proposition \ref{l: if X td then PX mt terminal}, then, $\ca{P}(\pt_{\ca M}(M))$ is fiber-terminal, and so the space is in $\Topt$ by the characterization in Proposition \ref{p: spaces in topt}. Then, \ref{T4} holds.
\end{proof}

\begin{lemma}\label{l: topt and frmt for mt}
       For the SFC-category $(\Op_{\ca M},\bd{MT},\bd 2_{\ca M},\id_{\bd{2}})$:
    
    \begin{enumerate}
        \item $\Topt=\bd{Top_D}$;
        \item $\Frm_{\ca T}\subseteq \Frm_{\ca{LE}}$.
    \end{enumerate}
\end{lemma}
\begin{proof}
 We use the characterization in Proposition \ref{p: spaces in topt}, which states that spaces $X$ in $\Topt$ are precisely those such that $\ca P X$ is fiber-terminal. If $X$ is a $T_D$ space, $\ca P X$ is fiber-terminal by Lemma \ref{l: if X td then PX mt terminal}. If $X$ is a space such that $\ca P X$ is fiber-terminal, then by Lemma \ref{l: fiber terminal implies funayama}, $\ca P X$ must be, up to isomorphism, the Funayama envelope of $\Om X$, and so by Lemma \ref{l: fiber terminal implies funayama} it is fiber-terminal.
 Let us show the second item. By Lemma \ref{l: fiber terminal implies funayama}, terminal MT-algebras are of the form $(L,\ca F L)$ for some frame $L$. If a frame map has a lift in $\bd{MT}_{\ca T}$, then, it lifts to the Funayama construction, and so it is locally exact by our definition. 
\end{proof}

\begin{theorem}\label{t: terminal duality for mt}
There is a morphism of SFC-categories:
\[
\begin{tikzcd}[row sep=large,column sep=large]
    \bd{MT}_{\ca T}
    \ar[rr,"I_{\ca T}",hookrightarrow]
    && \bd{MT}.
\end{tikzcd}
\]
The $\mf{Dual}$ functor sends them to maps in $\bd{SADJ}$ as follows.

\begin{equation*}
    \begin{tikzcd}[row sep=huge, column sep=large]
        \bd{MT}
            \ar[r,shift left=2,"\pt_{\ca M}"{above}]
            \ar[r,phantom,"\dashv"{rotate=-90,anchor=center}]
        & \bd{Top}^{\mathsf{op}}
            \ar[l,shift left=2,"\ca{P}"{below}] \\
        \bd{MT}_{\ca T}
            \ar[u,hookrightarrow,red,dashed]
            \ar[r,shift left=2,"\pt_{\ca M}"{above}]
            \ar[r,phantom,"\dashv"{rotate=-90,anchor=center}]
            \ar[d,"\ca O_{\ca M} I_{\ca T}",dotted,no head]
        & \bd{Top_D}^{\mathsf{op}}
            \ar[u,hookrightarrow,red,dashed]
            \ar[l,shift left=2,"\ca{P}"{below}]
             \ar[d,equal] \\
\bd{Frm}_{\ca T}
    \ar[d,hookrightarrow]
    \ar[r,shift left=2,"\ptt"{above}]
    \ar[r,phantom,"\dashv"{rotate=-90,anchor=center}]
& \Topt^{\mathsf{op}}
 \ar[l,shift left=2,"\Om"{below}]
 \ar[d,equal]
             \\
        \bd{Frm}_{\ca{LE}}
            \ar[r,shift left=2,"\pt_{\ca{LE}}"{above}]
            \ar[r,phantom,"\dashv"{rotate=-90,anchor=center}]
        & \bd{Top_D}^{\mathsf{op}}
            \ar[l,shift left=2,"\Om"{below}]
    \end{tikzcd}
\end{equation*}
\end{theorem}

\begin{proof}
We note that in the proof of Theorem \ref{t: the terminal duality}, only
  properties \ref{T2}, \ref{T1}, and \ref{T4} are used, as well as fullness of
  the functor $\Op I_{\ca T}:\bd{MT}_{\ca T}\to \bd{Frm}_{\ca T}$. By
  Lemma \ref{l: mt algebras are B1 B2 and B4}, the three properties
  \ref{T2}, \ref{T1}, and \ref{T4} are satisfied. Then, the first two squares of the
  diagram commute as in Theorem \ref{t: the terminal duality}. Additionally,
  by Lemma \ref{l: topt and frmt for mt}, $\Topt=\bd{Top_D}$; while
  $\bd{Top_D}=\mf{Fix}_{\ca T}(\Topt)$ follows by the observation that, once
  again, in Theorem \ref{t: main initial terminal theorem} the item on
  fiber-terminals does not use property \ref{T3}. By Lemma \ref{l: topt and
  frmt for mt}, there is an inclusion $\Frm_{\ca T}\subseteq \Frm_{\ca{LE}}$. Let us show that the natural transformation in the claim is
  a homeomorphism. It is easy to see that the two vertical morphisms in the
  third square are a right map of adjunctions. By Proposition \ref{p: the zeta
  generalized}, then, for the desired result it suffices to observe that
  $\Frm_{\ca T}\subseteq \Frm_{\ca{LE}}$ is full.  
\end{proof}

\begin{remark}
There are two more approaches in the literature for obtaining a duality for $T_D$-spaces as a result of a dual adjunction with subcategories of $\bd{MT}$. In \cite{bezhanishvili25td}, this is done by restricting the morphisms in $\bd{MT}$ (to those such that $\Op_{\ca M} f$ is a D-morphism). Another such duality is the content of Theorem 6.4 in \cite{bezhanishvili23}, where objects of $\bd{MT}$ are restricted to those of the form $(L,\ca F L)$ for some frame $L$. By Lemma \ref{l: fiber terminal implies funayama}, all fiber-terminal objects are of this form, so our dual adjunction is a restriction of that in \cite{bezhanishvili23}; it is also a proper restriction as there are objects of the form $(L,\ca F L)$ that are not fiber-terminal (see Example \ref{e: a funayama which is not terminal}).    
\end{remark}

\printbibliography 

\end{document}